\documentclass[11pt,letterpaper,reqno]{amsart}
\usepackage{fullpage}
\usepackage{amsmath,amsthm,amssymb,amscd}
\usepackage{enumerate}
\usepackage{enumitem}
\usepackage{array}
\usepackage{float}
\usepackage{bbm}
\usepackage{bm}
\usepackage{stmaryrd}
\usepackage{comment}
\usepackage{mathtools}
\usepackage{dsfont}
\usepackage{xcolor}

\usepackage{hyperref}
\hypersetup{
    colorlinks=true,
    linkcolor=blue,
    citecolor=blue,
    urlcolor=blue
}

\newtheorem{theorem}{Theorem}[section]

\newtheorem{corollary}[theorem]{Corollary}
\newtheorem{lemma}[theorem]{Lemma}

\newtheorem{proposition}[theorem]{Proposition}
\newtheorem{conjecture}[theorem]{Conjecture}

\theoremstyle{definition}
\newtheorem{definition}[theorem]{Definition}

\newtheorem{remark}[theorem]{Remark}

\mathtoolsset{showonlyrefs=true}
\numberwithin{equation}{section}
\numberwithin{figure}{section}
\numberwithin{table}{section}

\allowdisplaybreaks

\let\subsectiontemp\subsection
\renewcommand{\subsection}[1]{ 
    \subsectiontemp{#1} \hfill\vspace{0.5\linespacing} 
}

\newcommand{\lrp}[1]{\left(#1\right)}
\newcommand{\lrb}[1]{\left[#1\right]}
\newcommand{\lrcb}[1]{\left\{#1\right\}}

\newcommand{\QQ}{\mathbb{Q}}

\newcommand{\nw}{\mathrm{new}}
\newcommand{\eigen}{\operatorname{eigen}}

\title{Effective Hecke eigenvalue equidistribution over the Atkin--Lehner subspaces}

\author[A. J. Kumar]{Aarya J. Kumar}
\address[A. J. Kumar]{Princeton University}
\email{ajkumar@princeton.edu}

\author[S. Mondal]{Sargam Mondal}
\address[S. Mondal]{Massachusetts Institute of Technology}
\email{sargam@mit.edu}

\author[E. Ross]{Erick Ross}
\address[E. Ross]{School of Mathematical and Statistical Sciences, Clemson University, Clemson, SC}

\author[H. Xue]{Hui Xue}
\address[H. Xue]{School of Mathematical and Statistical Sciences, Clemson University, Clemson, SC}
\email{huixue@clemson.edu}


\begin{document}

\begin{abstract}
    For a fixed prime $p$, let $\mu_p$ denote the $p$-adic Plancherel measure. Then the first main goal of this paper is to prove effective (and moreover explicit) $\mu_p$-equidistribution of the $p$-th Hecke eigenvalues over the Atkin--Lehner subspaces $S_k^\sigma(N) \subseteq S_k(N)$ and $S_k^{\operatorname{new}, \sigma}(N) \subseteq S_k^{\operatorname{new}}(N)$. We then highlight five applications of this explicit equidistribution result. For the first application, we generalize Kim's vertical analog of the Atkin--Serre conjecture to the Atkin--Lehner setting. For the second application, we obtain explicit bounds on the number of newforms $f \in S_k^{\operatorname{new}, \sigma}(N)$ for which $p$ is extremal over $S_k^{\operatorname{new}, \sigma}(N)$. For the third application, we prove explicit asymptotics for the number of $\mathbb{F}_{p^r}$-points on the modular Jacobian $J_0(N),$ as well as on its factors $J_0^{\operatorname{new}}(N),$ $J_0^\sigma(N),$ and $J_0^{\operatorname{new}, \sigma}(N)$. We also make explicit an asymptotic result of Serre concerning point counts of the modular curves $X_0(N)$. For the fourth application, we generalize lower bounds due to Murty and Sinha on the sizes of large $\mathbb{Q}$-simple factors of $J_0(N)$ to analogous bounds for $J_0^\sigma(N)$. Finally, for the fifth application (the details of which are given in a separate paper), we use our explicit equidistribution result to prove that only finitely many modular Jacobians are supersingular modulo any fixed prime.
\end{abstract}

\maketitle

\begin{center}
    \textit{In memory of our coauthor and friend, Erick Ross, who tragically passed away during the preparation of this work.}
\end{center}
\vspace{1ex}

\tableofcontents

\section{Introduction}\label{section:introduction}

\subsection{Background}

Throughout this paper, we let $S_k(N)$ denote the space of cuspidal modular forms of even weight $k \geq 2$ and congruence subgroup $\Gamma_0(N)$, $S_k^\nw(N)$ denote its new subspace, and $H_k(N)$ denote the basis of (normalized) newforms for $S_k^{\nw}(N)$. Given $f \in H_k (N)$ and a prime $p \nmid N$, it follows from the Ramanujan-Petersson conjecture, now a theorem in this setting due to Deligne's work in \cite{Deligne1973Formes} and \cite{Del74}, that
$|a_p(f)| \leq 2p^{(k - 1)/2},$
where $a_p(f)$ denotes the $p$-th Fourier coefficient of $f.$ Consequently, the normalized Fourier coefficient $\frac{a_p(f)}{p^{(k - 1)/2}}$ lies in the interval $[-2, 2]$. Recall that these normalized Fourier coefficients can equivalently be understood as the eigenvalues of the normalized Hecke operator $\mathbf{T}_p' := \frac{1}{p^{(k-1)/2}} \mathbf{T}_p$.

One important problem is to determine the precise distribution of these normalized Fourier coefficients $\frac{a_p(f)}{p^{(k - 1)/2}}$ for $f \in H_k (N)$. The key result addressing this problem is the Sato--Tate conjecture, which has been proven in the case of elliptic curves without complex multiplication in \cite{CHT2008}, \cite{HSBT2010}, and \cite{Taylor2008}, and in the general case in \cite{BLGHT2011}.
{
\renewcommand{\thetheorem}{(Sato--Tate conjecture)}
\begin{theorem}
    Fix a non-CM $f \in H_k (N)$. Let $I$ be a sub-interval of $[-2, 2]$ and $\chi_I$ be its characteristic function. Then,
    $$\lim_{X \rightarrow \infty} \frac{\#\{ p \leq X : p \nmid N , \frac{a_p(f)}{p^{(k - 1)/2}}\in I\}}{\#\{ p \leq X : p \nmid N \}} = \int_{-2}^2 \chi_I (x)\, d\mu_{\infty}^{ST}(x),$$
    where the Sato--Tate measure is given by
    $$d\mu_{\infty}^{ST}(x) = \frac{1}{2\pi} \sqrt{4 - x^2} \,dx.$$
\end{theorem}
\addtocounter{theorem}{-1}
}
The primary focus of this paper is an analogous equidistribution result of Serre \cite{Ser97}, which is sometimes referred to as ``vertical Sato--Tate" \cite{Kim24}.  The ``vertical'' perspective refers to that in which one fixes a prime $p$ and studies the distribution of the eigenvalues of the $p$-th Hecke operator as the automorphic family varies (as opposed to the ``horizontal" perspective of classical Sato--Tate, where $f \in H_k(N)$ is fixed, and $p$ tends to infinity). In our setting of modular forms, this means that for fixed $p$, we study the distribution of $\mathrm{eigen}_{S_k(N)}(\mathbf{T}_p')$ as the weight $k$ or the level $N$ grows. Here, $\mathrm{eigen}_{S_k(N)}(\mathbf{T}_p')$ denotes the multiset of the eigenvalues of $\mathbf{T}_p'$ over $S_k(N).$ 

Serre's equidistribution result is the following.
{
\renewcommand{\thetheorem}{(Serre)}
\begin{theorem}
    Fix a prime $p$ and let $I$ be a sub-interval of $[-2, 2]$. Then along levels $N$ coprime to $p$,
    $$\frac{\# \operatorname{eigen}_{S_k(N)}(\mathbf{T}'_p) \cap I}{\dim S_k(N)} 
    \longrightarrow
    \int_{-2}^2 \chi_I(x)\, d\mu_p(x)
    \qquad
    \text{as }k + N \rightarrow \infty,
    $$
    where $d\mu_p(x)$ denotes the $p$-adic Plancherel measure 
    $$d\mu_p(x) := \frac{p + 1}{2\pi} \cdot \frac{\sqrt{4 - x^2}}{\left( p^{1/2} + p^{-1/2}  \right)^2 -  x^2} \,dx.$$
\end{theorem}
\addtocounter{theorem}{-1}
}

Serre's proof of this result in \cite{Ser97} also extends to analogous results over the spaces $S_k^\nw(N)$ and $S_k(N, \chi),$ and the third author of this paper has proved an analog for $S_k^\nw(N, \chi)$ in \cite{Ros26equidistribution}.

For many arithmetic applications, it is necessary to quantify the error term in Serre's equidistribution theorem. Effective equidistribution estimates provide explicit bounds for the discrepancy between the empirical distribution of Hecke eigenvalues and the limiting $p$-adic Plancherel measure. They can be applied to study exceptional Hecke eigenvalues, Hecke fields, and related arithmetic statistics. The first such result was established by Murty and Sinha \cite{MS07}, who proved the following effective form of Serre's equidistribution theorem.

{
\renewcommand{\thetheorem}{(Murty--Sinha: effective equidistribution)}
\begin{theorem}
Let $p$ be a prime, $I\subseteq [-2,2]$ be an interval, $N$ be coprime to $p$, and $k \ge 2$ be even. Then
\[
\left|
\frac{\# \operatorname{eigen}_{S_k(N)}(\mathbf{T}'_p) \cap I}
{\dim S_k(N)}
-
\int_{-2}^2 \chi_I(x)\, d\mu_p(x)
\right|
\ll
\frac{\log p}{\log kN},
\]
where the implied constant is effectively computable.
\end{theorem}
\addtocounter{theorem}{-1}
}

\begin{remark}
    Note that the displayed equation above is only meaningful when $\dim S_k(N) \ne 0$. Throughout the entire paper, we implicitly restrict any result to spaces of non-zero dimension in order to avoid division by zero.
\end{remark}

\subsection{The main theorem}\label{subsection:thm:main-theorem-statement}

In this paper, we will prove an analog of Murty-Sinha's results for the Atkin--Lehner sign pattern subspaces, defined below.

Define a \textit{sign pattern} for $N$ to be a multiplicative function $\sigma:\{Q:Q\parallel N\}\to\{\pm1\}$.
For each exact divisor $Q \| N$, let $W_{Q}$ denote the corresponding Atkin--Lehner involution acting on $S_k(N)$ (and $S_k^\nw(N)$). The involutions $\{W_{Q}\}_{Q \| N}$ commute with one another (see \cite[$\S$13.2.2]{cohen2017modular}) and have eigenvalues $\pm 1$.
Hence $S_k(N)$ can be decomposed into the common eigenspaces for the $W_Q$,
\[
S_k(N)=\bigoplus_\sigma S_k^\sigma(N) 
\quad \text{where} \quad
S_k^\sigma(N)
:=
\{f\in S_k(N): W_{Q}f=\sigma(Q)f
\text{ for all } Q\parallel N\},
\]
with the direct sum running over all sign patterns $\sigma$ for $N$ (see \cite[$\S$13.4]{cohen2017modular}). Here, $S_k^\sigma(N)$ is called the \textit{Atkin--Lehner sign pattern subspace} of $S_k(N)$ associated to the sign pattern $\sigma$.
In exactly the same way, the newspace $S_k^{\nw}(N)$ can similarly be decomposed as
\begin{align}
    S_k^\nw(N)=\bigoplus_\sigma S_k^{\nw,\sigma}(N).
\end{align}
Note that this Atkin--Lehner sign pattern decomposition can be understood as a refinement of the more familiar Fricke sign decompositions $S_k^\pm(N)$ and $S_k^{\nw,\pm}(N)$. Likewise, the set $H_k(N)$ of newforms decomposes as $H_k(N)=\bigsqcup_{\sigma} \, H_k^{\sigma}(N)$, where $H_k^{\sigma}(N)$ is the set of newforms in $S_k^{\nw,\sigma}(N)$.

We remark here one reason we study the distribution of Hecke operator eigenvalues over the Atkin--Lehner sign pattern subspaces in particular.
Atkin--Lehner sign pattern subspaces play an important role in the generalized Maeda conjecture philosophy. Conjecturally, for $100 \%$ of $N$, the characteristic polynomial of any Hecke operator $\mathbf{T}_p$ (for $p\nmid N$) acting on $S_k^{\mathrm{new},\sigma}(N)$ is irreducible over $\mathbb Q$; consequently, the newforms in $H_k$ form a single Galois orbit. More generally, it is conjectured (see \cite{Chow2015} and \cite{Kimball2021}) that the decomposition by Atkin--Lehner sign patterns is the only obstruction to the irreducibility of Hecke polynomials on $S_k^{\mathrm{new}}(N)$. For this reason, the sign pattern subspaces $S_k^{\nw,\sigma}(N)$ are, in some sense, the finest possible natural decomposition on which one can study the Hecke operators and their corresponding eigenvalues.

Furthermore, Atkin--Lehner signs admit interpretations in terms of abelian varieties of $\operatorname{GL}_2$-type. Let $f \in H_2^{\sigma}(N)$. If $f$ has integral Fourier coefficients (and hence corresponds to an elliptic curve $E = E_f$), then the Atkin--Lehner sign $\sigma(q^r)$ coincides with the local root number $\omega_q(E)$, which provides information relating to the reduction type of $E$ at $q\mid N$. This fact extends to higher-dimensional abelian varieties: by \cite{schmidt2002remarks,Car86}, the Atkin--Lehner sign $\sigma(q^r)$ is equal to the local root number $\omega_q(A_f)$. Hence, if $f \in H_2^{\sigma}(N)$, then the corresponding abelian variety $A_f$, of conductor $N^{[K_f:\mathbb{Q}]}$, has local root numbers satisfying $\omega_q(A_f)=\sigma(q^r)$ for every prime power $q^r \mid N$. Furthermore, the negative of the global root number is equal to the Fricke sign of $f$, i.e $-\omega(A_f)=\sigma(N),$ and recall from \cite{Rohrlich1996OnThe} that the parity conjecture, which is implied by the Birch--Swinnerton-Dyer conjecture, predicts that
$\omega(A_f)=(-1)^{\operatorname{rank}(A_f)}.$

Over the newspace $S_k^\nw(N)$, it turns out that sign patterns $\sigma$ satisfying $\sigma(4)=+1$ (which is only possible when $4 \| N$, see \cite[Theorem 7]{AtkinLehner1970Hecke}) cannot appear. However, it was recently shown that apart from this one obstruction arising for the newspace, all other sign patterns do appear \cite[Theorem 1.1]{RVWX26}. For this reason, we define the notion of \textit{admissible} sign patterns $\sigma$, where all sign patterns are admissible over $S_k(N)$, and where the admissible sign patterns over $S_k^\nw(N)$ are those not satisfying $\sigma(4)=+1$. For the remainder of the paper, $\sigma$ will always denote an admissible sign pattern.

In \cite[Theorem 1.3]{RVWX26}, it was shown that for fixed $p$, $\operatorname{eigen}_{S_k^\sigma(N)}(\mathbf{T}'_p)$ and $\operatorname{eigen}_{S_k^{\nw,\sigma}(N)}(\mathbf{T}'_p)$ become equidistributed over $[-2,2]$ with respect to $d\mu_p$ as $N+k\to\infty$ (for admissible sign patterns $\sigma$). However, this result was ineffective, with no control over the dependence on $p$.
As the main result of this paper, we prove effective equidistribution of $\operatorname{eigen}_{S_k^\sigma(N)}(\mathbf{T}'_p)$ uniformly over all primes $p$.
\begin{theorem} \label{thm:main-theorem}
    There exists an effectively computable constant $C_0$ such that 
    $$
    \left| \frac{\# \operatorname{eigen}_{S_k^{\sigma}(N)}(\mathbf{T}'_p) \cap I}{\dim S_k^{\sigma}(N)} - \int_{-2}^2 \chi_I(x) \, d \mu_p(x)  \right| \le C_0  \frac{\log p}{\log kN}
    $$
    for all primes $p$, levels $N$ coprime to $p$, admissible sign patterns $\sigma$, weights $k$, and intervals $I \subseteq [-2,2]$. Moreover, the same result also holds for $S_k^{\nw,\sigma}(N)$ (with a different constant $C_0^\nw$).
\end{theorem}

We compute explicit values for $C_0$ and $C_0^\nw$ in Section \ref{section:explicit-constants}, with some technical details deferred to Appendix \ref{section:lambert-lemmata}.  Theorem \ref{thm:limit-constant} shows that these constants can be taken to be $\frac{3\sqrt{2}}{2\pi} +\epsilon$ for sufficiently large $N$, say $N \ge N_\epsilon$ and $N \ge N^\nw_\epsilon$, respectively.
Proposition \ref{prop:explicit-constant-chart} gives explicit values for $C_0$ and $C_0^\nw$ in the case of general $N$, as well as explicit values of $N_\epsilon$ and $N^\nw_\epsilon$ for various choices of $\epsilon$.
Furthermore, Proposition \ref{prop:explicit-constant-chart-prime} gives improved explicit constants in the case where $N$ is prime. 

Additionally, by summing over all sign patterns, our result recovers Murty-Sinha's effective vertical equidistribution result \cite[Theorem 2]{MS07} and also Serre's original result.
In fact, Theorem \ref{thm:main-theorem} can  even be considered as a slight strengthening of \cite[Theorem 2]{MS07} in that we actually give explicit values of $C_0$ and $C_0^\nw$.
Similarly, summing over all sign patterns for the newspace yields a novel effective equidistribution result for the newspace $S_k^\nw(N)$.
Lastly, summing over sign patterns with fixed global sign yields an effective equidistribution result for the Fricke eigenspaces $S_k^{\mathrm{new},\pm}(N)$.
These consequences form the content of Corollary \ref{cor:main-theorem-for-larger-spaces}.

\subsection{Applications}

There are many different motivations to study effective equidistribution on Atkin--Lehner sign pattern subspaces.
We discuss five applications in particular of our main result. The first four applications  are covered in the four subsections of $\S$\ref{section:applications}; the fifth will be treated in the forthcoming work \cite{supersingularity}. Applications \# 1 -- 2 concern the distribution of normalized eigenvalues near extreme points of the interval $[-2, 2]$; i.e., $0$ and $\pm 2.$ Applications \#3 -- 5 concern abelian varieties of $\operatorname{GL}_2$-type: in the case of weight $2$ newforms $f$, the construction of Eichler \cite{Eic54} and Shimura \cite{Shi58} allows one to associate such an abelian variety $A_f$ to $f$.

The following results use the same constant $C_0$ that appeared in Theorem \ref{thm:main-theorem}. In particular, this means that just as for Theorem \ref{thm:main-theorem}, each of these constants can be replaced by $\frac{3\sqrt 2}{2\pi} + \epsilon$ whenever $N \geq N_\epsilon$ or $N \geq N^{\nw}_\epsilon$, respectively.

\vspace{4mm}
\textbf{Application \#1 ($\S$\ref{subsection:vertical-Atkin--Serre}): Vertical Atkin--Serre}

An especially important problem in arithmetic statistics is the Atkin--Serre conjecture, which states that the coefficients $a_p$ are asymptotically bounded away from $0$.

{
\renewcommand{\thetheorem}{(Atkin--Serre)}
\begin{conjecture}
    Fix a non-CM newform $f \in H_k (N)$ of weight $k \geq 4.$ Then
    $\left| a_f(p) \right| \gg_\epsilon p^{(k - 3)/2 - \epsilon}$.
\end{conjecture}
\addtocounter{theorem}{-1}
}

It has been observed, for example in \cite{GTW21} and \cite{Kim24}, that effective horizontal and vertical Sato--Tate theorems can be used to make progress on the Atkin--Serre conjecture and vertical analogs thereof. An effective form of Sato-Tate from \cite{Tho21} was used in \cite{GTW21} to prove that for any non-CM $f \in S_k^\nw(N)$, the set of primes that fail Atkin--Serre is of zero density. More recently the techniques used in those proofs have been adapted by Kim \cite{Kim24} to prove an analogous result that he called ``vertical Atkin--Serre".  We will adapt this result to the setting of the Atkin--Lehner subspaces $S_k^\sigma(N)$ and $S_k^{\nw, \sigma}(N)$ through the following proposition.

\begin{proposition}\label{prop:vertical-Atkin--Serre}
    Let $p$ be a prime coprime to $N$. Then,
    $$\frac{\#\left\{ \lambda \in \operatorname{eigen}_{S_k^\sigma(N)} \! \left( \mathbf{T}_{p}'  \right) : \left| \lambda \right| \leq \frac{\log p}{\log kN} \right\}}{\dim S_k^\sigma(N)} \leq \left(C_0 + \frac{3\sqrt{2}}{2\pi} \right) \cdot  \frac{\log p}{\log  kN}$$
    This statement also holds for $S_k^{\nw,\sigma}(N)$, with $C_0^\nw$ in place of $C_0.$
\end{proposition}

In particular, Proposition \ref{prop:vertical-Atkin--Serre} implies that as $k + N \rightarrow \infty$, a vanishing proportion of the eigenvalues $\operatorname{eigen}_{S_k^\sigma(N)} \! \left( \mathbf{T}_{p}'  \right)$ have magnitude less than $\frac{\log p}{\log kN}$.

\vspace{4mm}
\textbf{Application \#2 ($\S$\ref{subsection:extremal-primes}): Extremal Primes}

Recall the definition of an extremal prime.

\begin{definition}\label{def:extremal-prime}
    Let $f \in H_k (N)$ be of weight $k \geq 4.$ We say that $p$ is an \textit{extremal prime} for $f$ if
    $\left| a_f(p)  \right| \geq \left\lfloor 2 p^{(k - 1)/2} \right\rfloor.$
\end{definition}

Our results give bounds on how often a prime $p$ can be extremal. As with vertical Atkin--Serre, we have adapted results in \cite{Kim24} on extremal primes to the setting of Atkin--Lehner subspaces; the most relevant of these applications is the following.

\begin{proposition}[Sparsity of extremality] \label{prop:k-dependent-extremality}
    Let $p$ be a prime coprime to $N$ satisfying $\log p \geq \frac{4\log \log  kN}{3(k - 1)}.$ Then if $kN \geq e^8,$
    $$\frac{\#\left\{ \lambda \in \operatorname{eigen}_{ S_k^\sigma(N)} \! \left( \mathbf{T}_{p}  \right) :  |\lambda|  \geq \left\lfloor 2 p^{(k - 1)/2}  \right\rfloor \right\}}{\dim S_k^\sigma(N)} \leq 83.5853 \cdot \, C_0 \left( \frac{p + 1}{p - 1} \right)^2  \cdot \frac{ \log p \,  \log \left(  \log kN + 1000 \right)^2}{\log kN}.$$
    The same result also holds for $S_k^{\nw, \sigma}(N)$, with $C_0^\nw$ in place of $C_0.$
\end{proposition}

In particular, this means that the proportion of newforms in $S_k^{\nw,\sigma}(N)$ for which $p$ is extremal approaches $0$ as $k$ grows large for fixed $p$. We prove Proposition \ref{prop:k-dependent-extremality} in $\S$\ref{subsection:extremal-primes}, while some bounds used to calculate explicit constants are deferred to Appendix \ref{subsection:integral-bounds}. Also useful both for this result and more generally (e.g. for the Dirichlet divisor problem) is an effective version of the error $|V_M(x) - s(x)|$ between the Vaaler polynomial and the sawtooth function, which we prove in Appendix \ref{appendix:Vaaler}.

\vspace{4mm}
\textbf{Application \#3 ($\S$\ref{subsection:point-count}): Estimates for the point counts $\# J_0^{\nw, \sigma}(N)(\mathbb{F}_{p^n})$ }

For this application, we estimate the point counts of the newparts of the Atkin--Lehner isotypic components of the Jacobian factor $
J_0^{\nw,\sigma}(N)(\mathbb{F}_{p^n})$. Specifically, we prove (an effective version of) the asymptotic
\[
\frac{\log \#J_0^{\nw,\sigma}(N)(\mathbb F_{p^n})}
     {\dim S_2^{\nw, \sigma}(N)}
\longrightarrow
\begin{cases}
\displaystyle \log p^n - \frac{p-1}{2}\log\!\left(1-{p^{-2n}}\right),
& \text{$n$ odd},\\[10pt]
\displaystyle \log p^n - (p-1)\log\!\left(1-{p^{-n}}\right),
& \text{$n$ even},
\end{cases}
\qquad\text{as }N\to\infty
\]
in Proposition \ref{prop:asymptotic-point-count-Jacobian-q}. We also give an analogous result for $J_0^\nw(N)$. In Proposition \ref{prop:asymptotic-point-count-Jacobianreg-q}, we demonstrate that the analogous result holds for the full Jacobian, $J_0(N).$ Next, we make effective Serre's result \cite[Th\'eor\`eme 9]{Ser97} providing an asymptotic of the number of points $\#X_0(N)(\mathbb{F}_{p^n}),$ and we generalize it to quotients of $X_0$ by arbitrary subgroups of the Atkin--Lehner group. Finally, assuming the generalized Maeda conjecture (see Conjecture \ref{conj:gen-Maeda}), we give the asymptotic behavior of point counts $
\#A_f(\mathbb{F}_q)$ for $100\%$ of abelian varieties of $\operatorname{GL}_2$-type. 

Proposition \ref{prop:asymptotic-point-count-Jacobianreg-q} can also be viewed as a refinement of the well-known bound of Tsfasman and Vl{\u{a}}du{\c{t}} (see \cite[{Theorem G$'$}]{tsfasman2002infinite}, there stated in terms of function fields), which gives that for any family of curves $\{C_i\}_{i}$ over $\mathbb{F}_{p^n}$,
$$1 \leq \limsup_{i \rightarrow \infty} \frac{\log_{p^n} h(C_i)}{g(C_i)} \leq 1 - \left(\sqrt{p^n} - 1 \right) \log_{p^n} \left( \frac{p^n}{p^n - 1}  \right).$$
Here, the class number $h(C_i)$ is equal to the number of points $\#J(C_i).$ Proposition \ref{prop:asymptotic-point-count-Jacobianreg-q} is consistent with the well-known fact that Tsfasman and Vl{\u{a}}du{\c{t}}'s upper-bound is tight for the family $\{ X_0(N) \}_N$ when $n = 2.$ Moreover, since Proposition \ref{prop:asymptotic-point-count-Jacobianreg-q} provides a general asymptotic for any $n,$ it proves that Tsfasman and Vl{\u{a}}du{\c{t}}'s bound is not tight over the family $\{X_0(N) \}_N$ when $n \neq 2.$

\vspace{4mm}
\textbf{Application \#4 ($\S$\ref{subsection:large-hecke-fields}): Large Hecke Fields}
In this subsection, we apply our multiplicity bounds to study the degree of Hecke fields. At weight $k=2$, we show that the isotypic factors $J_0^\sigma(N)$ of the Jacobian have large $\QQ$-simple factors. Specifically, we prove the following.

\begin{proposition}\label{prop:lower-bound-dimension-q-simple-factor}
$J_0^{\sigma}(N)$ has a $\QQ$-simple factor of dimension 
$$d \geq \sqrt{\max \left\{ 0, \frac{\log \log (2 \sqrt{N} ) - \log (C_0^\nw \log 3)}{\log (4\sqrt{3} + 1)} \right\}}.$$
\end{proposition}
In particular, this proposition shows that $J_0^\sigma(N)$ can only be isogenous to a product of elliptic curves for finitely many $N$. See Corollary \ref{cor:J0sigma-isogenous-to-product-elliptic-curves} for explicit bounds.

This application was inspired by \cite[\S6.2]{Ser97} and \cite[\S16]{MS07}.

\vspace{4mm}
\textbf{Application \#5: Frobenius eigenvalues and supersingularity}

Let $\Lambda_N$ denote the multiset of $p$-Frobenius eigenvalues of $J_0^{\nw, \sigma}(N)$. This multiset can be written as
$\Lambda_N = \{\alpha_f, \overline{\alpha}_f \}_{f\in H_k^\sigma(N)}$, where $\alpha_f, \overline{\alpha_f}$ are the roots of $x^2 - a_p(f) x + p$ for $f \in H_k^\sigma(N)$ (i.e. the newforms $f$ in $S_k^\sigma(N)$). Via this correspondence, our equidistribution result for $\{a_f(p)\}_{f \in H_k^\sigma(N)}$ is equivalent to an equidistribution result for $\Lambda_N$. In the forthcoming work \cite{supersingularity}, we formalize this equidistribution result for $\Lambda_N$ and use it to show that $J_0^{\nw, \sigma}(N)$ (and consequently $J_0^{\nw}(N)$, $J_0(N)$ and $J_1(N)$) is supersingular modulo a given prime for only finitely many levels $N.$

\subsection{Notation}

Throughout this paper, we use the following notational conventions.

\begin{enumerate}
    \item We often let $r^{\sigma}$ denote $\dim S_k^\sigma(N)$ and $r^{\nw, \sigma}$ denote $\dim S_k^{\nw, \sigma}(N).$
    \item We let $T_n(x)$ and $U_n(x)$ for $x\in[-1,1]$ respectively be the $n$-th Chebyshev polynomials of the first and second kind, which satisfy
    $$T_n(\cos \theta) = \cos(n \theta) \qquad U_n(\cos \theta) = \frac{\sin((n + 1)\theta)}{\sin \theta}.$$
    Several computations involving $T_n(x)$ and $U_n(x)$ are carried out in Appendix \ref{section:chebyshev}.
    \item We use bold $\mathbf{T}_p$ to notate Hecke operators in order to avoid confusion with the Chebyshev polynomial $T_n$.
    \item We use $\mathds{1}_{\text{condition}}$ to represent the indicator function for whether ``condition" holds. 
    For a set or interval $I,$ we use 
    $$\chi_I(x) := \begin{cases}
        1 & x \in I\\
        0 & x \notin I
    \end{cases}$$
    to denote the indicator function for $I.$
    \item We use the convention that $0$ is not in $\mathbb{N}.$
\end{enumerate}

\section{Effective equidistribution of Hecke eigenvalues}
\label{section:proof-of-thm:main-theorem}

\subsection{The trace formula}

Let $\operatorname{Tr}_{S_k^{\sigma}(N)} \mathbf{T}_m'$ and $\operatorname{Tr}_{S_k^{\nw, \sigma}(N)} \mathbf{T}_m'$ denote the traces of $\mathbf{T}_m'$ over these respective spaces. We begin with a discussion of formulas for these traces. Note that at $m=1$, these traces are equal to the dimensions $r^\sigma$ and $r^{\nw, \sigma}.$

The result \cite[Corollary 3.2]{RVWX26} gave estimates for $\operatorname{Tr}_{S_k^\sigma(N)}\mathbf{T}_m'$ and $\operatorname{Tr}_{S_k^{\nw, \sigma}(N)}\mathbf{T}_m'$. We write down an effective version of these estimates in the following lemma.
\begin{lemma}\label{lem:trace}
    For coprime $m, N \in \mathbb{N}$ and even weight $k \geq 2$,
    $$\left| \operatorname{Tr}_{S_k^\sigma(N)}\mathbf{T}_m' - M(k, m, N) \right| \leq E(m, N)$$
    and
    $$\left| \operatorname{Tr}_{S_k^{\nw, \sigma}(N)}\mathbf{T}_m' - M^{\nw}(k, m, N) \right| \leq E(m, N).$$
Here,
\begin{align*}
    M(k, m, N) &= \frac{\mathds{1}_{m = \square}}{\sqrt{m}} \cdot \frac{k - 1}{12} \cdot \frac{\psi(N)}{2^{\omega(N)}}\\
    M^{\nw}(k, m, N) &= \frac{\mathds{1}_{m = \square}}{\sqrt{m}} \cdot \frac{k - 1}{12} \cdot \frac{\psi^{\nw}(N)}{2^{\omega(N)}} \prod_{p^r || N} \left( 1 + \sigma(p^r) \cdot \frac{-\mathds{1}_{r = 2}}{p^2 - p - 1}  \right)\\
    E(m, N) &= 6.7261717 \cdot  m \sigma_0(m) \sigma_0(N)^2 \log(N + 2) \sqrt{N}\\
    \psi(N) &= N \prod_{p \mid N} \left(1 + \frac{1}{p} \right)\\
    \psi^{\nw}(N) &= N \prod_{p \mid N} \left(1 - \frac{1}{p} -  \frac{\mathds{1}_{r \geq 2}}{p^2} + \frac{\mathds{1}_{r \geq 3}}{p^3} \right).
\end{align*}
\end{lemma}

\begin{proof}
    The error bound obtained by explicitly tracking error terms in \cite[Lemma 2.2]{RVWX26} is
    \begin{align}
    &\quad \frac{\sigma_0(N)^2 \sigma_1(m)}{m^{(k - 1)/2}} + m \sigma_0(m) \sigma_0(N) \sqrt{N} + \frac{\sqrt{N} \lrp{\log(4mN)+2} \lrp{4\sqrt{m}+1} \sigma_0(N)^2}{\pi}\\
    &\leq \frac{\sigma_0(N)^2 \sigma_1(m)}{m^{1/2}} + m \sigma_0(m) \sigma_0(N) \sqrt{N} + \frac{\sqrt{N} \lrp{\log(4mN)+2} \lrp{4\sqrt{m}+1} \sigma_0(N)^2}{\pi}\\
    &\qquad \text{(maximizing at $k = 2$)}\\
    &\leq 2 m \sigma_0(m) \sigma_0(N)^2 \sqrt{N} + \frac{\sqrt{N} \lrp{\log(4mN)+2} \lrp{4\sqrt{m}+1} \sigma_0(N)^2}{\pi}\\
    & \qquad \left(\text{since the technique of \cite[Lemma 2.4]{Ros26} gives $\frac{\sigma_1(m)}{\sqrt{m}} \leq m \sigma_0(m)$}\right)\\
    &\leq 2 m \sigma_0(m) \sigma_0(N)^2 \sqrt{N} + \frac{\sqrt{N} \log(N + 2) \lrp{\frac{\log(4m) + 2}{\log 3}} \lrp{4\sqrt{m}+1} \sigma_0(N)^2}{\pi}\\
    & \qquad \left( \text{since for any $m, N \geq 1$, we have $\log(4mN) + 2 \leq \frac{\log (N + 2) (\log (4m) + 2)}{\log 3}$} \right)\\
    &\leq 2 m \sigma_0(m) \sigma_0(N)^2 \sqrt{N} + \frac{\sqrt{N} \log(N + 2) \sigma_0(N)^2}{\pi} \cdot \frac{10 + 5\log 4}{\log 3}m\\
    & \qquad \left(  \text{since $\lrp{\frac{\log(4m) + 2}{\log 3}} \lrp{4\sqrt{m}+1} \leq \frac{10 + 5\log 4}{\log 3}m$ for $m \geq 1$}\right)\\
    &\leq \lrp{\frac{2}{\log 3} + \frac{10 + 5\log 4}{\pi \log 3} } m\sigma_0(m) \sigma_0(N)^2 \log (N + 2) \sqrt{N} \\
    &\le 6.7261717 \cdot m\sigma_0(m) \sigma_0(N)^2 \log (N + 2) \sqrt{N},
    \end{align}
    completing the proof.
\end{proof}

Note that the error function $E(m, N) $ has order of growth $O_\epsilon \! \lrp{m^{1 + \epsilon} N^{1/2 + \epsilon}}.$ For fixed $m,$ this grows slower than both main terms $M(k, m, N)$ and $M^{\nw}(k, m, N)$ as $k + N \rightarrow \infty.$  This is since $\frac{\psi(N)}{2^{\omega(N)}} \gg_\epsilon  N^{1 - \epsilon}$ and $\frac{\psi^\nw(N)}{2^{\omega(N)}} \gg_\epsilon  N^{1 - \epsilon},$ and since $\prod_{p^r || N} \left( 1 + \sigma(p^r) \cdot \frac{-\mathds{1}_{r = 2}}{p^2 - p - 1}  \right) \ge 0.715468$ for admissible sign patterns $\sigma$ (to be shown in the proof of Lemma \ref{lem:explicit-constant-for-big-N-newspace}).

Setting $m = 1$ in the equations in Lemma \ref{lem:trace} yields the following approximation of the dimensions of $S_k^{\sigma}(N)$ and $S_k^{\nw, \sigma}(N).$

\begin{corollary}\label{cor:dimensions-of-spaces} 
    The dimensions of $S_k^{\sigma}(N)$ and $S_k^{\nw,\sigma}(N)$ satisfy
    $$\left| r^{\sigma} - \frac{k - 1}{12} \cdot \frac{\psi(N)}{2^{\omega(N)}} \right| \leq  6.7261717 \cdot  \sigma_0(N)^2 \log(N + 2) \sqrt{N}$$
    $$\left|r^{\nw, \sigma} - \frac{k - 1}{12} \cdot \frac{\psi^{\nw}(N)}{2^{\omega(N)}} \prod_{p^r || N} \left( 1 + \sigma(p^r) \cdot \frac{-\mathds{1}_{r = 2}}{p^2 - p - 1}  \right) \right| \leq  6.7261717 \cdot  \sigma_0(N)^2 \log(N + 2) \sqrt{N}.$$
\end{corollary}

\subsection{The Proof of Theorem \ref{thm:main-theorem}}

We follow the strategy used by Murty and Sinha to prove their equidistribution result, \cite[Theorem 2]{MS07}. We focus specifically on the space $S_k^\sigma(N)$ here, and the proof for $S_k^{\nw, \sigma}(N)$ is exactly the same.

For each eigenvalue $\lambda$ in $\eigen_{S_k^\sigma(N)}( \mathbf{T}_p')$ or $\eigen_{S_k^{\nw, \sigma}(N)}(\mathbf{T}_p'),$ we define $\theta_\lambda \in [0, \pi]$ to be the angle satisfying $2 \cos (\theta_\lambda) = \lambda.$ We have from \cite[\S10]{MS07} that for any $M \in \mathbb{N},$
\begin{align}
    & \quad \left| \frac{\text{eigen}_{S_k^\sigma(N)}(\mathbf{T}_p') \cap I}{r^\sigma} - \int_{-2}^2 \chi_I(x) \, d \mu_p(x)  \right|\\
    &\leq \frac{\| \mu_p \|}{M + 1} + \frac{1}{r^\sigma}\sum_{1 \leq |m| \leq M} \left( \frac{1}{M + 1} + \frac{1}{\pi|m|} \right) \left| \sum_{\lambda \in \eigen_{S_k^\sigma(N)}(\mathbf{T}'_p)}  2\cos\left(m \theta_\lambda \right) - r^\sigma c_m \right|, \label{eqn:ms-section-theorem-8-variant}
\end{align}
where
$$\| \mu_p\| := \sup_{x \in [-2, 2]} \frac{p + 1}{2\pi} \cdot \frac{\sqrt{4 - x^2}}{\left( p^{1/2} + p^{-1/2}  \right)^2 -  x^2}$$
and where the Weyl limits $c_m$ (and $c_m^\nw$) are defined via
\begin{align*}
    c_m &:=
\lim_{k+N\to\infty}
\frac{1}{r^\sigma}
\sum_{\lambda \in \operatorname{eigen}_{S_k^{\sigma}(N)}(\mathbf{T}_p')}
2\cos(m\theta_\lambda).\\
c_m^{\nw} &:= \lim_{k + N\to\infty}
\frac{1}{r^{\nw, \sigma}}
\sum_{\lambda \in \operatorname{eigen}_{S_k^{\nw,\sigma}(N)}(\mathbf{T}_p')}
2\cos(m\theta_\lambda)
\end{align*}
A simple calculation reveals that $\| \mu_p \| \leq \frac{3\sqrt{2}}{4\pi},$ with the bound being tight at $p = 2$ (and that this constant can be improved to $\frac{1}{\pi}$ for $p \geq 7$). Additionally, the Weyl limits will be computed in \eqref{eqn:value-of-cm-for-m-odd},\eqref{eqn:value-of-cm-for-m-even} using the trace estimates of Lemma \ref{lem:trace}. 
We also use the bounds from Lemma \ref{lem:trace} to prove the following result. 

\begin{lemma} \label{lem:weyl-limits}
    For any $m \in \mathbb{Z}$, we have
    $$ \left| \sum_{\lambda \in \eigen_{S_k^\sigma(N)}(\mathbf{T}_p')} 2\cos(m\theta_\lambda) -c_m r^\sigma \right| \leq  2E(p^{|m|},N) + 
    \mathds{1}_{2 \mid m} \cdot \left| p^{-|m|/2} - p^{-(|m| - 2)/2} \right| E(1, N).
    $$
    The same result also holds over $S_k^{\nw, \sigma}(N).$
\end{lemma}

\begin{proof}
    We only prove the result for $S_k^\sigma(N)$, as the proof for the newspace is analogous. Furthermore, since the cosine function is even, the left hand side of the inequality is even, so it suffices to prove the result for $m > 0$. Indeed, if $m = 0$, the left hand side is $0$ and the right side is positive, so the inequality follows.

    When $m = 1,$ we have $\sum_{\lambda \in \eigen_{S_k^\sigma(N)}(\mathbf{T}_p')} 2\cos(m\theta_\lambda) = \operatorname{Tr}_{S_k^\sigma(N)}(\mathbf{T}_p')$ and
    $$|c_1| = \left| \lim_{N + k \rightarrow \infty} \frac{\operatorname{Tr}_{S_k^\sigma(N)}(\mathbf{T}_p')}{\dim S_k^\sigma(N)}  \right| \leq \lim_{N + k \rightarrow \infty} \frac{E(p, N)}{\dim S_k^\sigma(N)} = 0$$
    by Lemma \ref{lem:trace} and Corollary \ref{cor:dimensions-of-spaces}. Therefore,
    $$\left| \sum_{\lambda \in \eigen_{S_k^\sigma(N)}(\mathbf{T}_p')} 2\cos(\theta_\lambda) - c_1 r^\sigma \right| = \left| \operatorname{Tr}_{S_k^\sigma(N)}(\mathbf{T}_p')  \right| \leq E(p, N)$$
    by Lemma \ref{lem:trace}.

    The proof is similar for $m \geq 2.$ Since \(\mathbf{T}_{p^m}' =  U_m\left(\frac{\mathbf{T}_p'}{2} \right)\) \cite[Theorem 10.2.9]{cohen2017modular} (where $U_m$ is the $m$-th Chebyshev polynomial of the second kind), 
the eigenvalues of $\mathbf{T}_{p^m}'$ are
\(
U_m\left( \frac{\lambda}{2}\right) = U_m(\cos\theta_\lambda)\) for \(\lambda \in 
\operatorname{eigen}_{S_k^{\sigma}}(\mathbf{T}_p')\).
Hence
\[
\operatorname{Tr}_{S_k^{\sigma}(N)} \mathbf{T}_{p^m}'
=
\sum_{\lambda \in \operatorname{eigen}_{S_k^{\sigma}}(\mathbf{T}_p')}
U_m(\cos\theta_\lambda).
\]

For $m \geq 2$, using the identity $2\cos(m\theta) = U_m(\cos\theta)-U_{m-2}(\cos\theta),$ we obtain

\begin{align*}
    \sum_{\lambda \in \operatorname{eigen}_{S_k^{\sigma}}(\mathbf{T}_p')} 2\cos(m\theta_\lambda)
&=
\sum_{\lambda \in \operatorname{eigen}_{S_k^{\sigma}}(\mathbf{T}_p')}
\Bigl(
U_m(\cos\theta_\lambda)
-
U_{m-2}(\cos\theta_\lambda)
\Bigr)\\ 
&= \operatorname{Tr}_{S_k^{\sigma}(N)} \mathbf{T}_{p^m}'
-
\operatorname{Tr}_{S_k^{\sigma}(N)} \mathbf{T}_{p^{m-2}}'.
\end{align*}

When $m \geq 2$ is odd, we have
\begin{align} \label{eqn:value-of-cm-for-m-odd}
    |c_m| = \left| \lim_{N+k\to\infty} \frac{ \operatorname{Tr}_{S_k^{\sigma}(N)}\mathbf{T}_{p^m}' - \operatorname{Tr}_{S_k^{\sigma}(N)}\mathbf{T}_{p^{m-2}}'}{\dim S_k^{\sigma}(N)} \right| \leq \lim_{N+k\to\infty} \frac{E(p^m, N) + E(p^{m - 2}, N) }{\dim S_k^{\sigma}(N)} = 0,
\end{align}
by Lemma \ref{lem:trace} and Corollary \ref{cor:dimensions-of-spaces}. Therefore,
\begin{align}
    \left| \sum_{\lambda \in \eigen_{S_k^\sigma(N)}(\mathbf{T}_p')} 2\cos(m\theta_\lambda) - c_m r^\sigma \right| &= \left| \operatorname{Tr}_{S_k^{\sigma}(N)} \mathbf{T}_{p^m}'
-
\operatorname{Tr}_{S_k^{\sigma}(N)} \mathbf{T}_{p^{m-2}}'  \right| \\
&\leq E(p^m, N) + E(p^{m - 2}, N) \\
&\leq 2 E(p^m, N),
\end{align}
by Lemma \ref{lem:trace}. When $m \geq 2$ is even, Lemma \ref{lem:trace} and Corollary \ref{cor:dimensions-of-spaces} give that
\begin{align}
    c_m  &= \lim_{k + N \rightarrow \infty } \frac{\operatorname{Tr}_{S_k^{\sigma}(N)}\mathbf{T}_{p^m}' - \operatorname{Tr}_{S_k^{\sigma}(N)}\mathbf{T}_{p^{m-2}}' }{ \dim S_k^{\sigma}(N)} \\
    &= \lim_{k + N \rightarrow \infty }  \frac{ 
        \left( p^{-m/2}-p^{-(m-2)/2} \right) \dfrac{k-1}{12}\dfrac{\psi(N)}{2^{\omega(N)}} 
        + O_\epsilon \! \left( N^{1/2+\epsilon} \right)
    }{ 
        \dfrac{k-1}{12}\dfrac{\psi(N)}{2^{\omega(N)}} 
        + O_\epsilon \! \left( N^{1/2+\epsilon} \right) 
    }  \\
    &= p^{-m/2}-p^{-(m-2)/2}. \label{eqn:value-of-cm-for-m-even}
\end{align}
Therefore, it follows from Lemma \ref{lem:trace} and Corollary \ref{cor:dimensions-of-spaces} that
\begin{align}
    & \qquad \left| \sum_{\lambda \in \eigen_{S_k^\sigma(N)}(\mathbf{T}_p')} 2\cos(m\theta_\lambda) - c_m r^\sigma \right| \\
    &\leq \left| \operatorname{Tr}_{S_k^{\sigma}(N)} \mathbf{T}_{p^m}'
-
\operatorname{Tr}_{S_k^{\sigma}(N)} \mathbf{T}_{p^{m-2}}' - \left( p^{-m/2}-p^{-(m-2)/2} \right) \frac{k - 1}{12} \cdot \frac{\psi(N)}{2^{\omega(N)}} \right| +\\
&\qquad \left| \left( p^{-m/2}-p^{-(m-2)/2} \right)   \cdot E(1, N) \right| \\
&\leq E(p^m, N) + E(p^{m - 2}, N) + \left|  p^{-m/2}-p^{-(m-2)/2}  \right| E(1, N) \\
&\leq 2 E(p^m, N) + \left|  p^{-m/2}-p^{-(m-2)/2}  \right|  E(1, N),
\end{align}
which proves the desired result.
\end{proof}

Lemma \ref{lem:weyl-limits} can be used to simplify the right hand side of \eqref{eqn:ms-section-theorem-8-variant}, as shown in the following lemma.

\begin{lemma}\label{lem:sum-upper-bound}
    For all $M \in \mathbb{N}$ and sub-intervals $I \subseteq [-2, 2],$ we have
    \begin{align}
    &\quad \sum_{1 \leq |m| \leq M} \left( \frac{1}{M + 1} + \frac{1}{\pi|m|} \right) \left| \sum_{\lambda \in \eigen_{S_k^\sigma(N)}(\mathbf{T}'_p)}  2\cos\left(m \theta_\lambda \right) - r^\sigma c_m \right| \label{eqn:sum-from-lemma-sum-upper-bound}\\
    &\le 70.9375 \cdot p^M \cdot N^{1/2} \sigma_0(N)^2\log(N + 2).
    \end{align}
\end{lemma}

\begin{proof}
    Since the cosine and the Weyl limit are even functions in $m,$ we can rewrite \eqref{eqn:sum-from-lemma-sum-upper-bound} as
\[
\eqref{eqn:sum-from-lemma-sum-upper-bound} =
2\sum_{1\leq m\leq M}
\left(
\frac{1}{M+1}
+
\frac{1}{\pi m}
\right)
\left|
\sum_{\lambda \in \eigen_{S_k^\sigma(N)}(\mathbf{T}'_p)} 2\cos(m\theta_{\lambda})-c_m r^\sigma
\right|.
\]

By Lemmas \ref{lem:trace} and \ref{lem:weyl-limits},
\begin{align*}
    & \quad \left|
\sum_{\lambda \in \eigen_{S_k^\sigma(N)}(\mathbf{T}'_p)} 2\cos(m \theta_\lambda )-c_m r^\sigma
\right| \\
&\leq 2\cdot 6.7261717 \cdot  p^m \sigma_0(p^m) \sigma_0(N)^2 \log(N + 2) \sqrt{N} \\
&\qquad + \mathds{1}_{2 \mid m} \cdot 6.7261717 \,  \sigma_0(N)^2 \log(N + 2) \sqrt{N} \cdot \left| p^{-m/2} - p^{-(m - 2)/2}  \right|  \\
&= \left( (m + 1)p^m  +  \frac{\mathds{1}_{2 \mid m}}{2}\left| p^{-m/2} - p^{-(m - 2)/2}  \right|\right) \cdot 2 \cdot 6.7261717  \,  \sigma_0(N)^2 \log(N + 2) \sqrt{N}.
\end{align*}
Therefore, we obtain
\begin{align}
    \eqref{eqn:sum-from-lemma-sum-upper-bound} &\leq 26.904687 \,\sigma_0(N)^2 \log(N + 2) \sqrt{N}\\
    & \quad\ \times \left[ \sum_{1 \leq m \leq M} \left( \frac{1}{M + 1} + \frac{1}{\pi m} \right) (m + 1)p^m \right. \label{eqn:sum-upper-bound-expression-1}\\
    & \qquad\quad + \left. \sum_{\substack{0 < m \leq M\\ \text{$m$ even}}} \left( \frac{1}{M + 1} + \frac{1}{\pi m} \right) \cdot \frac{1}{2}\left| p^{-m/2} - p^{-(m - 2)/2} \right| \right] . \label{eqn:sum-upper-bound-expression-2}
\end{align}

First, we bound \eqref{eqn:sum-upper-bound-expression-1}. Observe that
\begin{align}
    & \eqref{eqn:sum-upper-bound-expression-1} \\
    &= \frac{1}{M + 1} \left(  \sum_{1 \leq m \leq M} p^m + \sum_{1 \leq m \leq M} mp^m\right) + \frac{1}{\pi} \left( \sum_{1 \leq m \leq M} p^m  + \sum_{1 \leq m \leq M} \frac{p^m}{m} \right)\\
    &= \frac{1}{M + 1} \left( \frac{p(p^M - 1)}{p -1} + \frac{M p^{M+ 2} - (M + 1)p^{M + 1} + p}{(p - 1)^2} \right) + \frac{1}{\pi} \left(\frac{p(p^M - 1)}{p -1}  + \sum_{1 \leq m \leq M} \frac{p^m}{m} \right)\\
    &=  \left(  \frac{(M + 1) p^{M+ 2} - (M + 2)p^{M + 1} + 2p - p^2}{(M + 1)(p - 1)^2} \right) + \frac{1}{\pi} \left( \frac{p^{M + 1}}{p - 1}  + \sum_{1 \leq m \leq M} \frac{p^m}{m} \right) - \frac{1}{\pi} \cdot \frac{p}{p - 1}\\
    &=  \left(  \frac{(M + 1) \left( p^{M + 2} - p^{M + 1} \right)}{(M + 1)(p - 1)^2} - \frac{p^{M + 1}}{(M + 1)} \right) + \frac{1}{\pi} \left(\frac{p^{M + 1}}{p - 1} + \sum_{1 \leq m \leq M} \frac{p^m}{m} \right) - \left( \frac{1}{\pi} + \frac{p^2 - 2p}{(M + 1)(p - 1)^2} \right)\\
    &=  \left( 1 + \frac{1}{\pi} \right) \frac{ p^{M + 1} }{p - 1} - \frac{p^{M + 1}}{(M + 1)}  + \frac{1}{\pi}  \sum_{1 \leq m \leq M} \frac{p^m}{m} - \left( \frac{1}{\pi} + \frac{p^2 - 2p}{(M + 1)(p - 1)^2} \right). \label{eqn:simplified---sum-upper-bound-expression-1}
\end{align}
Additionally, note that
\begin{align*}
    \frac{ \sum_{1 \leq m \leq M} \frac{p^m}{m}}{\frac{p^{M + 1}}{(M + 1)}} &= (M + 1) \sum_{1 \leq m \leq M} \frac{p^{m - M - 1}}{m} \\
    &\leq \frac{1}{p} \sum_{k = 1}^{M - 1} \frac{M + 1}{M - k} \cdot  \frac{1}{p^k} \qquad  \text{substituting $k = M - m$} \\
    &\leq \frac{1}{2} \sum_{k = 1}^{M - 1} \frac{M + 1}{M - k} \cdot  \frac{1}{2^k} \qquad \text{since $p \geq 2$} \\
    &\leq \frac{1}{2} \left( \sum_{k = 0}^{M - 2} \frac{k + 3}{2} \cdot \frac{1}{2^k}  +  \frac{M + 1}{2^{M - 1}} \right)\\
    &< \frac{1}{2} \left( \sum_{k = 0}^{\infty} \frac{k + 3}{2} \cdot \frac{1}{2^k} +  2 \right) \\
    &< \frac{1}{2} \left( \frac{1}{2}\sum_{k = 0}^{\infty} k \cdot \frac{1}{2^k} + \frac{3}{2} \sum_{k = 0}^{\infty} \frac{1}{2^k} +  2 \right) \\
    &\leq \frac{1}{2} \left( 1 + 3 + 2 \right) \qquad  \text{since $\sum_{k = 0}^\infty k x^k = \frac{x}{(1 - x)^2}$}\\
    &< \pi.
\end{align*}
This implies that $- \frac{p^{M + 1}}{(M + 1)}  + \frac{1}{\pi}  \sum_{1 \leq m \leq M} \frac{p^m}{m} < 0,$ so we have from \eqref{eqn:simplified---sum-upper-bound-expression-1} that
\begin{align}
     \eqref{eqn:sum-upper-bound-expression-1} &\leq \left( 1 + \frac{1}{\pi} \right)\frac{ p^{M + 1} }{p - 1} - \left( \frac{1}{\pi} + \frac{p^2 - 2p}{(M + 1)(p - 1)^2} \right) \\
    &\leq \left(2 + \frac{2}{\pi} \right) p^M - \left( \frac{1}{\pi} + \frac{p^2 - 2p}{(M + 1)(p - 1)^2} \right) \label{eqn:simplified-expression-2}.
\end{align}

Second, we bound \eqref{eqn:sum-upper-bound-expression-2}. Observe that for all $M \in \mathbb{N},$
\begin{align}
    \eqref{eqn:sum-upper-bound-expression-2} &\leq \sum_{\substack{0 < m \leq M\\ \text{$m$ even}}} \left( \frac{1}{M + 1} + \frac{1}{\pi} \right) \cdot \frac{1}{2}\left( p^{-(m - 2)/2} - p^{-m/2}\right)\\
    &= \frac{1}{2}\left( \frac{1}{M + 1} + \frac{1}{\pi} \right) \left( 1 - \frac{1}{p^{\lfloor M/2 \rfloor}} \right)\\
    &\leq  \frac{1}{\pi} + \frac{p^2 - 2p}{(M + 1)(p - 1)^2}.
\end{align}
The final inequality is verified through the following cases.
\begin{itemize}
    \item If $M \geq 3,$ then $M + 1 > \pi,$ so
    $$\frac{1}{2}\left( \frac{1}{M + 1} + \frac{1}{\pi} \right) \left( 1 - \frac{1}{p^{\lfloor M/2 \rfloor}} \right) \leq \frac{1}{\pi} \left( 1 - \frac{1}{p^{\lfloor M/2 \rfloor}} \right) \leq \frac{1}{\pi} \leq \frac{1}{\pi} + \frac{p^2 - 2p}{(M + 1)(p - 1)^2}.$$
    \item If $M = 1,$ then $\frac{1}{2}\left( \frac{1}{M + 1} + \frac{1}{\pi} \right) \left( 1 - \frac{1}{p^{\lfloor M/2 \rfloor}} \right) = 0.$
    \item If $M = 2$ and $p \leq 43,$ then
    $$\frac{1}{2}\left( \frac{1}{M + 1} + \frac{1}{\pi} \right) \left( 1 - \frac{1}{p^{\lfloor M/2 \rfloor}} \right) \leq \frac{1}{\pi} \leq \frac{1}{\pi} + \frac{p^2 - 2p}{(M + 1)(p - 1)^2}.$$
    \item If $M = 2$ and $p \geq 47,$ then
    $$\frac{1}{2}\left( \frac{1}{M + 1} + \frac{1}{\pi} \right) \left( 1 - \frac{1}{p^{\lfloor M/2 \rfloor}} \right) \leq \frac{1}{2} \left( \frac{1}{3} + \frac{1}{\pi} \right) \leq \frac{1}{\pi} + \frac{47^2 - 2 \cdot 47}{3 (47 - 1)^2} \leq \frac{1}{\pi} + \frac{p^2 - 2 \cdot p}{(M + 1) (p - 1)^2},$$
    since $\frac{1}{\pi} + \frac{p^2 - 2 \cdot p}{(M + 1) (p - 1)^2}$ is increasing in $p$ for $M = 2.$
\end{itemize}

Combining the above bounds for \eqref{eqn:sum-upper-bound-expression-1} and \eqref{eqn:sum-upper-bound-expression-2}, we obtain $\eqref{eqn:sum-upper-bound-expression-1} + \eqref{eqn:sum-upper-bound-expression-2} \leq \left( 2 + \frac{2}{\pi} \right) p^M,$ so that
\begin{align}
    \eqref{eqn:sum-from-lemma-sum-upper-bound} &\le 26.904687 \, \sigma_0(N)^2 \log(N + 2) \sqrt{N} \Big[\eqref{eqn:sum-upper-bound-expression-1} + \eqref{eqn:sum-upper-bound-expression-2}\Big] \\
    &\le 70.9375 \cdot p^M \cdot \sigma_0(N)^2 \log(N + 2) \sqrt{N},
\end{align}
as desired.
\end{proof}

Finally, we apply Lemma \ref{lem:sum-upper-bound} to \eqref{eqn:ms-section-theorem-8-variant} and set $M:=  \left\lfloor \frac{c \log kN}{\log p} \right\rfloor$, where $c$ is any fixed constant in $\left(0, \frac{1}{2} \right).$ This yields
\begin{align}
    & \quad \left| \frac{\text{eigen}_{S_k^\sigma(N)}(\mathbf{T}_p') \cap I}{r^\sigma} - \int_{-2}^2 \chi_I(x) \, d \mu_p(x)  \right| \\
    &\leq  \frac{\frac{3\sqrt{2}}{4\pi}}{M + 1} + \frac{70.9375 \cdot p^M \cdot N^{1/2} \sigma_0(N)^2\log(N + 2)}{r^\sigma}\\
    &\leq \frac{\frac{3\sqrt{2}}{4\pi}}{\left\lfloor \frac{c \log kN}{\log p} \right\rfloor + 1} + \frac{70.9375 \cdot (kN)^c \cdot N^{1/2} \sigma_0(N)^2\log(N + 2)}{r^\sigma}\\
    &\leq \frac{3\sqrt{2}}{4\pi c} \cdot \frac{\log p}{ \log kN} + \frac{70.9375 \cdot k^c \cdot N^{1/2 + c} \sigma_0(N)^2 \log(N + 2)}{r^\sigma}. \label{eqn:MS-Theorem-19-Analog}\\
    &\ll \frac{\log p}{\log kN} \qquad \text{by Corollary \ref{cor:dimensions-of-spaces}.}
\end{align}
This completes the proof of Theorem \ref{thm:main-theorem}.

\subsection{Corollaries of Theorem \ref{thm:main-theorem}}

From Theorem \ref{thm:main-theorem}, we can extract several effective results on the equidistribution of Hecke eigenvalues over the larger spaces $S_k(N)$, $S_k^\pm(N)$, and $S_k^{\nw}(N),$ $S_k^{\nw, \pm }(N)$. This is because these spaces are direct sums of the sign pattern spaces $S_k^\sigma(N)$ and $S_k^{\nw, \sigma}(N)$, respectively.

\begin{corollary} \label{cor:main-theorem-for-larger-spaces}
    Let $S = S_k(N)$ or $S_k^\pm(N)$. For a prime $p \nmid N$ and $I$ a sub-interval of $[-2, 2]$, we have
    $$\left| \frac{\# \operatorname{eigen}_S(\mathbf{T}_p') \cap I}{\dim S} - \int_{-2}^2 \chi_I(x) \, d \mu_p(x)  \right| \leq C_0  \frac{\log p}{\log  kN}.$$
    The same result holds for $S_k^{\nw}(N)$ and $S_k^{\nw, \pm}(N)$ with the constant $C_0^{\nw}$ in place of $C_0.$
\end{corollary}

\begin{proof}
    We supply the proof for $S_k(N)$, as the method is the same for all other cases. Recall that
    $$S_k(N) = \bigoplus_{\sigma} S_k^\sigma(N),   \qquad \dim S_k(N) = \sum_{\sigma} \dim S_k^\sigma(N),$$
    where the sum is over all admissible sign patterns. Then we have that
    \begin{align}
        &\left| \# \operatorname{eigen}_{S_k(N)}(\mathbf{T}_p') \cap I  - \dim S_k (N) \int_{-2}^2 \chi_I(x) \, d\mu_p(x) \right|\\
        & = \left| \# \sum_{\sigma} \operatorname{eigen}_{S_k^\sigma(N)}(\mathbf{T}_p') \cap I  - \sum_{\sigma} \dim S_k^\sigma (N) \int_{-2}^2 \chi_I(x) \, d\mu_p(x) \right|\\
        &\leq \sum_{\sigma} \left| \# \operatorname{eigen}_{S_k^\sigma(N)}(\mathbf{T}_p') \cap I - \dim S_k^\sigma (N) \int_{-2}^2 \chi_I(x) \, d\mu_p(x) \right| \quad\quad \\
        &\leq \sum_{\sigma} \dim S_k^\sigma(N) \cdot C_0 \frac{\log p}{\log kN} \qquad \text{(by Theorem \ref{thm:main-theorem})}\\
        &= \dim S_k(N) \cdot C_0 \frac{\log p}{\log kN},
    \end{align}
    completing the proof.
\end{proof}

Next, we note that the effective equidistribution of Theorem \ref{thm:main-theorem} applies to more general test functions than just characteristic functions $\chi_I.$ Specifically, we have the following.

\begin{corollary}\label{cor:main-thm-for-general-function-g}
    Let $g$ be a test function of bounded variation over $[-2, 2]$. For all primes $p,$ levels $N$ coprime to $p,$ admissible sign patterns $\sigma,$ and weights $k,$ we have 
    \begin{equation}
    \left| \frac{1}{\dim S_k^\sigma(N)} \sum_{\lambda \in \operatorname{eigen}_{S_k^\sigma(N)}(\mathbf{T}_p')} g\left( \lambda  \right) \ -\  \int_{-2}^2 g(x) \,d\mu_p \right| \le C_0 \frac{\delta(g)\log p}{\log (k N)},
\end{equation}
    where $\delta(g)$ denotes the total variation of $g$ over $[-2, 2].$ Moreover, the same result also holds for $S_k^{\nw,\sigma}(N)$ with $C_0^\nw$ in place of $C_0.$
\end{corollary}

The proof of Corollary \ref{cor:main-thm-for-general-function-g} is exactly the same as that of Murty and Sinha's analogous \cite[Theorem 9]{MS07}, which in turn follows a similar classical result of Koksma (see \cite[Theorem 5.1]{kuipers1974uniform}). Lastly, as an analog of \cite[Theorem 4]{MS07} over the Atkin--Lehner subspaces $S_k^\sigma(N)$ and $S_k^{\nw, \sigma}(N)$, we establish a bound on the number of normalized Hecke eigenvalues equal to any fixed value in $[-2, 2].$ For any $\alpha \in [-2, 2]$, we have that $\int_{-2}^2 \chi_{\{\alpha\}}(x) \, d\mu_p(x) = 0,$ so Theorem \ref{thm:main-theorem} gives the following.

\begin{corollary}\label{cor:main-theorem-with-singleton} 
    For any $\alpha \in [-2, 2],$
    $$\# \lrcb{\lambda \in  \operatorname{eigen}_{S_k^{ \sigma}(N)}(\mathbf{T}'_p) : \lambda = \alpha} \leq \dim S_k^\sigma(N) \cdot C_0 \frac{\log p}{\log kN}.$$
    Moreover, the analogous result also holds over $S_k^{\nw, \sigma}(N)$.
\end{corollary}

Similar bounds hold over the larger spaces $S_k(N), S_k^\nw(N), S_k^\pm(N),$ and $S_k^{\nw, \pm}(N).$

\section{Calculation of \texorpdfstring{$C_0$}{C0}}\label{section:explicit-constants}

Additional work allows us to gain information about the constants $C_0$ and $C_0^\nw$ from Theorem \ref{thm:main-theorem}. In particular, we show these constants can be taken as $\frac{3\sqrt{2}}{2\pi} +\epsilon$ for sufficiently large $N$.
\begin{theorem}\label{thm:limit-constant}
For all $\epsilon > 0$, there exists an effectively computable constant $N_\epsilon$ such that for all primes $p$, even $k \geq 2,$ $N \geq N_\epsilon$ coprime to $p,$ admissible $\sigma$, and intervals $I \subset [-2, 2],$ we have
\[
\left|
\frac{\# \operatorname{eigen}_{S_k^\sigma(N)}(\mathbf{T}'_p) \cap I}{\dim S_k^\sigma(N)}
-
\int_{-2}^2 \chi_I(x) \, d\mu_p(x)
\right|
\leq
\left( \frac{3\sqrt{2}}{2\pi} +\epsilon \right) \frac{\log p}{\log kN}.
\]
The same result also holds for $S_k^{\nw,\sigma}(N)$ (with a different constant $N_\epsilon^\nw$ replacing $N_\epsilon$). 
\end{theorem}
Moreover, we write down the absolute constants $C_0$ and $C_0^\nw$ and establish several valid pairs $(\epsilon, N_\epsilon)$ and $(\epsilon, N_\epsilon^\nw).$

\begin{proposition} \label{prop:explicit-constant-chart}
    The absolute constants $C_0$ and $C_0^\nw$ in Theorem \ref{thm:main-theorem} can be taken as $1952$ and $2085,$ respectively.
    
    Moreover, the following table includes some values of $N_\epsilon$ for various values of $\frac{3\sqrt{2}}{2\pi}+\epsilon.$

\begin{center}
    \begin{tabular}{|c||c|c|c|c|c|c|c|c|c|c|}
    \hline
       \phantom{$\frac{1^{1^1}}{1^{1^1}}$}
       $\frac{3\sqrt{2}}{2\pi} + \epsilon$ & $100$ & $10$ & $5$ & $2$ & $1.5$ & $1$ & $0.9$\\
    \hline
       $\log N_\epsilon$ & $719.038$ & $1070.23$ & $1682.40$ & $12279.7$ & $78148.5$ & $1.85386\times10^{8}$ & $5.64894 \times10^{10}$ \\
    \hline
       $\log N_\epsilon^\nw$ & $763.470$ & $1124.84$ & $1753.68$ & $12466.7$ & $78505.7$ & $1.85388\times10^{8}$ & $5.64894 \times10^{10}$  \\
    \hline
    \end{tabular}
\end{center}
\end{proposition}

If one restricts to prime levels $N$, the bounds on these constants improve significantly. We use $C_0^{\operatorname{prime}}$ to denote a universal constant such that
$$\left| \frac{\# \operatorname{eigen}_{S_k^{\sigma}(N)}(\mathbf{T}'_p) \cap I}{\dim S_k^{\sigma}(N)} - \int_{-2}^2 \chi_I(x) \, d \mu_p(x)  \right| \le C_0^{\operatorname{prime}}  \frac{\log p}{\log kN}$$
over all prime levels $N$ and choices of $k, p, I,$ and $\sigma.$
Likewise, we use the notation $C_0^{\nw, \operatorname{prime}},$ $N_\epsilon^{\operatorname{prime}},$ and $N_\epsilon^{\nw, \operatorname{prime}}$.

\begin{proposition}\label{prop:explicit-constant-chart-prime}
    The constants $C_0^{\operatorname{prime}}$ and $C_0^{\nw, \operatorname{prime}}$ can both be taken to be $86.$
    
    Moreover, the following table includes some values of $N_\epsilon^{\operatorname{prime}}$, for various values of $\epsilon,$ over prime levels $N.$

\begin{center}
    \begin{tabular}{|c||c|c|c|c|c|c|c|}
    \hline
       \phantom{$\frac{1^{1^1}}{1^{1^1}}$}
       $\frac{3\sqrt{2}}{2\pi} + \epsilon$ & $10$ & $5$ & $2$ & $1.5$ & $1$ & $0.8$ & $0.7$\\
    \hline
       $\log N_\epsilon^{\operatorname{prime}}$ & $33.8901$ & $40.1397$ & $61.1604$ & $78.3464$ & $149.347$ & $349.248$ & $1835.92$ \\
    \hline
    \end{tabular}
\end{center}
The same values hold for $N_\epsilon^{\nw, \text{prime}}$ as well.
\end{proposition}

Theorem \ref{thm:limit-constant} will follow immediately from the stronger results Lemmas \ref{lem:explicit-constant-for-big-N} and \ref{lem:explicit-constant-for-big-N-newspace}, which we devote \S \ref{subsection:explicit-constants-for-large-N} to proving. These results allow one to compute $N_\epsilon$ and $N_{\epsilon}^\nw$ numerically for any given value of $\epsilon > 0$. In particular, the entries of the table in Proposition \ref{prop:explicit-constant-chart} also follow immediately from Lemmas \ref{lem:explicit-constant-for-big-N} and \ref{lem:explicit-constant-for-big-N-newspace}. The universal upper-bound on the coefficients $C_0$ and $C_0^\nw$, meanwhile, will be proved in \S \ref{subsection:c-upper-bound}. Finally, in \S \ref{subsection:constants-for-prime-level}, we prove Lemma \ref{lem:explicit-constant-for-big-N-prime}, which is a prime level analog of Lemmas \ref{lem:explicit-constant-for-big-N} and \ref{lem:explicit-constant-for-big-N-newspace}. Proposition \ref{prop:explicit-constant-chart-prime} will follow immediately from this result. 

\subsection{Explicit constants for large levels \textit{N}}\label{subsection:explicit-constants-for-large-N}

This subsection is devoted to proving Lemmas \ref{lem:explicit-constant-for-big-N} and \ref{lem:explicit-constant-for-big-N-newspace}. To do this, we will need to upper-bound the functions $\sigma_0(N)$ and $\omega(N).$ We recall that
Nicolas and Robin showed in \cite[Th\'eor\`eme 1]{NR83} that for $N \geq 3,$
    \begin{equation}
        \sigma_0(N) \leq 2^{1.53794\frac{ \log N}{\log \log N}}, \label{eqn:Nicolas-Robin-bound}
    \end{equation}
    and Robin showed in \cite[Th\'eor\`eme 11]{Rob83} that
    \begin{equation}
        \omega(N) \leq 1.38407\frac{\log N}{\log \log N}. \label{eqn:Robin-bound}
    \end{equation}
Throughout this section, let $W(x)$ denote the positive branch of the Lambert $W$ function, which is defined for $x > -1/e$ and provides an inverse of the function $y = xe^x.$ An explicit form of $W$ (see \cite[Theorem 3.1]{kalugin2012stieltjes}) is given by
\begin{align}
    W(x) &= \frac{x}{\pi} \int_0^\pi \frac{(1 - \theta \cot \theta)^2 + \theta^2}{x + \theta \csc(\theta) e^{-\theta \cot (\theta)}} d \theta.
\end{align}
We remark that to numerically compute the values of $W(x)$ for large $x$, it is useful to employ the series
$$W(x) = \log(x) - \log \log (x) + \frac{\log \log x}{\log x} + \frac{\log \log x \left( \log \log x - 2  \right)}{2 \log(x)^2} + \cdots,
$$
or truncations thereof as in \cite[(4.19)]{corless1996lambertw}.

For fixed constants $A > \frac{3\sqrt{2}}{\pi}$ and $B, C, D \in \mathbb{Z}_{\geq 0},$ we define the functions over the variables $x > 0$ and $K > 0$
\begin{align}
        V_x(K) &:= \frac{1}{2} - \frac{\log\left( \frac{x^{D + 1}}{K}  \right)}{x} - \frac{\log \left( 2^{1.38407B +  1.53794C}  \right)}{\log x}\\
        F_x(K) &:= \frac{3\sqrt{2}}{4\pi} \cdot V_x(K)^{-1} + AK.\\
        K(x) &:= \frac{3\sqrt{2} x}{16 \pi A} \cdot W\! \left( \sqrt{\frac{3\sqrt{2}}{16 \pi A}} \cdot x^{-D/2} \exp\left( \frac{x}{4} - \frac{x \log\left( 2^{1.38407B +  1.53794C}  \right)}{2\log x}  \right)  \right)^{-2}.
    \end{align}

Our main mechanism for proving Lemmas \ref{lem:explicit-constant-for-big-N} and \ref{lem:explicit-constant-for-big-N-prime} will be the following result, whose proof we defer to Appendix \ref{section:lambert-lemmata}.

\begin{lemma}\label{lem:lambert-w-function}
    Fix $A > \frac{3\sqrt{2}}{\pi}$ and $B, C, D \in \mathbb{Z}_{\geq 0}.$ Suppose that $S(k, N, p)$ is a function over integers $k, p \geq 2$ and $N \geq 1129$ such that for any $c \in \left( 0, \frac{1}{2} \right),$
    $$S(k, N, p) \leq \frac{3\sqrt{2}}{4\pi c} \cdot \frac{\log p}{\log kN} + \frac{A\log 2}{\sqrt{2} \log 3} \cdot \frac{k^c}{k - 1} \cdot \frac{2^{ B\omega(N)} \sigma_0(N)^C \log(N)^D}{N^{1/2 - c}}.$$
    Furthermore, suppose that $X \geq \log(1129)$ such that for all $x \geq X,$
    \begin{equation}
            \log K(x) < (D + 1) \left( \log x - 1 \right) + \frac{x \log\left( 2^{1.38407B +  1.53794C} \right)}{\left( \log x \right)^2}. 
            \label{eqn:criteron-for-F-decreasing}
        \end{equation}
    Then, for all $N \geq e^X,$
    \begin{align}
        S(k, N, p) &\leq F_{\log N} \left( K(\log N)  \right) \cdot \frac{\log p}{\log kN},\\
        S(k, N, p) &\leq F_X(K(X)) \cdot \frac{\log p}{\log kN}.
    \end{align}
\end{lemma}

We now apply Lemma \ref{lem:lambert-w-function} to prove our desired bounds on $C_0$ and $C_0^\nw$ for large $N.$

    \begin{lemma}\label{lem:explicit-constant-for-big-N}
    Let $p$ be a prime and $I$ a sub-interval of $[-2, 2]$. Then for all $N \geq e^X \geq e^{593.590},$
        $$\left| \frac{\# \operatorname{eigen}_{S_k^\sigma(N)}(\mathbf{T}'_p) \cap I}{r^\sigma} - \int_{-2}^2 \chi_I (x) \, d \mu_p(x)  \right| \leq F_{X}(K(X)) \cdot \frac{\log p}{\log kN}.$$
    Here, $F_X(K(X))$ is defined with the fixed constants $A = 4550.16, \ B = D = 1,$ and $C = 2.$
\end{lemma}

\begin{proof}
    First, we note that we can refine the constant on the error term $E(m, N)$ from Lemma \ref{lem:trace} when $N \geq e^{593.590}.$ In this case, the error is bounded by
    \begin{align}
    &\quad 2 m \sigma_0(m) \sigma_0(N)^2 \sqrt{N} + \frac{\sqrt{N} \sigma_0(N)^2\lrp{\log(4mN)+2} \lrp{4\sqrt{m}+1}}{\pi}\\
    &= 2 m \sigma_0(m) \sigma_0(N)^2 \sqrt{N} + \frac{\sqrt{N} \sigma_0(N)^2 \left( 1 + \frac{\log N}{\log(4m) + 2}  \right) \left(\log(4m) + 2 \right)  \lrp{4\sqrt{m}+1}}{\pi}\\
    &\leq 2 m \sigma_0(m) \sigma_0(N)^2 \sqrt{N} + \frac{\sqrt{N} \sigma_0(N)^2 \log N\left( \frac{1}{593.590} + \frac{1}{\log(4m) + 2}  \right) \left(\log(4m) + 2 \right) \lrp{4\sqrt{m}+1} }{\pi}\\
    &\leq \frac{2 \log N}{593.590} m \sigma_0(m) \sigma_0(N)^2 \sqrt{N} + \frac{\sqrt{N}\sigma_0(N)^2 \log N}{\pi} \cdot \left( \frac{1}{593.590} + \frac{1}{\log(4) + 2}  \right) \cdot(10 + 5\log 4)m\\
    &\leq 1.60400\, m \sigma_0(m) \sqrt{N} \sigma_0(N)^2 \log N  =: E'(m, N).
    \end{align}
    Applying Corollary \ref{cor:dimensions-of-spaces} with $E'(m, N)$ in place of $E(m, N),$ we obtain that
    \begin{align}
        r^{\sigma} &\geq \frac{k - 1}{12} \cdot \frac{\psi(N)}{2^{\omega(N)}} - 1.60400 \cdot  \sigma_0(N)^2 \log(N) \sqrt{N}\\
        &\geq \frac{k - 1}{12} \cdot \frac{\psi(N)}{2^{\omega(N)}} - \frac{9}{10} \cdot  \frac{k - 1}{12} \cdot \frac{\psi(N)}{2^{\omega(N)}} \qquad \text{(by \eqref{eqn:Nicolas-Robin-bound}, \eqref{eqn:Robin-bound}, and $\psi(N) \geq N$)}\\
        &= \frac{1}{10} \cdot \frac{k - 1}{12} \cdot \frac{\psi(N)}{2^{\omega(N)}}.
    \end{align}
    Then, using an analog of \eqref{eqn:MS-Theorem-19-Analog} proved from $E'(m, N)$ instead of $E(m, N),$ we obtain that for any $c \in \left(0, \frac{1}{2} \right),$
    \begin{align*}
        \left| \frac{\text{eigen}_{S_k^\sigma(N)}(\mathbf{T}_p') \cap I}{r^\sigma} - \int_{-2}^2 \chi_I(x) \, d \mu_p(x)  \right| &\leq \frac{3\sqrt{2}}{4\pi c} \cdot \frac{\log p}{\log kN}  + \frac{16.9166 \cdot k^c \cdot N^{1/2 + c} \sigma_0(N)^2 \log N}{r^\sigma}\\
        &\leq \frac{3\sqrt{2}}{4\pi c} \cdot \frac{\log p}{\log kN} + \frac{16.9166 \cdot k^c \cdot N^{1/2 + c} \sigma_0(N)^2 \log N}{\frac{1}{10} \cdot \frac{k - 1}{12} \cdot \frac{\psi(N)}{2^{\omega(N)}}}
    \end{align*}
The result then follows from using Lemma \ref{lem:lambert-w-function} in the case where $A = 4550.16,$ $B = D = 1,$ and $C = 2.$ We note that \eqref{eqn:criteron-for-F-decreasing} is satisfied for all $X \geq 593.590.$
\end{proof}

A similar result also holds over $S_k^{\nw, \sigma}(N).$

\begin{lemma}\label{lem:explicit-constant-for-big-N-newspace}  
    Let $p$ be a prime and $I$ a sub-interval of $[-2, 2]$. Then for all $N \geq e^X \geq e^{637.948},$
        $$\left| \frac{\# \operatorname{eigen}_{S_k^{\nw, \sigma}(N)}(\mathbf{T}'_p) \cap I}{r^{\nw, \sigma}} - \int_{-2}^2 \chi_I(x) \, d \mu_p(x)  \right| \leq F_{X}\left( K(X)  \right) \cdot \frac{\log p}{\log kN}.$$
    Here, $F_X(K(X))$ is defined with the fixed constants $A = 421.875,$ $B = 1,$ and $C = D = 2.$
\end{lemma}

\begin{proof}
    As in the proof of Lemma \ref{lem:explicit-constant-for-big-N}, we can replace the error term $E(m, N)$ with a new error term
    $$E'(m, N) = 1.60314 \, m \sigma_0(m) \sigma_0(N)^2 \log (N) \sqrt{N}.$$
    Corollary \ref{cor:dimensions-of-spaces}, with $E(m, N)$ replaced by $E'(m, N)$, implies that, for $N \geq e^{637.948}$
    \begin{align}
        r^{\nw, \sigma} &\geq \frac{k - 1}{12} \cdot \frac{N \prod_{p \mid N} \left(1 - \frac{\mathds{1}_{r \geq 2}}{p} -  \frac{1}{p^2}   +  \frac{\mathds{1}_{r \geq 3}}{p^3}  \right)}{2^{\omega(N)}} \prod_{p^r || N} \left( 1 + \sigma(p^r) \cdot \frac{-\mathds{1}_{r = 2}}{p^2 - p - 1}  \right) - E'(1, N)\\
        &\geq \frac{k - 1}{12} \cdot \frac{N \prod_{p \mid N} \left(1 - \frac{1}{p} - \frac{1}{p^2} \right)}{2^{\omega(N)}} \prod_{p \neq 2} \left( 1 - \frac{1}{p^2 - p - 1}  \right) - E'(1, N)\\
        &\qquad \text{(we have $p \neq 2$ since we cannot have $\sigma(p^r) = 1$ for $p = r = 2$ for any admissible $\sigma$}\\
        & \qquad \text{over $S_k^{\nw}(N)$)}\\
        &\geq \frac{k - 1}{12} \cdot \frac{N \prod_{p \mid N} \left(1 - \frac{1}{p} \right) \cdot \prod_{p \mid N} \left(1 - \frac{1}{p(p - 1)} \right)}{2^{\omega(N)}} \cdot 0.715468 - E'(1, N)\\
        &\geq \frac{k - 1}{12} \cdot \frac{\phi(N) \cdot C_{\text{Artin}}}{2^{\omega(N)}} \cdot 0.715468 - E'(1, N) \qquad \left( \text{where } C_{\text{Artin}} := \prod_p \lrp{1 - \frac{1}{p(p-1)}} \right)\\
        &\geq \frac{k - 1}{12} \cdot \frac{N}{e^{\gamma} \log\log N + \frac{2.5}{\log \log N}} \cdot \frac{1}{2^{\omega(N)}} \cdot C_{\text{Artin}} \cdot  0.715468 - E'(1, N)\\
        & \qquad \left( \text{for $N \geq 223 092 870$ by \cite[Theorem 15]{RS62}}  \right)\\
        & \geq \frac{1}{10} \cdot \frac{k - 1}{12} \cdot \frac{N}{e^{\gamma} \log\log N + \frac{2.5}{\log \log N}} \cdot \frac{1}{2^{\omega(N)}} \cdot C_{\text{Artin}} \cdot  0.715468 \qquad \left( \text{by \eqref{eqn:Nicolas-Robin-bound} and \eqref{eqn:Robin-bound}}  \right)\\
        & \geq \frac{1}{10} \cdot \frac{k - 1}{12} \cdot \frac{N}{0.0248200 \log N} \cdot \frac{1}{2^{\omega(N)}} \cdot C_{\text{Artin}} \cdot  0.715468.
    \end{align}
    Then an analog of \eqref{eqn:MS-Theorem-19-Analog} that uses $E'(m, N)$ in place of $E(m, N)$ implies that for any $c >0,$
    $$\left| \frac{\text{eigen}_{S_k^\sigma(N)}(\mathbf{T}_p') \cap I}{r^\sigma} - \int_{-2}^2 \chi_I(x) \, d \mu_p(x)  \right| \leq \frac{\log p}{c \log kN} + \frac{16.9075 \cdot k^c \cdot N^{1/2 + c} \sigma_0(N)^2 \log N}{\frac{k - 1}{120} \cdot \frac{N}{0.02482 \log N} \cdot \frac{1}{2^{\omega(N)}} \cdot C_{\text{Artin}} \cdot  0.715468}.$$
The result now follows from Lemma \ref{lem:lambert-w-function} in the case where $A = 421.875,$ $B = 1$, and $C = D = 2.$ Note that given these choices of variables, the condition \eqref{eqn:criteron-for-F-decreasing} holds for all $X \geq 637.947.$
\end{proof}

We prove in Lemma \ref{lem:limits-of-lambert-functions} that
$\lim_{x \rightarrow \infty} F(x) = \frac{3\sqrt{2}}{2\pi}.$ Therefore, Lemmas \ref{lem:explicit-constant-for-big-N} and \ref{lem:explicit-constant-for-big-N-newspace} imply Theorem \ref{thm:limit-constant}

\subsection{Calculation of the universal constants}\label{subsection:c-upper-bound}

Here, we use the above results to show that the universal constants $C_0$ and $C_0^\nw$ can be taken as $1952$ and $2085$, respectively.

\begin{proof}[Proof of values of $C_0$ and $C_0^\nw$]
    We demonstrate the proof for $C_0$ only, since the approach is similar for $C_0^\nw$. We split into three cases.

    \textbf{Case 1: $\log N \geq 647.511$}

    If $\log N \geq 647.511,$ then for any $k \geq 2$ it follows immediately from Lemma \ref{lem:explicit-constant-for-big-N} that $C_0$ can be taken as $1951.94.$

    \textbf{Case 2: $\log N \leq 647.511$ and $\log k \leq 29$}

    The left hand side of the inequality
    $$\left| \frac{\# \operatorname{eigen}_{S_k^{\sigma}(N)}(\mathbf{T}'_p) \cap I}{\dim S_k^{\sigma}(N)} - \int_{-2}^2 \chi_I(x) \, d \mu_p(x)  \right| \le C_0  \frac{\log p}{\log kN}$$
    obtained from Theorem \ref{thm:main-theorem} can be at most 2, since $\mu_p(x)$ is a probability measure. Therefore, for any individual choice of weight and level $(k, N),$ the inequality holds as long as
    $$C_0 \geq \frac{2}{\log p} \cdot \log kN \geq \frac{2}{\log 2} \log kN.$$
    With our restrictions on the sizes on $k$ and $N,$ it is clear that any $C_0 \geq 1952.00$ is valid.

    \textbf{Case 3: $\log N \leq 647.511$ and $\log k \geq 29$}

    By Corollary \ref{cor:dimensions-of-spaces} and \eqref{eqn:MS-Theorem-19-Analog}, the constant $C_0$ can be taken as
    \begin{align}
        &\quad \frac{3\sqrt{2}}{4\pi c} + \frac{\log kN}{\log 2} \cdot 70.9374 \cdot \frac{k^c \cdot N^{1/2 + c} \sigma_0(N)^2 \log (N + 2)}{\frac{k - 1}{12} \cdot \frac{N}{2^{\omega(N)}} - 6.7261717 \cdot \sigma_0(N)^2 \log(N + 2) \sqrt{N}}\\
        &\leq \frac{3\sqrt{2}}{4\pi c} + \frac{\log kN}{\log 2} \cdot 70.9374 \cdot \frac{k^c \cdot N^{1/2 + c} \sigma_0(N)^2 \log (N + 2)}{\left( 1 - 10^{-1} \right) \cdot \frac{k - 1}{12} \cdot \frac{N}{2^{\omega(N)}}}\\
        &\qquad \left( \text{by \eqref{eqn:Nicolas-Robin-bound} and \eqref{eqn:Robin-bound} when $N \geq 4$, and by explicit computation otherwise} \right)\\
        &\leq \frac{3\sqrt{2}}{4\pi c} + \frac{168.688 + \log N}{\log 2} \cdot 70.9374 \cdot \frac{10}{9} \cdot \frac{e^{168.688c} \cdot N^{1/2 + c} \sigma_0(N)^2 \log (N + 2)}{ \frac{e^{168.688} - 1}{12} \cdot \frac{N}{2^{\omega(N)}}}  \\
        &\leq 1858.17977 \qquad \left( \text{for $c = \frac{1}{500}$}  \right).
    \end{align}
    In each of the above cases, it suffices to take $C_0 = 1952$. This suffices for the proof.
\end{proof}

Recall that by Corollary \ref{cor:main-theorem-for-larger-spaces}, these constants $C_0$ and $C_0^\nw$ can be used to give absolute bounds for analogous results over the spaces $S_k(N),\  S_k^\pm(N)$ and $S_k^\nw(N),\ S_k^{\nw,\pm}{(N)}.$ Thus, Proposition \ref{prop:explicit-constant-chart} provides an explicit constant for Murty--Sinha's main result \cite[Theorem 2]{MS07}. Likewise, analogs of Theorem \ref{thm:limit-constant} hold over these larger spaces, as we summarize in the corollary below.

\begin{corollary} \label{cor-limit-constant-2}
    Let $S = S_k(N)$ or $S_k^\pm(N).$ For prime $p$ and $I$ a sub-interval of $[-2, 2]$. For any $\epsilon > 0,$ we have 
    $$\left| \frac{\# \operatorname{eigen}_{S}(\mathbf{T}_p') \cap I }{\dim S} - \int_{-2}^2 \chi_I(x) \, d \mu_p(x)  \right| \leq \left(\frac{3\sqrt{2}}{2\pi} + \epsilon \right) \cdot  \frac{\log p}{\log  kN},$$
    for all $N \geq N_\epsilon$ such that $(N, p) = 1,$ where $N_\epsilon$ is the same constant as in Theorem \ref{thm:limit-constant}. The analogous statement also holds for $S = S_k^{\nw}(N)$ or $S_k^{\nw, \pm }(N),$ with $N_\epsilon^\nw$ in place of $N_\epsilon.$
\end{corollary}

\subsection{Explicit constants for prime levels \textit{N}}\label{subsection:constants-for-prime-level}

For prime $N$, the result over the spaces $S_k^{\nw, \sigma}(N) = S_k^{\nw, \pm}(N)$ can be improved further due to a bound on $r^{\sigma, \nw}$ by Martin \cite[Theorem 2.2]{martin2018refined}, which gives that for square-free $N > 3,$
$$\left| \dim S_k^{\nw, \pm}(N) - \frac{1}{2} \dim S_k^\nw(N) \right| \leq \frac{1}{2} + \frac{1}{4} h(\Delta_N) \cdot \begin{cases}
    1 & N \not\equiv  3 \mod 4\\
    2 & N \equiv 7 \mod 8\\
    4 & N \equiv 3 \mod 8.
\end{cases}$$
where $\Delta_N$ denotes the discriminant of $\mathbb{Q} (\sqrt{-N} )$. Moreover, we have
\begin{align*}
    h(\Delta_N)  &= \frac{\sqrt{\left| \Delta_N \right|}}{\pi} L(1, \chi_{\left| \Delta_N \right|}) & \text{(Dirichlet's class number formula)}\\
    &\leq \frac{\sqrt{\left| \Delta_N \right|}}{\pi} \cdot \frac{1}{2}\left( \log \left| \Delta_N \right| + 5 - 2\log 6 + \frac{3\pi}{\left| 2\Delta_N \right|} \right) & \text{(\cite[Corollary 1]{ramare2004approximate} with $h = k = 1$)},
\end{align*}
which implies that
$$\left| \dim S_k^{\nw, \pm}(N) - \frac{1}{2} \dim S_k^\nw(N) \right| \leq \frac{1}{2} + \begin{cases}
    \frac{\sqrt{N}}{4\pi} \left( \log (4N) + 5 - 2\log 6 + \frac{3\pi}{8 N} \right) & N \not\equiv 3 \mod 4\\
    \frac{\sqrt{N}}{4\pi} \left( \log N + 5 - 2\log 6 + \frac{3\pi}{2 N} \right) & N \equiv 7 \mod 8\\
    \frac{\sqrt{N}}{2\pi}\left( \log N + 5 - 2\log 6 + \frac{3\pi}{2 N} \right) & N \equiv 3 \mod 8,
\end{cases}.$$
Since we have from \cite[Theorem 1]{martin2005dimensions} that $\dim S_k^\nw(N) \geq \frac{k - 1}{12} \cdot (N - 1) - \frac{7}{6},$ it follows that
\begin{align}
    \dim S_k^{\nw, \pm}(N) &\geq \frac{k - 1}{24} \cdot (N - 1) - \frac{13}{12} \\
    &- \begin{cases}
    \frac{\sqrt{N}}{4\pi} \left( \log (4N) + 5 - 2\log 6 + \frac{3\pi}{8 N} \right) & N \not\equiv 3 \mod 4\\
    \frac{\sqrt{N}}{4\pi} \left( \log N + 5 - 2\log 6 + \frac{3\pi}{2 N} \right) & N \equiv 7 \mod 8\\
    \frac{\sqrt{N}}{2\pi}\left( \log N + 5 - 2\log 6 + \frac{3\pi}{2 N} \right) & N \equiv 3 \mod 8, \end{cases} \label{eqn:dimension-lower-bound-for-prime-level}
\end{align}
Then, for all prime $N > 7861,$ we have
$$\dim S_k^{\nw, \pm}(N) \geq \frac{632}{1151} \cdot \frac{k - 1}{24} \cdot N.$$
In fact, a simple brute-force computation using the true values of $h(\Delta_N)$ rather than the estimate via Dirichlet's class number formula demonstrates that the same bound holds for all prime $N \in [1129, 7861]$. Moreover, $\dim S_k^\pm(N) \geq \dim S_k^{\nw, \pm}(N),$ so the same lower-bound works for $S_k^{\pm}(N)$ as well.

\begin{lemma}\label{lem:explicit-constant-for-big-N-prime}
    Let $p$ be a prime and $I$ a sub-interval of $[-2, 2]$. Let $A = 10923,$ $B = C = 0,$ and $D = 1$.
    \begin{enumerate}
        \item For all prime levels $N > 1129,$
        $$\left| \frac{\# \operatorname{eigen}_{S_k^{\sigma}(N)}(\mathbf{T}'_p) \cap I}{r^\sigma} - \int_{-2}^2 \chi_I(x) \, d \mu_p(x)  \right| \leq F_{X}\!\left( K(X)  \right) \cdot \frac{\log p}{\log kN}.$$
        \item For all prime levels $N > e^X > 55.4278,$
        $$\left| \frac{\# \operatorname{eigen}_{S_k^{\sigma}(N)}(\mathbf{T}'_p) \cap I}{r^\sigma} - \int_{-2}^2 \chi_I(x) \, d \mu_p(x)  \right| \leq F_{X}\!\left( K(X)  \right) \cdot \frac{\log p}{\log kN}.$$
    \end{enumerate}
    The same result holds over $S_k^{\nw, \sigma(N)}.$ 
\end{lemma}

\begin{proof}
    We prove the result over $S_k^{\sigma}(N),$ since
    the proof for $S_k^{\nw, \sigma}(N)$ is exactly the same. As in the proof of Lemma \ref{lem:explicit-constant-for-big-N}, given the restriction $N \geq 1129,$ we can replace the error term $E(m, N)$ by an improved error term
    $$E'(m, N) := 2.64282 \cdot m \sigma_0(m) \cdot 4 \cdot \log (N) \sqrt{N}$$
    We have by an analog of \eqref{eqn:MS-Theorem-19-Analog} for $E'(m, N)$ that for any $c > 0,$
    \begin{align*}
        \left| \frac{\text{eigen}_{S_k^\sigma(N)}(\mathbf{T}_p') \cap I}{r^\sigma} - \int_{-2}^2 \chi_I(x) \, d \mu_p(x)  \right| &\leq \frac{3\sqrt{2}}{4\pi c} \cdot \frac{\log p}{c \log kN} + \frac{111.490 \cdot k^c \cdot N^{1/2 + c} \log N}{r^\sigma}\\
        &\leq \frac{3\sqrt{2}}{4\pi c} \cdot \frac{\log p}{c \log kN} + \frac{111.490 \cdot k^c \cdot N^{1/2 + c} \log N}{\frac{632}{1151} \cdot \frac{k - 1}{24} \cdot N},
    \end{align*}
    for $N > 1129.$ The desired result then follows from Lemma \ref{lem:lambert-w-function}, in the case where $A = 10923,$ $B = 0,$ $C = 0,$ and $D = 1.$ Note that for these choices, assumption \eqref{eqn:criteron-for-F-decreasing} always holds for $X \geq \log(1129).$
\end{proof}

Finally, we establish the validity of the choice $C_0^{\text{prime}} = C_0^{\nw, \text{prime}} = 86.$ The proof for both constants is the same, so we only state it for $C_0^{\text{prime}}.$ Again, we consider three cases.

\textbf{Case 1: $N \geq e^{23.6605}$}

In this case, Lemma \ref{lem:explicit-constant-for-big-N-prime} immediately establishes that $C_0^{\text{prime}}$ can be taken as $85.9989.$

\textbf{Case 2: $N \leq e^{23.6605}$ and $k \leq 466$}

We recall that any choice of $C_0^{\text{prime}} \geq \frac{2}{\log 2} \log(kN)$ is valid. This condition is satisfied for $k$ and $N$ in the specified ranges as long as $C_0^{\text{prime}} \geq 85.9982.$

\textbf{Case 3: $N \leq e^{23.6605}$ and $k \geq 466$}

The cases $N = 2$ and $N = 3$ can be verified separately, so we skip them. For $N \geq 5$, \eqref{eqn:dimension-lower-bound-for-prime-level} gives that $r^\sigma \geq \frac{34}{35} \cdot \frac{k - 1}{24} \cdot (N - 1).$ Then, by \eqref{eqn:MS-Theorem-19-Analog}, it suffices to choose $C_0^{\text{prime}}$ at least
\begin{align}
    &\frac{\log kN}{\log 2} \cdot \frac{3\sqrt{2}}{4\pi c} + \frac{70.9375 \cdot 4 \cdot k^c \cdot N^{1/2 + c} \log N}{\frac{34}{35} \cdot \frac{k - 1}{24} \cdot (N - 1)}\\
    &\leq \frac{\log 466 + \log N}{\log 2} \cdot \frac{3\sqrt{2}}{4\pi c} + \frac{70.9375 \cdot 4 \cdot 466^c \cdot N^{1/2 + c} \log N}{\frac{34}{35} \cdot \frac{465}{24} \cdot (N - 1)}\\
    &\leq \ 23.6606 \qquad  \text{for all $N \leq e^{23.6605}$ by inspection, when we choose $c = \frac{9}{50}$.}
\end{align}
In all these cases, $C_0^{\text{prime}} = 86$ is a valid choice.

\begin{remark}
    Using the fact that $\omega(N) = 1$ and $\sigma_0(N) = O(\log(N))$ in the case where $N$ is a prime power, one can also establish improvements in this case. In \S \ref{subsection:large-hecke-fields}, we establish a bound on values of $N$ for which $J_0^\sigma(N)$ has $\mathbb{Q}$-simple factors of dimension at most $d.$ Finding an absolute constant $C_0^{\text{prime power}}$ over prime powers would allow one to find better bounds on exact prime divisors of $N$ (see Murty and Sinha \cite[Theorem 32]{MS07} for the analogous result for $J_0(N)$). 
\end{remark}

\begin{remark}
    Martin \cite{martin2005dimensions} gives dimension formulas for $S_k^{\nw, \sigma}(N)$ in the case where $N$ is square-free, which could be used to improve Lemma \ref{lem:explicit-constant-for-big-N-newspace} for square-free levels.
\end{remark}

\section{Applications}
\label{section:applications}

\subsection{Vertical Atkin--Serre}\label{subsection:vertical-Atkin--Serre}

We recall the Atkin--Serre conjecture, originally posed by Serre in \cite{Ser81}.

{
\renewcommand{\thetheorem}{(Atkin--Serre)}
\begin{conjecture}
    Fix a non-CM newform $f \in S_k^\nw(N)$ of weight $k \geq 4.$ Then
    $\left| a_f(p) \right| \gg_\epsilon p^{(k - 3)/2 - \epsilon}$.
\end{conjecture}
\addtocounter{theorem}{-1}
}

In \cite[Theorem 2.2]{Kim24}, Kim proved an analogous vertical result, which he called ``Vertical Atkin--Serre". We use Theorem \ref{thm:limit-constant} to generalize Kim's result to the setting of Atkin--Lehner subspaces.
{
\renewcommand{\thetheorem}{\ref{prop:vertical-Atkin--Serre}}
\begin{proposition}
    Let $p$ be a prime coprime to $N$. Then,
    $$\frac{\#\left\{ \lambda \in \operatorname{eigen}_{S_k^\sigma(N)} \! \left( \mathbf{T}_{p}'  \right) : \left| \lambda \right| \leq \frac{\log p}{\log kN} \right\}}{\dim S_k^\sigma(N)} \leq \left(C_0 + \frac{3\sqrt{2}}{2\pi} \right) \cdot  \frac{\log p}{\log  kN}$$
    The analogous statement also holds for $S_k^{\nw,\sigma}(N)$.
\end{proposition}
\addtocounter{theorem}{-1}
}
\begin{proof}
    We just provide the proof for $S_k^{\sigma}(N)$, since the proof for the newspace is identical. We adapt the proof of \cite[Theorem 2.2]{Kim24}, with additional attention being paid to the calculation of the coefficients in the upper bound. Recall that $d\mu_p(x) \leq \frac{3\sqrt{2}}{4\pi} dx$, (or $\frac{1}{\pi} dx$ for $p \geq 7$). Applying Theorem \ref{thm:main-theorem} to the interval $I = \left[ - \frac{\log p}{\log kN} , \frac{\log p}{\log kN}  \right]$, we have that
    $$\left| \frac{ \# \operatorname{eigen}_{S_k^\sigma(N)}(\mathbf{T}_p') \cap I}{r^\sigma} \right| \leq \left| \int_{-2}^2 \chi_I(x) \,  d \mu_p(x)  \right| + C_0   \frac{\log p}{\log  kN}  \leq  2 \cdot \frac{\log p}{\log kN} \cdot \frac{3\sqrt{2}}{4\pi} + C_0  \frac{\log p}{\log  kN},$$
    yielding the desired result.
\end{proof}

We compare Proposition \ref{prop:vertical-Atkin--Serre} and the Atkin--Serre conjecture.

\begin{enumerate}
        \item The Atkin--Serre conjecture regulates the frequency for which normalized Hecke eigenvalues are less than $O_\epsilon\!\left( p^{-1 - \epsilon} \right)$. Our result is a ``vertical" analogue inasmuch as it regulates when these eigenvalues are less than $\frac{\log p}{\log kN}$, dependent on $k$ and $N.$
         \item Our result holds both for the newspace $S_k^{\nw, \sigma}(N)$ and the whole space $S_k^\sigma(N).$ Furthermore, just as in Corollary \ref{cor:main-theorem-for-larger-spaces}, these results also hold for the Fricke spaces $S_k^\pm(N)$ and $S_k^{\nw, \pm}(N),$ as well as for $S_k(N)$ and $S_k^\nw(N).$
        \item Our result holds for all $k \geq 2$, while horizontal Atkin--Serre only breaks down at $k = 2,$ as seen by Elkies in \cite{Elk87}.
    \end{enumerate}

\subsection{Extremal Primes}
\label{subsection:extremal-primes}

As a preliminary, we note that under the substitution $x = 2\cos \theta$, $\mu_p$ can be reinterpreted as a measure on the space $[0, \pi]$ given by
\begin{equation}
    d\mu_p(\theta) = \frac{2(p + 1)}{\pi} \frac{\sin^2 \theta}{\left( p^{1/2} + p^{-1/2} \right)^2 - 4 \cos^2 \theta} \, d \theta. \label{eqn:def-mu-p-theta}
\end{equation}
One can easily show that for each prime $p,$ this measure satisfies
\begin{equation}
        d\mu_p (\theta) \leq \frac{p(p + 1)}{(p - 1)^2} \cdot \frac{2}{\pi} \cdot \sin^2(\theta) \, d\theta \leq \frac{p(p + 1)}{(p - 1)^2} \cdot \frac{2}{\pi} \cdot \theta^2 \, d\theta. \label{eqn:upper-bound-for-mu-p-in-terms-of-theta}
    \end{equation}

We now recall the definition of an extremal prime.

{
\renewcommand{\thetheorem}{\ref{def:extremal-prime}}
\begin{definition}
    Let $f \in H_k (N)$ be of weight $k \geq 4.$ We say that $p$ is an \textit{extremal prime} for $f$ if
    $\left| a_f(p)  \right| \geq \left\lfloor 2 p^{(k - 1)/2} \right\rfloor.$
\end{definition}
\addtocounter{theorem}{-1}
}

The results \cite[Theorems 2.3 and 2.4]{Kim24} concern the sparsity of Hecke eigenforms $f \in S_k(N)$ for which a prime $p$ is extremal.

In Propositions \ref{prop:k-dependent-extremality} and \ref{prop:N-dependent-extremality}, we generalize these results to the setting of the Atkin--Lehner spaces $S_k^\sigma(N)$ and $S_k^{\nw, \sigma}(N).$ Additionally, by Corollary \ref{cor:main-theorem-for-larger-spaces}, our results also hold over $S_k^{\nw}(N)$.

\begin{remark}
    Throughout this section, we make use of the fact that $C_0 \geq \frac{3\sqrt{2}}{2\pi}$ and $C_0^\nw \geq \frac{3\sqrt{2}}{2\pi}.$ This can be seen easily, for example by considering the case where $N = 1, k = 12,$ and $p = 2,$ with $\sigma$ being the unique Atkin--Lehner sign pattern for $N = 1$. 
\end{remark}

\begin{proposition} \label{prop:N-dependent-extremality}
    Fix a constant $A > 1$. Then for $k$, $N \geq 2$,  $p \nmid N$ such that
    $$\frac{4}{3(k - 1)} \log\log  kN \leq \log p \leq \frac{4A}{3(k - 1)} \log\log  kN,$$ 
    we have
    \begin{align}
        \frac{\#\left\{ \lambda \in \operatorname{eigen}_{ S_k^\sigma(N)} \! \left( \mathbf{T}_{p}  \right) : |\lambda| \geq \left\lfloor 2 p^{(k - 1)/2}  \right\rfloor \right\}}{\dim S_k^\sigma(N)} \leq \frac{8A C_0}{3} \cdot \frac{\log \log  kN + \sqrt{2}}{\log  kN}
    \end{align}
    The same result also holds for $S_k^{\nw, \sigma}(N)$, with $C_0^\nw$ in place of $C_0$.
\end{proposition}

\begin{proof}
    We prove the statement for $S_k^\sigma$, noting that the proof for $S_k^{\nw, \sigma}$ is analogous. We mainly follow Kim's proof of \cite[Theorem 2.3]{Kim24} while emphasizing the calculation of explicit constants. It suffices to demonstrate the two inequalities
    \begin{align}
        \frac{\#\left\{ \lambda \in \operatorname{eigen}_{ S_k^\sigma(N)} \! \left( \mathbf{T}_{p}  \right) : \lambda \geq \left\lfloor 2 p^{(k - 1)/2}  \right\rfloor \right\}}{\dim S_k^\sigma(N)} &\leq \frac{4A C_0}{3} \cdot \frac{\log \log  kN + \sqrt{2}}{\log  kN}, \label{eqn:temp eqn in-prop:n-dependent-extremality}\\
        \frac{\#\left\{ \lambda \in \operatorname{eigen}_{ S_k^\sigma(N)} \! \left( \mathbf{T}_{p}  \right) : \lambda \leq -\left\lfloor 2 p^{(k - 1)/2}  \right\rfloor \right\}}{\dim S_k^\sigma(N)} &\leq \frac{4A C_0}{3} \cdot \frac{\log \log  kN + \sqrt{2}}{\log  kN}.
    \end{align}
    
    We prove the first inequality, since the proof of the second inequality is analogous. As before, let $\theta_\lambda \in [0, \pi]$ be such that $2 \cos \theta_\lambda = \frac{\lambda}{p^{(k - 1)/2}}.$ We note that the condition $\lambda \geq \left\lfloor 2 p^{(k - 1)/2} \right\rfloor$ guarantees that
    \begin{equation}
        \cos \theta_\lambda \geq 1 - \frac{\{ 2 p^{(k - 1)/2}  \}}{2 p^{(k - 1)/2}} \geq 1 - \frac{1}{2 p^{(k - 1)/2}}, \label{eqn:extremal-cosine-bound}
    \end{equation}
    which implies that $\theta \leq p^{-(k - 1)/4}.$
    By \eqref{eqn:upper-bound-for-mu-p-in-terms-of-theta}, we have that $d\mu_p(\theta) \leq \frac{12}{\pi} \theta^2 d \theta,$ so Theorem \ref{thm:main-theorem} gives that the left hand side of \eqref{eqn:temp eqn in-prop:n-dependent-extremality} has an upper bound of
    $$\int_0^{p^{-(k - 1)/4}} \frac{12}{\pi} \cdot  \theta^2 d\theta + C_0 \frac{\log p}{\log  kN} \leq \frac{4}{\pi} \cdot  p^{-3(k - 1)/4} +  C_0 \frac{\log p}{\log  kN}.$$
    Our bounds on $p$ give that $p^{-3(k - 1)/4} \leq \frac{1}{\log kN }$. Therefore, our upper bound becomes
    \begin{align}
        \frac{4}{\pi} \cdot \frac{1}{\log  kN} +  \frac{4A C_0}{3} \cdot \frac{\log \log  kN}{(k - 1)\log  kN} &\leq \frac{4 A C_0}{3} \cdot \frac{3}{\pi \cdot A C_0} \frac{1}{\log  kN} +  \frac{4A C_0}{3} \cdot \frac{\log \log  kN}{\log  kN}\\
        &\leq \frac{4A C_0}{3} \cdot \frac{\log \log  kN + \sqrt{2}}{\log  kN},
    \end{align}
    since $A C_0 > \frac{3\sqrt{2}}{2\pi}.$ This completes the proof.
\end{proof}

A stronger result is given in Proposition \ref{prop:k-dependent-extremality}, although the proof requires considerably more effort.

{
\renewcommand{\thetheorem}{\ref{prop:k-dependent-extremality}}
\begin{proposition} 
    Let $p$ be a prime coprime to $N$ satisfying $\log p \geq \frac{4\log \log  kN}{3(k - 1)}.$ Then if $kN \geq e^8,$
    $$\frac{\#\left\{ \lambda \in \operatorname{eigen}_{ S_k^\sigma(N)} \! \left( \mathbf{T}_{p}  \right) :  |\lambda|  \geq \left\lfloor 2 p^{(k - 1)/2}  \right\rfloor \right\}}{\dim S_k^\sigma(N)} \leq 83.5853 \cdot \, C_0 \left( \frac{p + 1}{p - 1} \right)^2  \cdot \frac{ \log p \,  \log \left(  \log kN + 1000 \right)^2}{\log kN}.$$
    The same result also holds for $S_k^{\nw, \sigma}(N)$, with $C_0^\nw$ in place of $C_0.$
\end{proposition}
\addtocounter{theorem}{-1}
}

In particular, Proposition \ref{prop:k-dependent-extremality} implies that for fixed $p$ and $N$ coprime, the proportion of newforms in $S_k^{\nw, \sigma}$ for which $p$ is extremal tends to $0$ as $k \rightarrow \infty.$

The proof of Proposition \ref{prop:k-dependent-extremality}, which follows the proof of \cite[Theorem 2.4]{Kim24}, relies on approximating the characteristic function $\chi_I$ with trigonometric polynomials. First, we define the polynomials
\begin{equation}
    R_n(x) : [-1, 1] \rightarrow \mathbb{R}, \qquad R_n(x):= \frac{p - 1}{p} U_n(x) + \frac{2}{p} T_n(x). \label{eqn:definition-of-Rn(x)}
\end{equation}
Kim observed that the family $\{ R_n(x/2) \}_n$ is orthogonal with respect to $\mu_p$ over $[-2, 2]$. Now, let $M \in \mathbb{N}$ be at least $3,$ and let $I = \left[ 0, \frac{1}{M} \right] \subseteq \left[ 0, \pi \right].$ By \cite[Theorem 2.5]{Kim24}, there exist certain  trigonometric polynomials
$$F_{I, M}^+(\theta) = \sum_{n = 0}^M \widehat{F}_{I, M}^+(n) U_n(\cos \theta)$$
such that $\chi_I(\theta) \leq F_{I, M}^+(\theta)$ for all $\theta \in [0, \pi].$ Importantly, $F_{I, M}^+(\theta)$ can be represented as a linear combination of the polynomials $\{R_n(\cos \theta)\}_{n = 0}^M$, where the coefficients for each $R_n(\cos \theta)$ are very small. This is made precise in the following proposition, whose proof we defer to Appendix \ref{subsection:integral-bounds}.

\begin{proposition}\label{prop:kim-2-6-analog}
    For $M \in \mathbb{N}$ at least $3$ and $I = \left[0, \frac{1}{M} \right] \subseteq \left[0, 2\pi \right],$
    $$F_{I, M}^+(\theta) = \sum_{n = 0}^M b_n R_{n}(\cos \theta)$$
    for certain $b_0, \cdots, b_M$ satisfying
    \begin{align}
        |b_n| &\leq \left( \frac{p + 1}{p - 1} \right)^2 \cdot \frac{44.4751 \log M + 46.5666}{M^2} \qquad \text{for $n \geq 1$}, \\
        |b_0| &\leq \left( \frac{p + 1}{p - 1} \right)^2 \cdot \frac{314.255 \log M + 28.0356}{M^3}.
    \end{align}
\end{proposition}

We are now ready to prove Proposition \ref{prop:k-dependent-extremality}.

\begin{proof}[Proof of Proposition \ref{prop:k-dependent-extremality}]
    As in the proof of Proposition \ref{prop:N-dependent-extremality}, it suffices to prove both of the inequalities
    \begin{align}
        \frac{\#\left\{ \lambda \in \operatorname{eigen}_{ S_k^\sigma(N)} \! \left( \mathbf{T}_{p}  \right) :  \lambda  \geq \left\lfloor 2 p^{(k - 1)/2}  \right\rfloor \right\}}{\dim S_k^\sigma(N)} &\leq \ 41.79265  \, C_0 \left( \frac{p + 1}{p - 1} \right)^2  \cdot \frac{ \log p  \left(  \log kN + 1000 \right)^2}{\log kN}\\
        \frac{\#\left\{ \lambda \in \operatorname{eigen}_{ S_k^\sigma(N)} \! \left( \mathbf{T}_{p}  \right) :  \lambda  \leq - \left\lfloor 2 p^{(k - 1)/2}  \right\rfloor \right\}}{\dim S_k^\sigma(N)} &\leq \ 41.79265 \, C_0 \left( \frac{p + 1}{p - 1} \right)^2  \cdot \frac{ \log p  \left(  \log kN + 1000 \right)^2}{\log kN}.
    \end{align}
    Again, we only prove the first inequality; the proof of the second is analogous. For $M \in \mathbb{N}$ at least $3$, let $I := \left[0, \frac{1}{M} \right]$ and note that $\cos(I) = \left[ \cos \left( \frac{1}{M} \right), 1  \right]$. Suppose $M \leq p^{(k - 1)/4}$ so that, as in \eqref{eqn:extremal-cosine-bound}, we have
    $$\left[\left\lfloor 2 p^{(k - 1)/2}  \right\rfloor,  2 p^{(k - 1)/2}   \right] \subseteq 2p^{(k - 1)/2} \cdot \cos(I).$$
    Then we have
    \begin{align*}
        & \#\left\{ \lambda \in \operatorname{eigen}_{ S_k^\sigma(N)} \! \left( \mathbf{T}_{p}  \right) :  \lambda \geq \left\lfloor 2 p^{(k - 1)/2}  \right\rfloor \right\} \\
        &\leq \sum_{\lambda \in \operatorname{eigen}_{ S_k^\sigma(N)}} \chi_{\cos(I)}(\cos \theta_\lambda)\\
        &= \sum_{\lambda \in \operatorname{eigen}_{ S_k^\sigma(N)}} \chi_{I}(\theta_\lambda) \\
        &\leq \sum_{\lambda \in \operatorname{eigen}_{ S_k^\sigma(N)}} F_{I, M}^+(\theta_\lambda)\\
        &= \sum_{\lambda \in \operatorname{eigen}_{ S_k^\sigma(N)}} \sum_{n = 0}^M b_{n} R_n(\cos \theta_\lambda) & \text{(by Proposition \ref{prop:kim-2-6-analog})}\\
        &\leq \frac{314.255 \log M + 28.0356}{M^3} \cdot \left( \frac{p + 1}{p - 1} \right)^2 \cdot r^\sigma 
        & \text{(by Proposition \ref{prop:kim-2-6-analog})}
        \\
        &\qquad + \frac{44.4751 \log M + 46.5666}{M^2}  \cdot \left( \frac{p + 1}{p - 1} \right)^2 \sum_{n = 1}^M  \left| \sum_{\lambda \in \operatorname{eigen}_{ S_k^\sigma(N)}} R_n(\cos \theta_\lambda) \right|.
    \end{align*}
    
    Now, observe that orthogonality of $\{R_n(x/2)\}_n$ implies that $$\int_{-2}^2 R_n(x/2) \, d\mu_p(x) = \int_{-2}^2 R_n(x/2) \cdot \frac{p}{p + 1}R_0(x/2) \, d\mu_p(x) = 0$$ 
    for all $n \geq 1$. Hence we have that
    \begin{align}
        \left| \sum_{\lambda \in \operatorname{eigen}_{ S_k^\sigma(N)}(\mathbf{T}_p')} R_n(\cos \theta_\lambda) \right| 
        &=
        \left| \sum_{\lambda \in \operatorname{eigen}_{ S_k^\sigma(N)}(\mathbf{T}_p')} R_n(\lambda/2)  - \int_{-2}^2 R_n(x/2) \, d\mu_p(x)\right| \\
        &\leq C_0 \frac{\log p}{\log kN} \cdot r^\sigma \cdot \delta(R_n) \qquad \text{(by \eqref{cor:main-thm-for-general-function-g})}.
    \end{align}
    Additionally, by the triangle inequality and Lemma \ref{lem:variation-of-U}, we have
    \begin{align*}
        \sum_{n = 1}^M \delta(R_n) 
        &\leq \sum_{n = 1}^M
        \lrp{\frac{p - 1}{p} \cdot \delta(U_n) + \frac{2}{p} \cdot \delta(T_n)}\\
        &\leq \sum_{n = 1}^M 
        \lrp{\frac{p - 1}{p} \cdot 2(n + 1) \left( \log n + 2 - \log 2 \right) + \frac{2}{p} \cdot 2n}\\
        &\leq \sum_{n = 1}^M 2  (n + 1) \left( \log n + 2 - \log 2 \right)\\
        &\leq   2   \left( \log M + 2 - \log 2 \right) \sum_{n = 1}^M  (n + 1)
        \\
        &=  \left( \log M + 2 - \log 2\right) \left( M^2 + 3M \right).
    \end{align*}
    Combining all of these estimates, we obtain
    \begin{align}
        & \#\left\{ \lambda \in \operatorname{eigen}_{ S_k^\sigma(N)} \! \left( \mathbf{T}_{p}  \right) :  \lambda  \geq \left\lfloor 2 p^{(k - 1)/2}  \right\rfloor \right\}\\
        &\leq \frac{314.255 \log M + 28.0356}{M^3} \cdot \left( \frac{p + 1}{p - 1} \right)^2 \cdot r^\sigma\\
        &\quad +  \frac{44.4751 \log M + 46.5666}{M^2}  \cdot \left( \frac{p + 1}{p - 1} \right)^2   \sum_{n = 1}^M  \left| C_0 \frac{\log p}{\log kN} \cdot r^\sigma \cdot \delta(R_n) \right|\\
        &\leq \frac{314.255 \log M + 28.0356}{M^3} \cdot \left( \frac{p + 1}{p - 1} \right)^2 \cdot r^\sigma\\
        & \quad +  \frac{44.4751 \log M + 46.5666}{M^2}  \cdot \left( \frac{p + 1}{p - 1} \right)^2  \cdot C_0 \frac{\log p}{\log kN} \cdot r^\sigma \cdot \sum_{n = 1}^M \delta(R_n).\\
        &\leq \left( \frac{p + 1}{p - 1}\right)^2 \cdot r^\sigma \cdot S,
    \end{align}
    where 
    $$S := \frac{314.255 \log M + 28.0356}{M^3} +  \frac{44.4751 \log M + 46.5666}{M^2} \cdot C_0 \frac{\log p}{\log kN} \cdot \left( \log M + 2 - \log 2\right) \left( M^2 + 3M \right) .$$
    We choose the value $M = \left\lceil (\log kN)^{1/3} \right\rceil.$ Note that the upper bound $kN > e^8$ ensures that $M \geq 3,$ while the assumption $\frac{4 \log \log kN}{3(k - 1)} \leq \log p$ ensures that $M \leq p^{(k - 1)/4}.$ This value of $M$ then yields
    \begin{align*}
        S &\leq \frac{\log p}{\log kN} \cdot  \frac{314.255 \log \left( (\log kN)^{1/3} + 1  \right) + 28.0356}{\log 2}  \\
        & \quad + C_0 \frac{\log p}{\log kN} \left(  44.4751 \log \left( (\log kN)^{1/3} + 1  \right) + 46.5666  \right)\\ &\quad \cdot\left( \log \left( (\log kN)^{1/3} + 1  \right) + 2 - \log 2\right) \left( 1 + \frac{3}{(\log kN)^{1/3}} \right).\\
        &\leq \frac{\log p}{\log kN} \cdot  \frac{(314.255 \log \left( (\log kN)^{1/3} + 1  \right) + 28.0356) \cdot C_0/\frac{3\sqrt{2}}{2\pi}}{\log 2} \\
        & \qquad + C_0 \frac{\log p}{\log kN}\left(  44.4751 \log \left( (\log kN)^{1/3} + 1  \right) + 46.5666  \right)\\
        &\qquad \cdot \left( \log \left( (\log kN)^{1/3} + 1  \right) + 2 - \log 2\right) \left( 1 + \frac{3}{(\log kN)^{1/3}} \right) \\
        &=: S'.
    \end{align*}
    Now, the expression
    $$ \frac{S'}{C_0 \log p} \cdot \frac{\log(kN)}{\log\left( \log kN + 1000  \right)^2},$$
    when considered as a function of $kN,$ is bounded and attains a maximum value of $\ 41.79262...$ at $kN = e^{742.77987...}.$ Thus, we conclude that
    $$\#\left\{ \lambda \in \operatorname{eigen}_{ S_k^\sigma(N)} \! \left( \mathbf{T}_{p}  \right) :  \lambda  \geq \left\lfloor 2 p^{(k - 1)/2}  \right\rfloor \right\} \leq \ 41.79263 C_0 \left( \frac{p + 1}{p - 1} \right)^2 \cdot r^\sigma \cdot \frac{ \log p \cdot \log  \left(  \log kN + 1000 \right)^2}{\log kN},$$
    completing the proof.
\end{proof}

\begin{remark}
    The proof of \cite[Theorem 2.4]{Kim24} differs from our proof of Proposition \ref{prop:k-dependent-extremality} in a number of ways. In particular, it employs the inequality
    $$\left| \sum_{i = 1}^{r^\sigma} R_n(\cos \theta_i) \right| \ll r^\sigma  \frac{\log p}{\log kN}$$
    and then claims that
    $$\sum_{n = 1}^{M} \left| \sum_{i = 1}^{r^\sigma} R_n(\cos \theta_i)  \right| \ll M \cdot r^\sigma \frac{\log p}{\log kN}.$$
    It is not obvious to us how such a bound can be obtained, given that the implied constant in the first inequality seems to depend on $\delta(R_n)$. Without this bound, it is not clear to us how the methods of \cite{Kim24} can yield the final stated result \cite[Theorem 2.4]{Kim24}.
\end{remark}

\begin{remark}
    The coefficient $83.5853$ can be improved to any number greater than $44.4751 \cdot \frac{2}{9} = 9.88336$, at the expense of replacing the constant $1000$ with something larger. Conversely, the constant $1000$ can be lowered all the way to zero, though this will increase the size of the coefficient. The former option is superior for large $kN$ and the latter option is superior for small $kN.$
\end{remark}

\subsection{Estimates for the point counts \texorpdfstring{$\#J_0^{\nw,\sigma}(N)(\mathbb{F}_{p^n})$}{J-0-sigma(N)(F-p-n)}} \label{subsection:point-count}

In the following, we let $J_0(N)$ denote the Jacobian of the modular curve $X_0(N)$ and $J_0^{\nw}(N)$ denote its newpart. Note that in the case where $N$ is prime, $J_0^{\nw}(N) = J_0 (N)$.  

Recall that Galois orbits of the normalized Hecke eigenforms $f \in S_2^{\mathrm{new}}(N)$ correspond to abelian varieties of $\mathrm{GL}_2$-type. If $f \in S_2^{\mathrm{new}}(N)$ is a newform, the corresponding abelian variety is $A_f$, whose $p$-Weil polynomial is
$
\prod_{\tau : K_f \hookrightarrow \mathbb{C}}
\left(x^2-\tau(a_p(f))x+p\right),
$
where $\tau : K_f \to \mathbb{C}$ ranges over all embeddings. One has the decomposition
\begin{align}\label{eqn:composite-jacobian-decomp}
J_0(N)
\sim_{\mathbb Q}
\bigoplus_{M\mid N}
\bigoplus_{f}
A_f^{\sigma_0(\frac{N}{M})},
\end{align}
where $f$ runs over Galois orbits of newforms in $S_2^{\mathrm{new}}(M),$ see \cite{MurtySinhaDecomp} .
Each $A_f$ is a $\mathbb{Q}$-simple abelian variety satisfying
$\operatorname{Conductor}(A_f) = M^{[K_f:\mathbb{Q}]},  \dim A_f=[K_f:\mathbb{Q}], \text{ and (if $f$ is non-CM) }\operatorname{End}^0(A_f)\cong K_f$
by \cite{Ribet1980Twists}. 

Note that the Atkin--Lehner involutions act naturally on $X_0(N)$ and thus also act on $J_0(N)$ by functoriality, so we have a decomposition $J_0 (N) \sim_{\mathbb{Q}} \prod_{\sigma} J_0^{\sigma}(N)$ We restrict our attention to the newpart of $J_0^{\sigma}(N),$ for which we have the decomposition $J_0^{\nw,\sigma}(N) \sim \prod_{f_i} A_{f_i},$ where the $f_i$ are the newforms in $S_2^{\nw, \sigma}(N).$ Then, the $p$-Weil polynomial of $J_0^{\nw, \sigma}(N)$ is given by $\prod_{f_i} (x^2 - a_p(f_i)x + p)$, where $f_i$ are newforms in $S_2^{\text{new}, \sigma}(N)$.  Hence, by the Grothendieck-Lefschetz theorem, we have $\#J_0^{\nw,\sigma}(N)(\mathbb{F}_p) = \prod_{f_i} (p+1 - a_p(f_i)).$  We now give asymptotics for the point counts $\#J_0^{\nw,\sigma}(N)(\mathbb{F}_{p^n})$ over arbitrary finite extensions $\mathbb F_{p^n}$ of $\mathbb F_p$.

\begin{proposition}
\label{prop:asymptotic-point-count-Jacobian-q}
Fix a prime power $p^n$, and consider $N$ coprime to $p$. Then
\[
\frac{\log \#J_0^{\nw,\sigma}(N)(\mathbb F_{p^n})}
     {\dim S_2^{\nw, \sigma}(N)}
\longrightarrow
\begin{cases}
\displaystyle n\log p - \frac{p-1}{2}\log\!\left(1-p^{-2n}\right),
& \text{$n$ odd},\\[10pt]
\displaystyle n\log p - (p-1)\log\!\left(1- p^{-n}\right),
& \text{$n$ even},
\end{cases}
\qquad\text{as }N\to\infty.
\]
with an error bound of $C_0^{\nw}
\frac{\log p}
     {\log(2N)} \cdot \frac{4n}{p^{n/2} - 1}.$ The analogous result also holds for $J_0^{\nw}(N)$.
\end{proposition}

\begin{proof}
Let $\{\alpha_i,\overline{\alpha_i}\}_{i=1}^r$, where $r^{\nw, \sigma}= \dim S_2^{\nw, \sigma}(N) =\dim J_0^{\nw,\sigma}(N)$, denote the roots of the polynomials
$\{x^2-a_p(f_i)x+p\}_{i=1}^{r^{\nw, \sigma}}$. These are precisely the $p$-Frobenius eigenvalues of $J_0^{\nw,\sigma}(N)$, and so the $p$-Weil polynomial is 
$$
P_{J_0^{\nw,\sigma}(N),p}(x)
=
\prod_{i=1}^{r^{\nw, \sigma}} (x-\alpha_i)(x-\overline{\alpha_i})
=
\prod_{i=1}^{r^{\nw, \sigma}} (x^2-a_p(f_i)x+p).$$ Now, the $p^n$-Weil polynomial is
\[
P_{J_0^{\nw,\sigma}(N),p^n}(x)
=
\prod_{i=1}^{r^{\nw, \sigma}} (x-\alpha_i^n)(x-{\overline{\alpha_i}}^n).
\]
Hence,
\begin{align}
    \#J_0^{\nw,\sigma}(N)(\mathbb F_{p^n}) &= P_{J_0^{\nw,\sigma}(N),p^n}(1)\\
    &= \prod_{i=1}^{r^{\nw, \sigma}}
\bigl(p^n+1-(\alpha_i^n+\overline{\alpha_i}^n)\bigr)\\
&= \prod_{i=1}^{r^{\nw, \sigma}}
\left(p^n+1-  2p^{n/2}  \cdot T_n \left( \frac{1}{2} \cdot \frac{a_p(f_i)}{\sqrt{p}}  \right) \right),
\end{align}
by the recursion $(\alpha_i^{n+2} + \overline{\alpha_i}^{n+2}) =
(\alpha_i^{1} + \overline{\alpha_i}^{1})\,(\alpha_i^{n+1} + \overline{\alpha_i}^{n+1}) - p\,(\alpha_i^{n} + \overline{\alpha_i}^{n})$ . Taking logs, we obtain
\[
\log \#J_0^{\nw,\sigma}(N)(\mathbb F_{p^n})
=
\sum_i
g_n\!\left(\frac{a_p(f_i)}{\sqrt p}\right),
\]
for $g_n(x) := \log \left( p^n + 1 - 2p^{n/2} \cdot T_n \!\left(  \frac{x}{2} \right)  \right).$ Applying Corollary \ref{cor:main-thm-for-general-function-g} to the test function $g_n$ yields
\[
\left| \frac{\log \#J_0^{\nw,\sigma}(N)(\mathbb F_{p^n})}
     {r^{\sigma,\nw}}
- 
\int_{-2}^{2}
g_n(x)\,d\mu_p(x) \right| \leq \delta(g_n) \cdot C_0^{\nw} \cdot \frac{\log p}{\log 2N}.
\]
The value of the integral is given by Lemma \ref{lem:g-integral}. Meanwhile, since $T_n\!\left( \frac{x}{2} \right)$ increases or decreases monotonically between $-1$ and $1$  (over $ x \in [-2, 2]$ ) a total of $n$ times, we can bound
    \begin{align*}
        \delta(g_n) &= n \cdot \left( \log\left(p^n + 1 + 2p^{n/2} \right) - \log\left(p^n + 1 - 2p^{n/2}\right) \right) \\
        &= 2n \cdot \log\left(\frac{p^{n/2} + 1}{p^{n/2} - 1} \right) = 2n \cdot \log\left(1 + \frac{2}{p^{n/2} - 1} \right) \leq \frac{4n}{p^{n/2}-1},
    \end{align*}
yielding the desired result.
\end{proof}

The analogous result also holds for the full Jacobian, $J_0(N).$

\begin{proposition}\label{prop:asymptotic-point-count-Jacobianreg-q}
    Fix a prime power $p^n$, and consider $N$ coprime to $p$. Then
\[
\frac{\log \#J_0(N)(\mathbb F_{p^n})}
     {\dim S_2(N)}
\longrightarrow
\begin{cases}
\displaystyle n\log p - \frac{p-1}{2}\log\!\left(1-p^{-2n}\right)
& \text{$n$ odd}\\[10pt]
\displaystyle n\log p - (p-1)\log\!\left(1- p^{-n}\right)
& \text{$n$ even}
\end{cases}
\qquad\text{as }N\to\infty,
\]
with an error bound of $C_0
\frac{\log p}
     {\log(2N)} \cdot \frac{4n}{p^{n/2} - 1}.$
\end{proposition}

\begin{proof}
Recall that
\[
J_0(N)\sim_{\mathbb{Q}}
\prod_{M\mid N}
\prod_{f\in S_2^{\mathrm{new}}(M)/G_{\mathbb{Q}}}
A_f^{\,\sigma_0(N/M)},
\]
where $A_f$ is the simple abelian variety attached to the Galois orbit of $f$. Consequently, if $p\nmid N$, the $p$-Weil polynomial of $J_0(N)$ is
\[
\prod_{M\mid N}
\prod_{f\in S_2^{\mathrm{new}}(M)/G_{\mathbb{Q}}}
\left(x^2-a_p(f)x+p\right)^{\sigma_0(N/M)}.
\]
Hence, if $\lambda_f,\overline{\lambda}_f$ denote the roots of $x^2-a_p(f)x+p,$ then we have the equivalence of multisets
\begin{align}
    \left\{ \text{$p$-Frobenius eigenvalues of $J_0(N)$} \right\} = \bigsqcup_{M|N} \bigsqcup_{i = 1}^{\sigma_0(N/M)} \bigsqcup_{f\in S_2^{\mathrm{new}}(M)/G_{\mathbb{Q}}} \left\{ \lambda_f, \overline{\lambda_{f}} \right\}.
\end{align}
Meanwhile, we have the equivalence of multisets
\begin{align}
    \eigen_{S_2(N)}(\mathbf{T}_p) = \bigsqcup_{M \mid N} \bigsqcup_{i = 1}^{\sigma_0(N/M)} \eigen_{S_2^\nw(M)}(\mathbf{T}_p),
\end{align}
by the oldform decomposition $S_2(N) = \bigoplus_{M\mid N}
\bigoplus_{d\mid N/M} V_d\!\left(S_2^{\mathrm{new}}(M)\right)$ and the identity $\mathbf{T}_p V_d = V_d \mathbf{T}_p.$ Therefore, we have
\[
\#J_0(N)(\mathbb{F}_p)
=
\prod_{\lambda \in \operatorname{eigen}_{S_2(N)}(\mathbf{T}_p)}
(p+1-\lambda).
\]
The rest of the proof follows identically to that of Proposition \ref{prop:asymptotic-point-count-Jacobian-q}.
\end{proof}

Ihara \cite{Ihara1974} and Hashimoto \cite{Hashimoto1981} showed that for prime level $q$, we have $\#J_0(q)(\mathbb F_{p}) = h(G_q(p)) (1+g(X_0(q)),$ where $h(G_q(p))$ denotes the class number of the supersingular $p$-isogeny graph of elliptic curves over $\overline{\mathbb{F}}_q$. Hence, restricting the above to prime level $N=q,$ we have the following.

\begin{corollary}
Let $p$ be fixed, and let $q\neq p$ vary over the primes. Let $h(G_q(p))$ denote the class number of the supersingular $p$-isogeny graph of elliptic curves over $\overline{\mathbb{F}}_q$. Then
\[
\frac{\log h(G_q(p))}
{\dim S_2(q)}
\longrightarrow
\log p-\frac{p-1}{2}\log\!\left(1-p^{-2} \right)  \qquad \text{as $q \rightarrow \infty$}.
\]
\end{corollary}

We now use this approach to derive asymptotic averages for the point counts of quotients of modular curves by Atkin--Lehner subgroups. 

\begin{proposition}\label{thm:asymptotic-point-count-W-quotient}
    Fix a prime power $p^n$, and consider $N$ coprime to $p.$ Let $W \leq W(N)$ be a subgroup of the full Atkin--Lehner group. Then
    \begin{align}
        \frac{\#X_0^W(N)(\mathbb F_{p^n})} 
{g(X_0^W(N))} \longrightarrow \begin{cases}
p-1 & \text{$n$ even}\\
0 & \text{$n$ odd}
\end{cases} \qquad \text{as $N \rightarrow \infty$},
    \end{align}
    with error bound $2 p^{n/2} \cdot C_0 \frac{2n \cdot \log p}{\log 2N} + \frac{p^n + 1}{g(X_0^W(N))}.$
\end{proposition}

\begin{proof}
    We define
    $S_2^W(N):=\{f\in S_2(N): wf=f\text{ for all }w\in W\},$
    and let $r^W := \dim S_2^W(N).$ 
    
    The Weil conjectures imply that
$$
\#X_0^W(N)(\mathbb F_q)
=
p^n +1-\sum_i(\alpha_i^n+\overline{\alpha_i}^n),
$$
where the $\alpha_i$ denote the $p$-Frobenius eigenvalues of $J_0^W(N)$. Hence, we have
$$
\#X_0^W(N)(\mathbb F_q)
=
p^n +1
-
2p^{n/2}
\sum_i
T_n\!\left(\frac{a_p(f_i)}{2\sqrt p}\right),
$$
where the $f_i$ form a basis of $S_2^W(N)$ arising from the canonical
bases of the relevant Atkin--Lehner sign pattern spaces. The
equidistribution of $\frac{a_p(f_i)}{\sqrt p}$ gives that
$$
\left|
\frac{\#X_0^W(N)(\mathbb F_q)}{r^W}
-
2p^{n/2}{}
\int_{-2}^2
T_n\left(\frac{x}{2}\right)d\mu_p
\right|
\leq
2p^{n/2}\cdot C_0
\frac{\log p \cdot \delta \left(T_n\left(\frac{x}{2}\right)\right)}
{\log 2N} + \frac{p^n + 1}{r^W},
$$
for $N\geq N_\epsilon$. The integral is evaluated in Lemma \ref{lemma:Chebyshev-integral}, and $\delta(T_n(\frac{x}{2})) = 2n$ by Lemma \ref{lem:variation-of-U}.
\end{proof}

\begin{remark}
    When $W$ is the trivial group, $X_0^W$ equals the full modular curve. In this special case, Proposition \ref{thm:asymptotic-point-count-W-quotient} thus obtains an effective version of Serre's asymptotic \cite[Th\'eor\`eme 9]{Ser97}.
\end{remark}

Now, we will assume a Maeda-type assumption \cite{hida1997non} to give an asymptotic average of the point counts $\#A_f(\mathbb{F}_p)$ for modular forms $f$ for $100\%$ of levels $N.$

\begin{conjecture}[Generalized Maeda conjecture for weight $2$] \label{conj:gen-Maeda}
    ~ 
    \begin{enumerate}
        \item The Hecke polynomials $\mathbf{T}_p^{\nw, \sigma}(N, 2)(x)$ are irreducible for $100\%$ of $N$.
        \item 
        There exists a $d_0$ such that 
        $\mathbb{Q} \left( a_p(f) \right) = K_{f}$ for all newforms $f \in S_2^{\nw, \sigma}(N)$
        with $[ K_{{f}} : \mathbb{Q}] \ge d_0$ and all primes $p \nmid N$.
    \end{enumerate}
\end{conjecture}
This conjecture comes from \cite[Conjecture A, Question 4]{Kimball2021}.

\begin{remark}
    To satisfy the conclusion in Conjecture \ref{conj:gen-Maeda}(2) that $[K_f:\mathbb{Q}] \ge d_0$, it suffices to have that $[\mathbb{Q}({a_p(f)}):\mathbb{Q}] \geq d_0,$ since $\mathbb{Q}({a_p(f)})$ is a subfield of $K_f.$ But Conjecture \ref{conj:gen-Maeda}(1) implies that
    $[\mathbb{Q}({a_p(f)}):\mathbb{Q}] = \dim S_2^{\mathrm{new},\sigma}(N),$
    so it suffices to verify that $\dim S_2^{\mathrm{new},\sigma}(N) \geq d_0, $ which holds for sufficiently large $N$ by the results of \cite{RVWX26}. Hence part (2) follows from part (1) for sufficiently large $N$. 
\end{remark}

Conjecture \ref{conj:gen-Maeda} implies that for any newform $f \in S_2^{\nw,\sigma}(N)$,
\begin{equation}
    \left\{\frac{\tau(a_p(f))}{2\sqrt{p}}\right\}_{\tau: K_f \hookrightarrow \mathbb{C}}
    =
    \left\{\frac{\tau(a_p(f))}{2\sqrt{p}}\right\}_{\tau: \mathbb{Q}({a_p(f)}) \hookrightarrow \mathbb{C}}
    =
    \left\{\frac{a_p(f_i)}{2\sqrt{p}}\right\}_{\text{newforms } f_i \in S_2^{\mathrm{new},\sigma}(N)}.\label{eqn:conjugates-are-equidistributed}
\end{equation}
Hence, assuming Conjecture \ref{conj:gen-Maeda}, Theorem \ref{thm:main-theorem} implies that, for fixed $p,$ the Galois conjugates $\left\{ \frac{\tau(a_p(f))}{\sqrt{p}} \right\}_{\tau: K_f \hookrightarrow \mathbb{C}}$ are equidistributed as $N$ grows. The following proposition then follows from the proof of Proposition \ref{prop:asymptotic-point-count-Jacobianreg-q}.

\begin{proposition}
\label{thm:asymptotic-point-count}

Fix a prime power $p^n$, and let $N$ be coprime to $p$. Suppose that the conclusion of the generalized Maeda conjecture holds for $N$. Then, for every admissible sign pattern $\sigma$, there exists an Abelian variety $A_\sigma$ and number field $K_\sigma$ such that
\begin{enumerate}
    \item $A_\sigma=A_f$ and $K_\sigma=K_f$
for every newform $f\in S_2^{\mathrm{new},\sigma}(N).$
\item As $N \rightarrow \infty$,
\begin{align}
    \frac{\log \#A_\sigma(\mathbb F_{p^n})}
{\dim A_\sigma}
\longrightarrow
\begin{cases}
\displaystyle
n\log p - \frac{p-1}{2}\log\!\left(1-p^{-2n}\right)
& \text{$n$ odd}\\[10pt]
\displaystyle
n\log p - (p-1)\log\!\left(1-p^{-n}\right)
& \text{$n$ even}
\end{cases}
\qquad\text{as }N\to\infty,
\end{align}
with error bound $C_0^\nw \cdot \frac{\log p}{\log (2N)} \cdot \frac{4n}{p^{n/2} - 1}.$
\end{enumerate}
\end{proposition}

\subsection{Large Hecke Fields} \label{subsection:large-hecke-fields}

We now apply our multiplicity bounds to study the arithmetic of Hecke fields and the dimensions of $\QQ$-simple factors of the Atkin--Lehner isotypic components of modular Jacobians. Let $\{f_1, \cdots, f_r\} \subseteq S_k^{\nw, \sigma}(N)$ be the basis of newforms. For each $f_i$, let
$K_{i} = \mathbb{Q}(\{a_n(f)\}_{n > 1})$ denote the Hecke field of $f_i$. We define, for $d\ge1$,
\begin{align*}
    s^\nw(N,k,\sigma)_d &:=\#\{\,1 \le i \le \text{dim } S_k^{\nw, \sigma}(N) : [K_i:\mathbb{Q}] = d\,\},\\
    s^\nw(N,k,\sigma,p)_d &:= \#\left\{\, \lambda \in \operatorname{eigen}_{S_k^{\nw, \sigma}(N)}(\mathbf{T}_p) : [\mathbb{Q}(\lambda) : \mathbb{Q}] = d \,\right\}.
\end{align*}

For any prime $p$ and $d \in \mathbb{N},$ we establish an upper bound on the number of eigenvalues $\lambda \in \eigen_{S_k^{\nw, \sigma}(N)} \left(\mathbf{T}_{p} \right)$ that are algebraic integers of degree $d.$ We do this through the following result, which can be viewed as an analog of \cite[Theorem 5]{MS07} over the space $S_k^{\nw, \sigma}(N).$

\begin{proposition}\label{prop:degree-bound}
    Fix a prime $p$. For any degree $d\ge1$ and $N$ coprime to $p$, we have
    $$s^\nw(N,k,\sigma)_d \leq d^2 \prod_{j=1}^{d} \left( 2 \binom{d}{j} \left( 2p^{(k - 1)/2} \right)^{j} + 1 \right) \left( \operatorname{dim } S_k^{\nw, \sigma}(N) \cdot C_0^\nw \frac{\log p}{\log kN} \right).$$
\end{proposition}

\begin{proof}
    For an algebraic integer $\alpha,$ define its height $H(\alpha) := \max_{\substack{\tau: \mathbb{Q}(\alpha) \hookrightarrow \mathbb{C} \\ \text{embeddings}}} \{ \left| \tau(\alpha) \right|  \}.$ Given $M \in \mathbb{R}^+$, the number of degree-$a$ algebraic integers $\alpha$ with $H(\alpha) \leq M$ is at most
    $$a \prod_{j=1}^{a} \left( 2 \binom{a}{j} M^{j} + 1 \right),$$
    by \cite[Proposition 30]{MS07}. Then the Deligne bound gives that each $\lambda \in \eigen_{S_k^{\nw, \sigma}(N)}(\mathbf{T}_p)$ can take at most
    $$a \prod_{j=1}^{a} \left( 2 \binom{a}{j} \left( 2p^{(k - 1)/2}  \right)^{j} + 1 \right)$$
    values. Meanwhile, Corollary \ref{cor:main-theorem-with-singleton} gives that any given algebraic integer $\alpha$ is equal to at most $r^\sigma \cdot C_0^\nw \frac{\log p}{\log kN}$ eigenvalues in $\eigen_{S_k^{\nw, \sigma}(N)}(\mathbf{T}_p).$ Therefore,
    \begin{equation}
\label{eq:bound-p}
s^\nw(N, k, \sigma, p)_a
\leq
a \prod_{j=1}^{a}
\left(
2 \binom{a}{j}
\left(2p^{(k-1)/2}\right)^j
+1
\right)
\left(
\dim S_2^{\nw, \sigma}(N)\cdot C_0^\nw
\frac{\log p}{\log(kN)}
\right),
\end{equation}
the right hand side of which is increasing in $a.$ But for any given $p,$ we have
    $$s^\nw(N, k, \sigma)_d \leq \sum_{a = 1}^d s^\nw(N, k, p, \sigma)_a \leq d \cdot \max_{a \leq d} s^\nw(N, k, p, \sigma)_a,$$
    which suffices to prove the result.
\end{proof}

We can use Proposition \ref{prop:degree-bound} to lower-bound the dimension of the largest $\mathbb{Q}$-simple factor of $J_0^{\nw,\sigma}(N)$, for $N$ coprime to a fixed prime $p$.

\begin{proposition}\label{prop:lower-bound-dimension-q-simple-factor-new}
Fix a prime $p$. For all $N$ coprime to $p$, $J_0^{\nw, \sigma}(N)$ has a $\mathbb{Q}$-simple factor of dimension $$d \ge \sqrt{ \max \left\{ 0, \frac{\log \log (2N) - \log(C_0^{\nw} \log p)}{\frac{\log (4\sqrt{2} + 1)}{\log 2} \cdot  \log p} \right\}}.$$
\end{proposition}

\begin{proof}
It is evident that the total number of $\QQ$-simple factors $A_f$ of $J_0^{\nw, \sigma}(N)$ with $\dim A_f = d$ is $\frac{1}{d} s^{\nw}(N,2,\sigma)_d$. In particular, if we let $d_0$ denote the maximum dimension of a $\mathbb{Q}$-simple factor of $J_0^{\nw,\sigma}(N),$ then $s^{\nw}(N, 2, \sigma)_d = 0$ for all $d > d_0.$ Thus, we can write
\begin{align*}
    \operatorname{dim } J_0^{\nw,\sigma}(N) =r^{\nw,\sigma} &= \sum_{d = 1}^{d_0} s^{\nw}(N, 2, \sigma)_d\\
    &\leq \sum_{d=1}^{d_0} d^2 B_p(d) \left(  r^{\nw,\sigma} \cdot C_0^{\nw} \frac{\log p}{\log 2N} \right) &\text{ (by Proposition \ref{prop:degree-bound})}\\
    &\leq  B_p(d_0) \left( r^{\nw,\sigma} \cdot C_0^{\nw} \frac{\log p}{\log 2N} \right) \sum_{d=1}^{d_0} d^2\\
    &\leq  B_p(d_0) \left( r^{\nw,\sigma} \cdot C_0^{\nw} \frac{\log p}{\log 2N} \right) d_0^3,
\end{align*}
where $B_p(d) := \prod_{j=1}^{d} \left( 2 \binom{d}{j} \left( 2\sqrt{p} \right)^{j} + 1 \right)$ is increasing in $d$. Dividing both sides by $r^{\nw,\sigma}$ then gives
\begin{align}
    1 \le C_0^{\nw} \frac{\log p}{\log 2N} \, d_0^3 B_p(d_0), \quad \text{so} \quad \log \left( \frac{\log 2N}{C_0^{\nw} \log p} \right) \le \log(d_0^3 B_p(d_0)) \le \frac{\log (4\sqrt{2} + 1)}{\log 2} \cdot \log p \cdot d_0^2. 
    \\
    \label{eqn:temp-log(d^2B(d))-bound-used}
\end{align}
As desired this yields $d_0^2 \ge  \frac{\log \log (2N) - \log(C_0^{\nw} \log p)}{\frac{\log (4\sqrt{2} + 1)}{\log 2} \cdot  \log p}$. The bound $d_0^2 \geq 0$ is trivial.

In \eqref{eqn:temp-log(d^2B(d))-bound-used} above, we utilized the bound $\log(d^3 B_p(d)) \le \frac{\log (4\sqrt 2 + 1)}{\log 2} \cdot  \log (p)  d^2$. This holds for $1 \le d \le 5$ by direct computation, and for $d \ge 6$ because
\begin{align*}
    &\log\left( d^3 B_p(d) \right)\\
    &= 3 \log(d) + \sum_{1 \le j \le d} \log\left( 2 \binom{d}{j} (2\sqrt{p})^{j} + 1 \right) \\
    &\leq 3 \log(d) + \sum_{1 \le j \le d} \log\left( 2.5 \,  \binom{d}{j} (2\sqrt{p})^{j} \right) \\
    &\le 3 \log(d) + \sum_{1 \le j \le d} \lrb{ \log 2.5 + \log \binom{d}{j} + j\, \log\lrp{ 2\sqrt{p}}} \\
    &= 3 \log(d) + (\log 2.5) d + \log\lrp{ 2\sqrt{p}} \frac{d(d+1)}{2} + \sum_{1 \le j \le d} \log \binom{d}{j} \\
    &\le 3 \log(d) + (\log 2.5) d + \log\lrp{ 2\sqrt{p}} \frac{d(d+1)}{2} + \sum_{1 \le j \le d} j (\log(ed) - \log j ) \qquad \text{since $\binom{d}{j} \leq \left( \frac{ed}{j} \right)^j$}\\
    &\le 3 \log(d) + (\log 2.5) d + \log\lrp{ 2\sqrt{p}} \frac{d(d+1)}{2} + \frac{d(d+1) \log(ed)}{2} - \int_{1}^d x \log x\,dx \\
    &= 3 \log(d) + (\log 2.5) d + \log\lrp{ 2\sqrt{p}} \frac{d(d+1)}{2} + \frac{d(d+1) \log(ed)}{2} - \frac{2 d^2 \log d - d^2 + 1}{4} \\
    &= 3 \log(d) + (\log 2.5)d + \log\lrp{ 2\sqrt{p}} \frac{d(d+1)}{2} +  \frac{3d^2+2d \log d + 2d - 1}{4}\\
    &\le \frac{\log (4\sqrt 2 + 1)}{\log 2} \cdot  \log\lrp{p} d^2 \qquad \text{for $d \ge 6$}.
\end{align*}
This completes the proof.
\end{proof}

\begin{remark}\label{rmk:replace-3.17-if-p-not-2}
    In the above proof, when $p \geq 3,$ we may replace $\frac{\log (4\sqrt{2} + 1)}{\log 2}$ with $\frac{\log (4\sqrt 3 + 1)}{\log 3}.$
\end{remark}

We now follow the strategy of \cite[Corollary 7]{MS07} to eliminate this $p$-dependence and give a lower bound on the dimension of the largest $\QQ$-simple factor of $J_0^{\sigma}(N)$.

{
\renewcommand{\thetheorem}{\ref{prop:lower-bound-dimension-q-simple-factor}}
\begin{proposition}\label{full-sigma-jac-bound}
$J_0^{\sigma}(N)$ has a $\QQ$-simple factor of dimension 
$$d \geq \sqrt{\max \left\{ 0, \frac{\log \log (2 \sqrt{N} ) - \log (C_0 \log 3)}{\log (4\sqrt{3} + 1)} \right\} }.$$
\end{proposition}
\addtocounter{theorem}{-1}
}

\begin{proof}
Recall that if $M \| N,$ then $J_0^{\nw, \sigma_M}(M)$ is a factor of $J_0^\sigma(N),$ where $\sigma_M$ is the sign pattern $\sigma$ for $N$ restricted to $M$. Write $N = AB,$ where $A$ is odd and $B$ is a power of $2.$ In particular, $2 \nmid A$ and $3 \nmid B.$ We have either that $A \geq \sqrt{N}$ or $B \geq \sqrt{N}.$ Since $J_0^{\nw, \sigma_A}(A)$ and $J_0^{\nw, \sigma_B}(B)$ are factors of $J_0^\sigma(N)$, Proposition \ref{prop:lower-bound-dimension-q-simple-factor-new} guarantees that $J_0^\sigma(N)$ has a $\mathbb{Q}$-simple factor of dimension
$$d \geq \begin{cases}
    \sqrt{\max \left\{ 0, \frac{\log \log (2 \sqrt{N} ) - \log (C_0 \log 2)}{\log (4\sqrt{2} + 1)} \right\}} & \text{if $A \geq \sqrt{N}$}\\
    &\\
    \sqrt{\max \left\{ 0, \frac{\log \log (2 \sqrt{N} ) - \log (C_0 \log 3)}{\log (4\sqrt{3} + 1)} \right\} } & \text{if $B \geq \sqrt{N}$},
\end{cases}$$
which suffices for the proof.
\end{proof}

Now, if all the $\QQ$-simple factors of $J_0^{\sigma}(N)$ have dimension at most $d$, then it follows that
\begin{align}
    N &\leq \frac{1}{4} \cdot \exp\! \left( (4\sqrt{3} + 1)^{d^2} \cdot 2C_0 \log 3 \right) \qquad \text{by Proposition \ref{prop:lower-bound-dimension-q-simple-factor}}\\
    N &\leq \frac{1}{2} \cdot \exp\! \left( (4\sqrt{2} + 1)^{d^2} \cdot C_0^{\nw, \text{prime}} \log 2 \right) \qquad \text{if $N$ is prime, by Proposition \ref{prop:lower-bound-dimension-q-simple-factor-new}}.
\end{align}
In particular, by setting $d = 1$ and using our bounds on $C_0$ and $C_0^{\nw, \operatorname{prime}}$ for large $N$ from Lemmas \ref{lem:explicit-constant-for-big-N} and \ref{lem:explicit-constant-for-big-N-prime}, we obtain the following.

\begin{corollary} \label{cor:J0sigma-isogenous-to-product-elliptic-curves}
Suppose $J_0^\sigma(N)$ is isogenous to a product of elliptic curves. Then
\begin{enumerate}
    \item $N \leq e^{809.863}$
    \item If $N$ is prime, then $N \leq e^{35.5290}.$
\end{enumerate}
\end{corollary}

\appendix

\section{Optimization of upper bounds on \texorpdfstring{$C_0$}{C0}}\label{section:lambert-lemmata}

The objective of this appendix is to prove Lemma \ref{lem:lambert-w-function}, which concerns the functions
\begin{align}
    \label{eqn:temp-def-V_x(K)}
        V_x(K) &:= \frac{1}{2} - \frac{\log\left( \frac{x^{D + 1}}{K}  \right)}{x} - \frac{\log \left( 2^{1.38407B +  1.53794C}  \right)}{\log x}\\
        F_x(K) &:= \frac{3\sqrt{2}}{4\pi} \cdot V_x(K)^{-1} + AK.\\
            K(x) &:= \frac{3\sqrt{2} x}{16 \pi A} \cdot W\! \left( \sqrt{\frac{3\sqrt{2}}{16 \pi A}} \cdot x^{-D/2} \exp\left( \frac{x}{4} - \frac{x \log\left( 2^{1.38407B +  1.53794C}  \right)}{2\log x}  \right)  \right)^{-2}, \label{eqn:temp-def-K(x)}
\end{align}
where $A > \frac{3\sqrt{2}}{\pi}$, $B, C, D \in \mathbb{Z}_{\geq 0}$ are fixed constants and $W$ is the Lambert $W$ function. We establish some basic properties of the functions $V_x(K),$ $F_x(K),$ and $K(x)$ in the next few lemmas.

\begin{lemma}\label{lem:limits-of-lambert-functions}
    Define $V(x) := V_x(K(x))$ and $F(x) := F_x(K(x)).$ Then
    \begin{align}
        \lim_{x \rightarrow \infty} K(x) = 0,\qquad \lim_{x \rightarrow \infty} V(x) = \frac{1}{2}, \qquad \lim_{x \rightarrow \infty} F(x) = \frac{3\sqrt{2}}{2\pi}.
    \end{align}
\end{lemma}

\begin{proof}
    We note that
    \begin{align}
        K(x) &\asymp  x \log\! \left( \sqrt{\frac{3\sqrt{2}}{16 \pi A}} \cdot x^{-D/2} \exp\left( \frac{x}{4} - \frac{x \log\left( 2^{1.38407B +  1.53794C}  \right)}{2\log x}  \right)  \right)^{-2} \\
        & \qquad \text{(since $W(x) \asymp \log x$)}\\
        &\asymp  x \left( \frac{x}{4} - \frac{x \log\left( 2^{1.38407B +  1.53794C}  \right)}{2\log x}  \right)^{-2} \\
        &\asymp \frac{1}{x}, \label{eqn:K(x)-asymp-1/x}
    \end{align}
    proving the first limit. Then
    \begin{align}
        \lim_{x \rightarrow \infty} V(x) &= \lim_{x \rightarrow \infty} \left( \frac{1}{2} - \frac{\log\left( \frac{x^{D + 1}}{K(x)}  \right)}{x} - \frac{\log \left( 2^{1.38407B +  1.53794C}  \right)}{\log x} \right)\\
        &= \lim_{x \rightarrow \infty} \left( \frac{1}{2} - \frac{\log\left( \frac{x^{D + 1}}{1/x}  \right) + O(1)}{x} - \frac{\log \left( 2^{1.38407B +  1.53794C}  \right)}{\log x} \right) \qquad \text{(by \eqref{eqn:K(x)-asymp-1/x})}\\
        &= \frac{1}{2},
    \end{align}
    proving the second. The third limit follows immediately from the first two limits.
\end{proof}

\begin{lemma}\label{lem:K-critical-point}
    For any fixed $x > 1,$ $K = K(x)$ is a critical point of $F_x(K).$
\end{lemma}

\begin{proof}
    Differentiating $F_x(K)$ gives
    \begin{equation}
        \frac{\partial }{\partial K} F_x(K) = A - \frac{3\sqrt{2}}{4\pi x K} \, V_x(K)^{-2}. \label{eqn:derivative-of-T}
    \end{equation}
    As shorthand, make the substitution
    \begin{align*}
            u &:= \frac{x}{2} - \log\left( x^{D + 1} \right) - \frac{x\log \left( 2^{1.38407B +  1.53794C}  \right)}{\log x}\\
            v &:= \frac{3\sqrt{2}}{4\pi A x}
        \end{align*}
    so that setting $\frac{\partial }{\partial K} F_x(K) = 0$ yields $K \cdot \left( \frac{\log K + u}{x} \right)^{2} = v$. This implies that
        $$\sqrt{K} \left( \log \sqrt{K} + u/2 \right) = \pm \frac{x \sqrt{v}}{2},$$
        and multiplying both sides by $e^{u/2}$ yields
        $$\left( e^{\log \sqrt{K} + u/2}  \right) \left( \log \sqrt{K} + u/2 \right) = \pm \frac{x\sqrt{v}}{2} \cdot e^{u/2}.$$
        The Lambert W function can then be used to obtain that
        $$\log \sqrt{K} + u/2 = W\left( \pm \frac{x\sqrt{v}}{2} \cdot e^{u/2} \right).$$
        This then gives that
        \begin{align*}
            K &= \exp\left(W\left( \pm \frac{x\sqrt{v}}{2} \cdot e^{u/2} \right)  \right)^2 \cdot e^{-u}\\
            &= \left( \frac{x\sqrt{v}}{2} \cdot e^{u/2}  \right)^2 W\left( \pm \frac{x\sqrt{v}}{2} \cdot e^{u/2} \right)^{-2} \cdot e^{-u}   \qquad \left( \text{by the identity $e^{W(z)} = \frac{z}{W(z)}$} \right)\\
            &= \left( \frac{x\sqrt{v}}{2} \right)^2W\left( \pm \frac{x\sqrt{v}}{2} \cdot \exp\left( \frac{x}{2} - \frac{x \log \left( 2^{1.38407 B + 1.53794 C}   \right)}{\log x}  \right) \cdot x^{-(D + 1)/2} \right)^{-2},
        \end{align*}
        which equals $K(x)$ when the sign is taken as $+.$ Therefore, $K(x)$ is a critical point.
\end{proof}

\begin{lemma}\label{lem:V(x)-in-proper-range}
    Let $V(x) = V_x(K(x)).$ Then for any $x > 1,$ $V(x) \in \left(0, \frac{1}{2} \right).$
\end{lemma}

\begin{proof}
    Lemma \ref{lem:K-critical-point} guarantees that the right hand side of \eqref{eqn:derivative-of-T} is zero when $K = K(x)$. Solving for $V(x) = V_x(K(x))$ then gives
        \begin{equation}
            V(x) = \pm \sqrt{\frac{3\sqrt{2}}{4\pi Ax K(x)}}. \label{eqn:formula-for-V}
        \end{equation}
        In fact, the following three facts imply that this sign is $+1$.
        \begin{enumerate}
            \item $K(x)$ is positive and continuous over $x \in (1,\infty)$ (by \eqref{eqn:temp-def-K(x)}), so $V(x)$ is nonzero (by \eqref{eqn:formula-for-V}).
            \item $V(x) = V_x(K(x))$ is continuous over $x \in (1,\infty)$ (by \eqref{eqn:temp-def-V_x(K)}).
            \item $V(x)= V_x(K(x))$ tends to $\frac{1}{2}$ as $x \to \infty$ (by Lemma \ref{lem:limits-of-lambert-functions}).
        \end{enumerate}
        This shows that $V(x) > 0.$ Moreover, we can simplify \eqref{eqn:formula-for-V} to obtain
        \begin{align}
            V(x) &= \frac{2}{x} \cdot  W\left( \sqrt{\frac{3\sqrt{2}}{16 \pi A}} \cdot x^{-D/2} \exp\left( \frac{x}{4} - \frac{x \log\left( 2^{1.38407B +  1.53794C}  \right)}{2\log x}  \right) \right) \\
            &\leq \frac{2}{x} \cdot  W\left( \sqrt{\frac{3\sqrt{2}}{16 \pi A}} \cdot x^{-D/2} \exp\left( \frac{x}{4}
            \right) \right) \qquad \text{since $W$ is an increasing function }\\
            &< \frac{2}{x} \cdot W\left( \frac{x}{4} \exp\left( \frac{x}{4}
            \right)\right)  \qquad \text{since $A > \frac{3\sqrt{2}}{\pi}$} \\
            &= \frac{1}{2},
        \end{align}
        which suffices for the proof.
\end{proof}

\begin{lemma}\label{lem:K-global-minimum}
    For fixed $x > 1,$ $F_x$ has a global minimum of $K(x)$ over $I_x := \{K \in \mathbb{R}^+ : V_x(K) > 0 \}.$
\end{lemma}

\begin{proof}
    First, we remark that $I_x$ is clearly an interval and therefore has one connected component. By Lemma \ref{lem:K-critical-point}, $K(x)$ is a critical point, so it suffices to show that $\frac{\partial^2}{\partial K^2} F_x(K)$ is positive over $I_x.$ Since we have
    $$\frac{\partial^2}{\partial K^2} F_x(K) = \frac{3\sqrt{2}}{2\pi V_x(K)^3} + Ax^2 K,$$
    the result then follows from Lemma \ref{lem:V(x)-in-proper-range}, which guarantees that $V_x(K) > 0.$
\end{proof}

\begin{lemma}\label{lem:criterion-for-F-decreasing}
    Suppose $x$ satisfies
    \begin{equation}
            \log K(x) < (D + 1) \left( \log x - 1 \right) + \frac{x \log\left( 2^{1.38407B +  1.53794C} \right)}{\left( \log x \right)^2}. 
            \label{eqn:criteron-for-F-decreasing-appendix}
        \end{equation}
    Then $F(x) := F_x(K(x))$ is decreasing at $x$.
\end{lemma}

\begin{proof}
        Observe that
        \begin{align}
            \frac{d}{dx}F(x) 
            &= \left. \frac{\partial F_{x}(K)}{\partial x} \right|_{K = K(x)} + \left. \frac{\partial F_x(K)}{\partial K} \right|_{K = K(x)} \cdot \frac{d}{dx} K(x) \\
            &= \left. \frac{\partial F_{x}(K)}{\partial x} \right|_{K = K(x)} 
            \qquad\qquad \text{(by Lemma \ref{lem:K-critical-point})}
            \\
            &= - \frac{3\sqrt{2}}{4\pi} \left( V_x(K(x))^{-2} \cdot \left.\frac{\partial V_x(K)}{\partial x} \right|_{K = K(x)}\right) \\
            &= 
            - \frac{3\sqrt{2}}{4\pi} V_x(K(x))^{-2} 
            \lrb{
                - \frac{\log K(x)}{x^2} + (D + 1)\left( \frac{\log(x)}{x^2} - \frac{1}{x^2}  \right) - \frac{\log \left( 2^{1.38407 B + 1.53794 C}  \right)}{x \log(x)^2}
            } \\
            &< 0. \qquad \text{(by Lemma \ref{lem:V(x)-in-proper-range} and \eqref{eqn:criteron-for-F-decreasing-appendix})}
        \end{align} 
    This suffices for the proof.
\end{proof}

We are now ready to prove Lemma \ref{lem:lambert-w-function}, which we restate here.

{
\renewcommand{\thetheorem}{\ref{lem:lambert-w-function}}
\begin{lemma}
    Fix $A > \frac{3\sqrt{2}}{\pi}$ and $B, C, D \in \mathbb{Z}_{\geq 0}.$ Suppose that $S(k, N, p)$ is a function over integers $k, p \geq 2$ and $N \geq 1129$ such that for any $c \in \left( 0, \frac{1}{2} \right),$
    $$S(k, N, p) \leq \frac{3\sqrt{2}}{4\pi c} \cdot \frac{\log p}{\log kN} + \frac{A\log 2}{\sqrt{2} \log 3} \cdot \frac{k^c}{k - 1} \cdot \frac{2^{ B\omega(N)} \sigma_0(N)^C \log(N)^D}{N^{1/2 - c}}.$$
    Furthermore, suppose that $X \geq \log(1129)$ such that for all $x \geq X,$
    \begin{equation}
            \log K(x) < (D + 1) \left( \log x - 1 \right) + \frac{x \log\left( 2^{1.38407B +  1.53794C} \right)}{\left( \log x \right)^2}.
        \end{equation}
    Then, for all $N \geq e^X,$
    \begin{align}
        S(k, N, p) &\leq F_{\log N} \left( K(\log N)  \right) \cdot \frac{\log p}{\log kN},\\
        S(k, N, p) &\leq F_X(K(X)) \cdot \frac{\log p}{\log kN}.
    \end{align}
\end{lemma}
\addtocounter{theorem}{-1}
}

\begin{proof}
    We choose the value
    \begin{equation}
        c := V_{\log(N)}(K(\log N)) = \frac{1}{2} - \frac{\log \left(\frac{\left( \log N  \right)^{D + 1}}{K(\log N)} \right)}{\log N} - \frac{\log \left( 2^{1.38407B +  1.53794C}  \right)}{\log \log N}. \label{eqn:value-of-c}
    \end{equation}
    Note that since $N \geq 1129,$ we certainly have $\log N > 1.$ Therefore, Lemma \ref{lem:V(x)-in-proper-range}(2) guarantees that $c \in \left(0, \frac{1}{2} \right).$ Then rearranging \eqref{eqn:value-of-c} and multiplying by $\log N$ yields
    $$\frac{\log N}{\log \log N} \log\left( 2^{1.38407B +  1.53794C}  \right) + (D + 1) \log\log N = \left( \frac{1}{2} - c \right)\log N + \log (K(\log N)).$$
    Exponentiating both sides then yields
    $$\left( 2^{1.38407 \frac{\log N}{\log \log N}} \right)^B \cdot \left( 2^{1.53794 \frac{\log N}{\log \log N}} \right)^C  \cdot \log(N)^{D +1} = N^{1/2 - c}\cdot K(\log N),$$
    which implies by \eqref{eqn:Nicolas-Robin-bound} and \eqref{eqn:Robin-bound} that
    $$\frac{2^{B\omega(N)} \cdot \sigma_0(N)^C  \log (N)^D}{N^{1/2 - c}} \leq \frac{K(\log N)}{\log N}.$$
    Therefore, by assumption, we have
    \begin{align*}
        S(k, N) &\leq \frac{3\sqrt{2}}{4\pi c} \cdot \frac{\log p}{ \log kN} +  A \cdot \frac{\log 2}{\sqrt{2} \log 3} \cdot \frac{k^c}{k - 1} \cdot \frac{K(\log N)}{\log N}\\
        &\leq \frac{3\sqrt{2}}{4\pi c} \cdot \frac{\log p}{ \log kN} +  A  \cdot \frac{\log p}{\log(k + 1)} \cdot \frac{K(\log N)}{\log N}\\
        & \qquad \left( \text{since $c < \frac{1}{2}$ guarantees that $ \sqrt{2} \log 3 \frac{k - 1}{k^c} \geq \log(k + 1)$ for $k \geq 2$} \right) \\
        &\leq \frac{3\sqrt{2}}{4\pi c} \cdot \frac{\log p}{ \log kN} +  A \cdot \frac{K(\log N) \cdot  \log p}{\log kN}\\
        & \qquad  \text{\bigg(since $N \geq 1129$ guarantees that $\log N \geq \frac{\log 2}{\log 3 - 1} \geq \frac{\log k}{\log(k + 1) - 1}$,}\\
        &\qquad\quad  \text{which then implies that $\frac{1}{\log(k + 1) \log N} \leq \frac{1}{\log k + \log N}$\bigg)}\\
        &= F_{\log N} (K(\log N)) \cdot \frac{\log p}{ \log kN},
    \end{align*}
    which proves the first inequality. The fact that \eqref{eqn:criteron-for-F-decreasing-appendix} holds over all $x \in [X, \infty)$ guarantees that $F_x(K(x))$ is decreasing over this interval by Lemma \ref{lem:criterion-for-F-decreasing}. Therefore, $F_{\log N} (K(\log N)) \leq F_{X} (X)$, proving (2).
\end{proof}

\begin{remark}
    The bound given in Lemma \ref{lem:lambert-w-function} is optimal, in the sense that we chose the value of $c \in (0, \infty)$, dependent on $N,$ that gives the smallest possible coefficient $C_1$ for the upper bound $S(k,N,p) \le C_1 \frac{\log p}{\log kN}$ in part (1). This is guaranteed by Lemma \ref{lem:K-global-minimum}.
\end{remark}

\section{Computations involving Chebyshev polynomials}\label{section:chebyshev}

To prove the results in $\S$\ref{subsection:point-count}, we need to understand how the Chebyshev polynomials behave under the measure $\mu_p$. We prove the following.

\begin{lemma}\label{lemma:Chebyshev-integral}
    For any $n \in \mathbb{N},$ we have
    $$\int_{-2}^2 T_n \left(\frac{x}{2}\right) d\mu_p(x) = \int_{0}^\pi \cos(n \theta) \, d \mu_p(\theta) = \begin{cases}
        1 & n = 0\\
        0 & \text{$n$ odd}\\
        \frac{1-p}{2 p^{n/2}} & \text{$n$ even, $n \geq 2$.}
    \end{cases}$$
\end{lemma}

\begin{proof}
    Since $\mu_p$ is an even probability measure, the first two cases are immediate. For $n \geq 2$ even,
\begin{align}
& \quad \int_{0}^\pi \cos(n \theta) \, d \mu_p(\theta)\\
&= \frac{2(p+1)}{\pi} \int_{0}^{\pi} \cos (n \theta) \cdot \frac{\sin^2\theta}{(p^{1/2} + p^{-1/2})^2 - 4\cos^2\theta} \, d\theta \\
&= \frac{2(p+1)}{\pi} \int_{0}^{\pi} \frac{\cos\left( n\theta \right) \sin^2\theta}{2(p^{1/2}+p^{-1/2})} \left[ \frac{1}{p^{1/2}+p^{-1/2} - 2\cos\theta}  + \frac{1}{p^{1/2}+p^{-1/2} + 2\cos\theta} \right] d\theta \\
&= \frac{p^{1/2}}{\pi} \int_{0}^{\pi} \left( \frac{2\cos(n\theta) \sin^2\theta}{p^{1/2}+p^{-1/2} - 2\cos\theta} \right) d\theta\\
& \qquad \left( \text{since the integrals of the two summands are equal when $n$ is even}  \right)\\
&= \frac{p^{1/2}}{\pi}\int_{0}^\pi \frac{\cos (n \theta) - \frac{1}{2} \cos\left( (n + 2) \theta\right) - \frac{1}{2} \cos \left( (n - 2) \theta \right)}{ p^{1/2} + p^{-1/2} - 2 \cos \theta} d\theta \label{eqn:trig-integral}
\end{align}

The identity \cite[(3.613.2)]{gradshteyn2014table} implies that
$$\int_{0}^\pi \frac{\cos k \theta}{p^{1/2} + p^{-1/2} - 2 \cos \theta} d \theta = \frac{\pi}{\sqrt{\left( p^{1/2} + p^{-1/2} \right)^2 - 4}} \cdot p^{-1/2}.$$
A direct application of this identity shows that \eqref{eqn:trig-integral} evaluates to $\frac{1 - p}{2 p^{n/2}}.$
\end{proof}

For a useful application of Lemma \ref{lemma:Chebyshev-integral}, we consider the function
$$g_n(x) = \log \left( p^n + 1 - 2p^{n/2} \cdot T_n \left(  \frac{x}{2} \right)  \right).$$
defined in the proof of Proposition \ref{prop:asymptotic-point-count-Jacobian-q}. Then we have the following.

\begin{lemma}\label{lem:g-integral}
    for $n \in \mathbb{N},$
    $$\int_{-2}^{2}
g_n(x)\,d\mu_p(x) = \begin{cases}
\displaystyle \log p^n-\frac{p-1}{2}\log\!\left(1- p^{-2n}\right),
& \text{$n$ odd},\\[10pt]
\displaystyle \log p^n -(p-1)\log\!\left(1-p^{-n}\right),
& \text{$n$ even},
\end{cases}$$
\end{lemma}
\begin{proof}
    We have
    \begin{align}
        &\quad \int_{-2}^2 \log\left( p^n + 1 - 2p^{n/2} \cdot T_n \left( \frac{x}{2} \right)  \right) \, d \mu_p(x)\\
        &= \int_{0}^\pi \log\left( p^n + 1 - 2p^{n/2} \cdot \cos (n \theta)  \right) \, d \mu_p(\theta)\\
        &= \int_{0}^\pi \log\left( p^n \left( 1 - \frac{e^{in\theta}}{p^{n/2}} \right)\left( 1 - \frac{e^{-in\theta}}{p^{n/2}} \right)  \right) \, d \mu_p(\theta)\\
        &= \int_{0}^\pi \left( \log(p^n) - \sum_{k \geq 1} \frac{e^{ink \theta}}{k p^{kn/2}} - \sum_{k \geq 1} \frac{e^{-ink \theta}}{k p^{kn/2}} \right) \, d \mu_p(\theta)\\
        & \qquad \left( \text{since $\frac{e^{\pm i n \theta}}{p^{n/2}}$ lies in the radius of convergence of the Taylor series for $-\log(1 - x)$} \right)\\
        &= \int_0^\pi \left( \log \left( p^n \right) - 2 \sum_{k \geq 1} \frac{p^{-kn/2}}{k} \cos(kn \theta)  \right) \, d \mu_p(\theta)\\
        &= \log \left( p^n \right) - 2 \sum_{k \geq 1} \frac{p^{-kn/2}}{k} \int_0^\pi \cos(kn \theta) \, d \mu_p(\theta)\\
        &= \log \left( p^n \right) + (p - 1) \begin{cases}
        \sum_{k \geq 1} \frac{1}{2k p^{2kn}}  & \text{$n$ odd}\\
        \sum_{k \geq 1} \frac{1}{k p^{kn}} & \text{$n$ even}.
    \end{cases} \qquad \left( \text{by Lemma \ref{lemma:Chebyshev-integral}}  \right)\\
    &= \log \left( p^n \right) - (p - 1) \cdot \begin{cases}
        \frac{1}{2}\log \left(1 - p^{-2n} \right) & \text{$n$ odd}\\
        \log \left(1 - p^{-n}\right) & \text{$n$ even},
    \end{cases}
    \end{align}
    which suffices for the proof.
\end{proof}

Finally, in order to apply Corollary \ref{cor:main-thm-for-general-function-g}, we calculate the total variations of $T_n(x/2)$ and $U_n(x/2)$ over the interval $[-2, 2].$

\begin{lemma}\label{lem:variation-of-U}
    For any integer $n \geq 1,$
    \begin{align}
        \delta(T_n(x/2)) &= 2n
        \\
        \delta(U_n(x/2)) &\leq 2(n + 1) \left( \log n + 2 - \log 2 \right),
    \end{align}
    where $\delta$ denotes the total variation over the interval $[-2, 2].$
\end{lemma}
\begin{proof}
    The calculation for $T_n$ is clear, so we focus on $U_n.$ It is equivalent to consider the total variation of $U_n(\cos \theta)$ over $[0, \pi].$ Partition the interval into $\{J_k\}_{k = 0}^n,$
    $$J_k := \left[ \frac{k \pi}{n + 1}, \frac{(k + 1)\pi}{n + 1}  \right],$$
    the endpoints of which are precisely the zeros of $U_n(\cos \theta),$ except for the endpoints $0$ and $\pi,$ at which $U(\cos \theta) = n + 1.$ The critical points of $U_n(\cos \theta)$ are at the points $\theta$ such that
    $$\tan((n + 1) x) = (n + 1) \tan x,$$
    of which there is exactly one in each $J_k.$ Thus, $U_n$ is unimodal over each $J_k,$ so the variation over each $J_k$ is given by $2\sup_{x \in J_k} \left| U_n(x) \right|.$ Moreover, the variation in the first and last intervals will be $n + 1.$ This yields
    \begin{align}
        \delta(U_n(x/2)) &= 2(n + 1) + 2\sum_{k = 1}^{n - 1} \sup_{\theta \in J_k} \left| U_n(\cos \theta) \right|\\
        &\leq 2(n + 1) + 2\sum_{k = 1}^{n - 1} \sup_{\theta \in J_k} \frac{1}{\left| \sin(\theta) \right|}\\
        &\leq 2(n + 1) + 4\sum_{k = 1}^{\lfloor n/2 \rfloor} \sup_{\theta \in J_k} \frac{1}{\left| \sin(\theta) \right|} \qquad \text{(by symmetry, possibly repeating the midpoint)}\\
        &= 2(n + 1) + 4 \sum_{k = 1}^{\lfloor n/2 \rfloor}  \frac{1}{\left| \sin\left( \frac{k \pi}{n + 1}  \right) \right|}  \qquad \left( \text{since $\sin(\theta)$ is increasing over $\left[ 0, \frac{\pi}{2} \right]$}  \right)\\
        &= 2(n + 1) + 4 \sum_{k = 1}^{\lfloor n/2 \rfloor}   \frac{n + 1}{2k}  \qquad \left( \text{since $\sin(x) \geq \frac{2x}{\pi}$} \right)\\
        &\leq 2(n + 1) + 2(n + 1) \left( \log \frac{n}{2} + 1 \right) \\
        &= 2(n+1) \lrp{\log n + 2 - \log 2},
    \end{align}
    as desired.
\end{proof}

\section{The Proof of Proposition \ref{prop:kim-2-6-analog}}\label{subsection:integral-bounds}

The goal of this section is to prove Proposition \ref{prop:kim-2-6-analog}, which was used in the proof of Proposition \ref{prop:k-dependent-extremality}. The proof will require explicit bounds on several separate integrals, computed in Lemmas \ref{lem:kim-integral-simple}, \ref{lem:kim-integral-2-1}, \ref{lem:obvious-fejer-fourier-integral}, \ref{lem:kim-integral-2.2}, and \ref{lem:kim-integral-1-2}. Each of the integrals in these lemmas contribute part of the total bound of $b_n$ described in  Proposition \ref{prop:kim-2-6-analog}.

We let the variable $\theta$ be such that $2\cos \theta = x$. Let $M$ be an integer greater than or equal to $3$, and let $I := \left[ 0, \frac{1}{M} \right] \subseteq \left[ 0, \pi \right]$. Let $y = \frac{\theta}{2\pi}$ and $\beta = \frac{1}{2\pi M}$. Then by \eqref{eqn:upper-bound-for-mu-p-in-terms-of-theta},
\begin{equation}
        d\mu_p(2 \pi y) = d\mu_p(\theta) \leq \frac{p(p + 1)}{(p - 1)^2} \cdot \frac{2}{\pi} \sin^2(\theta) \, d \theta = \frac{p(p + 1)}{(p - 1)^2} \cdot 4 \sin^2(2\pi y) \, dy. \label{eqn:mu-p-upper-bound-in-terms-of-2-pi-x}
    \end{equation}
\begin{remark}
    The equation \eqref{eqn:def-mu-p-theta} can be used to extend the range of $d\mu_p(\theta)$ to $[-\pi,\pi]$ rather than just $[0, \pi].$ We make use of this equality throughout this appendix. In particular, the bound \eqref{eqn:mu-p-upper-bound-in-terms-of-2-pi-x} holds over all of $[-\pi,\pi]$.
\end{remark}
We recall the definitions of the following functions.
\begin{enumerate}
    \item 
    \begin{align}
        R_n(x):= \frac{p - 1}{p} U_n(x) + \frac{2}{p} T_n(x), \qquad \text{(as in \eqref{eqn:definition-of-Rn(x)})}
    \end{align}
    \item The sawtooth function
        $$s(x)  := \begin{cases}
        \{x \} - 1/2 & x \notin \mathbb{Z}\\
        0 & x \in \mathbb{Z},
    \end{cases}$$
    where $\{x\}$ denotes the fractional part of $x,$
    \item The Fej\'er kernel
    $$\Delta_M(x) := \frac{1}{M} \left( \frac{\sin \pi M x}{\sin \pi x} \right)^2,$$
    \item The Vaaler polynomial
    \begin{align*}
        V_M(x) &:= \frac{1}{M + 1} \sum_{k = 1}^M \left( \frac{k}{M + 1} - \frac{1}{2}  \right) \Delta_{M + 1} \left( x - \frac{k}{M + 1}  \right)\\
        &\quad\ + \frac{1}{2\pi(M + 1)} \sin\left(2\pi(M + 1)x \right) - \frac{1}{2\pi} \Delta_{M + 1}\left( x  \right) \sin 2 \pi x,
    \end{align*}
    \item The Beurling polynomial
    $$B_M(x) := V_M(x) + \frac{1}{2(M + 1)} \Delta_{M + 1}(x).$$
    \end{enumerate}

We begin with a simple bound.
\begin{lemma} \label{lem:kim-integral-simple}
    Let $M \in \mathbb{N}$ and $I = \left[ 0, \frac{1}{M} \right]$. Then
    $$\left| \int_0^\pi \chi_I(\theta) R_n(\cos \theta) \,d\mu_p(\theta) \right| \leq \left(  \frac{p + 1}{p - 1} \right)^2 \cdot \frac{1}{\pi M^2}.$$
\end{lemma}

\begin{proof}
    We observe that
    \begin{align}
        &\quad \left| \int_0^\pi \chi_I(\theta) R_n(\cos \theta) \,d\mu_p(\theta) \right|\\
        &\leq  \int_0^\pi \left|\chi_I(\theta) R_n(\cos \theta) \left(  \frac{p(p + 1)}{(p - 1)^2} \cdot \frac{2}{\pi} \left| \sin^2 \theta \right| \right)\right| \,d\theta  \qquad \text{by \eqref{eqn:upper-bound-for-mu-p-in-terms-of-theta}}\\
        &\leq \frac{p(p + 1)}{(p - 1)^2} \cdot \frac{2}{\pi} \int_0^{1/M} \chi_I(\theta) \left( \frac{p - 1}{p}\left|\sin((n + 1) \theta) \cdot \sin \theta \right| + \frac{2}{p} \left| \cos \theta \cdot \sin^2 \theta \right|  \right) \,d\theta\\ 
        & \qquad \text{by the definitions of $\chi_I$ and $R_n$}\\
        &\leq \frac{p(p + 1)}{(p - 1)^2} \cdot \frac{2}{\pi} \int_0^{1/M}  \frac{p + 1}{p} \cdot \theta \,d\theta \qquad \text{since $\left| \sin \theta \right| \leq \theta$ for $\theta \in \left[0, 1\right]$}\\
        &\leq \left(  \frac{p + 1}{p - 1} \right)^2 \cdot \frac{2}{\pi} \cdot \frac{1}{2M^2},
    \end{align}
    which suffices for the proof.
\end{proof}

\begin{lemma}\label{lem:kim-integral-2-1}
For any $M \in \mathbb{N}$ at least $3$ and $\beta = \frac{1}{2\pi M},$
\begin{align}
    \left| \int_{-1/2}^{1/2}
 \left( V_M(y - \beta ) - s(y - \beta) \right) R_n ( \cos 2 \pi y ) \mu_p(2\pi y) \right| &\leq \left( \frac{p + 1}{p - 1}  \right)^2 \cdot  \frac{7.36063}{M^2},\\
 \text{and} \qquad
 \left| \int_{-1/2}^{1/2}
 \left( V_M(y) - s(y) \right) R_n ( \cos 2 \pi y ) \mu_p(2\pi y) \right| &= 0.
\end{align}
\end{lemma}

\begin{proof}
    It is clear that the second integral is zero since the integrand is an odd function. For the first inequality, a direct application of Proposition \ref{prop:explicit-constant-vaaler-bound} yields
    $$\left| V_M(y - \beta) - s(y - \beta) \right| \leq \begin{cases}
    \frac{1}{2} & \text{everywhere}\\
    \frac{0.14}{\left(M | y - \beta| \right)^3} & \frac{1}{M} \leq |y - \beta| \leq \frac{1}{2}
\end{cases}.$$
    
    Using \eqref{eqn:mu-p-upper-bound-in-terms-of-2-pi-x} and the definition of $R_n(\cos \theta),$ we have
    \begin{align}
        & \quad\left| \int_{-1/2}^{1/2}
 \left( V_M(y - \beta ) - s(y - \beta) \right) R_n ( \cos 2 \pi y ) \mu_p(2\pi y) \right|\\
 &\leq \frac{p(p + 1)}{(p - 1)^2} \cdot 4  \int_{-1/2}^{1/2}
 \left| V_M(y - \beta ) - s(y - \beta) \right| \cdot \left( \frac{p - 1}{p \left| \sin(2\pi y) \right|} + \frac{2}{p} \right) \left| \sin(2\pi y) \right|^2 dy\\
 &\leq \left(  \frac{p + 1}{p -1} \right)^2 \cdot 4  \int_{-1/2}^{1/2}
 \left| V_M(y - \beta ) - s(y - \beta) \right| \cdot  \left| \sin(2\pi y) \right| dy\\
 & \qquad \text{since $\sin^2(2\pi y) \leq \left| \sin(2\pi y) \right|$}\\
 &\leq  \left(  \frac{p + 1}{p -1} \right)^2 \cdot 2  \int_{0}^{\beta + B/M}
 \left| \sin(2\pi y) \right| dy \label{eqn:Kim-integral-2-1-1} \\
 & \quad + \left(  \frac{p + 1}{p -1} \right)^2  4  \cdot \frac{0.14}{M^3} \left( \int_{-1/2}^{0}
 \left| \frac{\sin(2\pi y)}{(y - \beta)^3}  \right| dy + \int_{\beta + B/M}^{1/2}
 \left|\frac{\sin(2\pi y)}{(y - \beta)^3} \right| dy \right), \label{eqn:Kim-integral-2-1-2}
 \end{align}
 for some small constant $B$ to be chosen later. We first bound \eqref{eqn:Kim-integral-2-1-1}.
\begin{align}
    \eqref{eqn:Kim-integral-2-1-1} &\leq \left(  \frac{p + 1}{p -1} \right)^2 \cdot 2  \int_{0}^{\beta + B/M}
 \left| 2\pi y \right| dy \\ 
    &= \left(  \frac{p + 1}{p -1} \right)^2 \cdot 2\pi \lrp{\frac{1}{2\pi M} + \frac{B}{M}}^2\\
    &= \left(  \frac{p + 1}{p -1} \right)^2 \cdot \frac{ 2\pi B^2 + 2B + \frac{1}{2\pi}}{M^2}.
\end{align}
Next, we bound the two integrals in \eqref{eqn:Kim-integral-2-1-2} separately.
\begin{align*}
    \int_{-1/2}^{0}
  \frac{\sin(2\pi y)}{(y - \beta)^3}   dy &\leq \int_{-1/2}^{0}
 \left| \frac{2\pi y}{(y - \beta)^3}  \right| dy\\
 &= 2\pi^2M - \frac{2\pi^2 M (2\pi M + 1)}{(\pi M + 1)^2}\\
 &\leq 2\pi^2M,
\end{align*}
and
\begin{align}
    \int_{\beta + B/M}^{1/2}
 \left|\frac{\sin(2\pi y)}{(y - \beta)^3} \right| dy &\leq \int_{\beta + B/M}^{1/2}
 \frac{2\pi y}{(y - \beta)^3}  dy\\
 &= \frac{M(4\pi B + 1)}{2B^2} - \frac{2\pi^2 M(2\pi M - 1)}{(\pi M - 1)^2}\\
 &\leq \frac{M(4\pi B + 1)}{2B^2}.
\end{align}
Putting these three integral bounds together, we get that
\begin{align}
    &\quad \left| \int_{-1/2}^{1/2}
 \left( V_M(y - \beta ) - s(y - \beta) \right) R_n ( \cos 2 \pi y ) \mu_p(2\pi y) \right|\\
 &\leq \left( \frac{p + 1}{p - 1} \right)^2 \cdot \frac{1}{M^2}  \left( 2\pi B^2 + 2B + \frac{1}{2\pi}   +  0.14\left(  2\pi^2 + \frac{4\pi B + 1}{2B^2} \right) \right)\\
 &\leq \left( \frac{p + 1}{p - 1} \right)^2 \cdot \frac{7.36063}{M^2} \qquad \text{when $B = 0.41213.$}
\end{align}
We note that the integral bound $\beta + \frac{B}{M}$ is indeed less than $\frac{1}{2}$ when $B = 0.41213$ for all $M \geq 3.$ 
\end{proof}

\begin{lemma}\label{lem:obvious-fejer-fourier-integral}
    For $M \in \mathbb{N}$ at least 3 and $\beta = \frac{1}{2\pi M}$,
    \begin{align}
        \left| \frac{1}{2(M + 1)} \int_{-1/2}^{1/2} \Delta_{M + 1}(y - \beta) T_n(\cos 2\pi y) \mu_p(2\pi y) \right| &\leq \frac{0.607056}{(M + 1)^2} \cdot \frac{p(p + 1)}{(p - 1)^2},\\
        \left| \frac{1}{2(M + 1)} \int_{-1/2}^{1/2} \Delta_{M + 1}(y) T_n(\cos 2\pi y) \mu_p(2\pi y) \right| &\leq \frac{0.5}{(M + 1)^2} \cdot \frac{p(p + 1)}{(p - 1)^2}.
    \end{align}
\end{lemma}

\begin{proof}
    We use \eqref{eqn:mu-p-upper-bound-in-terms-of-2-pi-x} to bound the first integral by
    $$\frac{p(p + 1)}{(p - 1)^2} \cdot \frac{2}{(M + 1)^2} \int_{-1/2}^{1/2} \left|\left( \frac{\sin \pi(M + 1)(y - \beta)}{\sin\pi(y - \beta)}  \right)^2  \cos (2\pi n y) \sin^2(2\pi y) \right| dy.$$
    We bound $|\cos (2\pi n y)| \leq 1.$ At this point, the integrand is positive, so the absolute value may be removed. It is well known (see, for example, \cite[Chapter 2, Lemma 5.1]{SS03}) that
    $$\left( \frac{\sin M\pi y}{\sin\pi y}  \right)^2 = \sum_{j = -M}^M \left(1 - \frac{|j|}{M} \right) e^{2\pi ij y}.$$
    We then obtain
    \begin{align}
        &\quad \int_{-1/2}^{1/2} \left( \frac{\sin \pi(M + 1)(y - \beta)}{\sin\pi(y - \beta)}  \right)^2  \cos (2\pi n y) \sin^2(2\pi y)  dy\\
        &= \int_{-1/2}^{1/2} \left( \sum_{j = -M - 1}^{M + 1} \left(1 - \frac{|j|}{M + 1} \right) e^{2\pi i j y} \cdot e^{- 2 \pi i j \beta} \right) \sin^2(2\pi y) dy\\
        &= \int_{-1/2}^{1/2} \left( \sum_{j = -M - 1}^{M + 1} \left(1 - \frac{|j|}{M + 1} \right) e^{2\pi i j y} \cdot e^{-2 \pi i j \beta} \right) \left(\frac{1 - \cos(4\pi y)}{2} \right) dy\\
        &= \int_{-1/2}^{1/2} \left( \sum_{j = -M - 1}^{M + 1} \left(1 - \frac{|j|}{M + 1} \right) e^{2\pi i j y} \cdot e^{-2 \pi i j \beta} \right) \left( \frac{1}{2} - \frac{1}{4} e^{4\pi i y} - \frac{1}{4} e^{-4 \pi i y} \right) dy\\
        &= \sum_{j = -M - 1}^{M + 1} \left(1 - \frac{|j|}{M + 1} \right) 
        e^{-2 \pi i j \beta}
        \int_{-1/2}^{1/2}   e^{2\pi i j y}  \left( \frac{1}{2} - \frac{1}{4} e^{4\pi i y} - \frac{1}{4} e^{-4 \pi i y} \right) dy\\
        &= \sum_{j = -M - 1}^{M + 1} \left(1 - \frac{|j|}{M + 1} \right) 
        e^{-2 \pi i j \beta}
        \left( \frac{1}{2} \delta_{j,0} - \frac{1}{4} \delta_{j,2} - \frac{1}{4} \delta_{j,-2} \right) \\
        & \qquad \text{(since $\left\{ e^{2\pi i j y} \right\}_{j \in \mathbb{Z}}$ is an orthonormal set for $L^2(S^1)$)}\\
        &= \frac{1}{2} - \frac{M - 1}{4(M + 1)} e^{4\pi i \beta}  - \frac{M - 1}{4(M + 1)} e^{-4\pi i \beta}\\
        &= \frac{1}{2} - \frac{M - 1}{2(M + 1)} \cos(4\pi \beta)\\
        &\leq \frac{1.21412}{(M + 1)} \qquad \text{for $M \geq 3$} \label{eqn:obvious-fejer-integral-actually-cubic}\\
        &\leq 0.303528,
    \end{align}
    for $M \geq 3,$ with the maximum achieved at $M = 3.$ The same proof with $0$ in place of $\beta$ yields the second bound.
\end{proof}

\begin{lemma} \label{lem:kim-integral-2.2}
    For $M \in \mathbb{N}$ at least $3,$
    \begin{equation}
        \left| \frac{1}{2(M + 1)} \int_{-1/2}^{1/2}
 \Delta_{M+1}(y) R_n ( \cos 2 \pi y ) \,d\mu_p(2\pi y) \right| \leq \frac{p(p + 1)}{(p - 1)^2} \cdot 3.34907 \cdot \frac{\log M}{M^2}. \label{eqn:kim-integral-2-2}
    \end{equation}
\end{lemma}
\begin{proof}
    The triangle inequality, combined with the second bound of Lemma \ref{lem:obvious-fejer-fourier-integral}, \eqref{eqn:mu-p-upper-bound-in-terms-of-2-pi-x}, and the definition of $R_n$, gives that
    \begin{align}
        &\text{LHS of \eqref{eqn:kim-integral-2-2}} \\
        &\leq \frac{p(p + 1)}{(p - 1)^2}\left( \left| \frac{2}{M + 1} \cdot \frac{p - 1}{p} \int_{-1/2}^{1/2}
        \Delta_{M+1}(y) U_n ( \cos 2 \pi y ) \sin^2(2\pi y) \, dy \right| + \frac{1}{p(M + 1)^2} \right).\\
        &= \frac{p(p + 1)}{(p - 1)^2}\left( \left| \frac{2}{(M + 1)^2} \cdot \frac{p - 1}{p} \int_{-1/2}^{1/2}
        \left( \frac{\sin \left( \pi (M + 1) y \right)}{\sin(\pi y)}  \right)^2 \frac{\sin \left( 2\pi (n + 1) y  \right)}{\sin (2\pi y)} \sin^2(2\pi y) \, dy \right| + \frac{1}{p(M + 1)^2} \right).\\
        &\leq \frac{p(p + 1)}{(p - 1)^2}\left(  \frac{4}{(M + 1)^2} \cdot \frac{p - 1}{p} \int_{0}^{B/M} (M + 1)^2 \cdot \sin(2\pi y) \, dy \right) \qquad \text{since $\frac{\sin \left( \pi (M + 1) y \right)}{\sin(\pi y)} \leq M + 1$} \\
        & \quad + \frac{p(p + 1)}{(p - 1)^2}\left(  \frac{4}{(M + 1)^2} \cdot \frac{p - 1}{p} \int_{B/M}^{1/2} \left( \frac{\sin \left( \pi (M + 1) y \right)}{\sin \left( \pi y \right)} \right)^2  |\sin(2\pi y )|  \, dy + \frac{1}{p(M + 1)^2} \right),
    \end{align}
    where $B$ is to be chosen later. We bound the two integrals separately.
    \begin{align}
        \int_{0}^{B/M} (M + 1)^2 \cdot \sin(2\pi y) \, dy &= (M + 1)^2 \cdot \left. \frac{- \cos(2\pi y)}{2\pi } \right|_{y = 0}^{y = B/M}\\
        &= (M + 1)^2 \cdot \frac{1}{\pi}\left( \frac{1 -  \cos\left( \frac{2\pi B }{M}  \right)}{2} \right)\\
        &= (M + 1)^2 \cdot \frac{\sin^2 \! \left( \frac{\pi B}{M} \right)}{\pi} \label{eqn:kim-2.2-1}
    \end{align}
Meanwhile,
\begin{align}
    &\quad \int_{B/M}^{1/2} \left( \frac{\sin \left( \pi (M + 1) y \right)}{\sin \left( \pi y \right)} \right)^2  |\sin(2\pi y )|  \, dy \\
    &= \int_{B/M}^{1/2} \left( \frac{\sin \left( \pi (M + 1) y \right)}{\sin \left( \pi y \right)} \right)^2  2 \sin(\pi y) \cos(\pi y)  \, dy\\
    &\leq 2\int_{B/M}^{1/2}  \frac{\sin \left( \pi (M + 1) y \right)}{\sin \left( \pi y \right)} \cdot  \cos(\pi y)  \, dy\\
    &\leq 2\int_{B/M}^{1/2}  \frac{ \sin(\pi M y) \cos (\pi y) + \cos(\pi M y) \sin(\pi y)}{\sin \left( \pi y \right)} \cdot  \cos(\pi y)  \, dy\\
    &\leq 2\int_{B/M}^{1/2}  \left( \cot(\pi y) + 1 \right)  \cos(\pi y)  \, dy\\
    &= \frac{2}{\pi} \left( 1 + \ln \left| \tan \left( \frac{\pi B}{2M} \right) \right| - \sin \left( \frac{\pi B}{M} \right) - \cos \left( \frac{\pi B}{M} \right)  \right)\\
    &\leq \frac{2}{\pi} \log \left( \tan \left( \frac{B \pi}{2M} \right) \right)  \qquad \text{for $M \geq 2B$}\\
    &\leq \frac{2}{\pi} \log \left( \frac{2M}{B\pi} \right) \label{eqn:kim-2.2-2}
\end{align}
Therefore, for $M \geq 2B,$
\begin{align}
    &\text{LHS of \eqref{eqn:kim-integral-2-2}} \\
    &\leq \frac{p(p + 1)}{(p - 1)^2} \cdot \frac{1}{(M + 1)^2} \left( \frac{p - 1}{p} \left( \frac{4(M + 1)^2}{\pi} \cdot \sin^2\left( \frac{B \pi}{2M}  \right) + \frac{8}{\pi} \log \left( \frac{2M}{B\pi} \right) \right) + \frac{1}{p} \right)\\
    &\leq \frac{p(p + 1)}{(p - 1)^2} \cdot \frac{1}{(M + 1)^2} \lim_{p \rightarrow \infty}\left( \frac{p - 1}{p} \left( \frac{4(M + 1)^2}{\pi} \cdot \sin^2\left( \frac{B \pi}{2M}  \right) + \frac{8}{\pi} \log \left( \frac{2M}{B\pi} \right) \right) + \frac{1}{p} \right)\\
    & \qquad \text{since, for any $M \geq 3$, the argument of the limit is increasing in $p$}\\
    &= \frac{p(p + 1)}{(p - 1)^2} \cdot \frac{1}{(M + 1)^2}  \left( \frac{4(M + 1)^2}{\pi} \cdot \sin^2\left( \frac{B \pi}{2M}  \right) + \frac{8}{\pi} \log \left( \frac{2M}{B\pi} \right) \right) \\
    &\leq \frac{p(p + 1)}{(p - 1)^2} \cdot \frac{\log(M)}{M^2} \left(  \frac{M^2}{\log(M) (M + 1)^2} \left(  \frac{4(M + 1)^2}{\pi} \cdot \sin^2\left( \frac{B \pi}{M}  \right) + \frac{8}{\pi} \log \left( \frac{2M}{B\pi} \right) \right) \right)\\
    &\leq \frac{p(p + 1)}{(p - 1)^2} \cdot \frac{\log(M)}{M^2} \cdot 3.34907 \qquad \text{when $B = 0.24403$}.
\end{align}
This completes the proof.
\end{proof}

\begin{lemma}\label{lem:kim-integral-1-2}
    For $M \in \mathbb{N}$ at least $3$ and $\beta = \frac{1}{2\pi M},$
    \begin{equation}
        \left| \frac{1}{2(M + 1)} \int_{-1/2}^{1/2}
        \Delta_{M+1}(y - \beta) R_n ( \cos 2 \pi y ) \,d\mu_p(2\pi y) \right| \leq \frac{p(p + 1)}{(p - 1)^2} \cdot 3.72936 \cdot \frac{\log M}{M^2} . \label{eqn:kim-integral-1-2}
    \end{equation}
\end{lemma}

\begin{proof}
    The proof is similar to that of Lemma \ref{lem:kim-integral-2.2}. Using the triangle inequality, the first bound of Lemma \ref{lem:obvious-fejer-fourier-integral}, \eqref{eqn:mu-p-upper-bound-in-terms-of-2-pi-x}, and the definition of $R_n$, we find that 
    \begin{align}
        &\quad \text{LHS of \eqref{eqn:kim-integral-1-2}}\\ &\leq \frac{p(p + 1)}{(p - 1)^2}\left( \left| \frac{2}{M + 1} \cdot \frac{p - 1}{p} \int_{-1/2}^{1/2}
 \Delta_{M+1}(y - \beta) U_n ( \cos 2 \pi y ) \sin^2(2\pi y)\, dy \right| + \frac{2}{p} \cdot \frac{0.607056}{(M + 1)^2} \right)\\
 &\leq \frac{p(p + 1)}{(p - 1)^2}\left( \left| \frac{2}{M + 1} \cdot \frac{p - 1}{p} \int_{-1/2}^{1/2}
 \Delta_{M+1}(z) U_n ( \cos 2 \pi (z + \beta )) \sin^2(2\pi (z + \beta))\, dz \right| + \frac{2}{p} \cdot \frac{0.607056}{(M + 1)^2} \right)\\
 &\qquad \text{where $z = y - \beta$ and the bounds of integration stay the same by periodicity of the integrand}\\
 &\leq \frac{p(p + 1)}{(p - 1)^2}\left( \frac{2}{M + 1} \cdot \frac{p - 1}{p} \int_{-1/2}^{1/2} \left|
 \Delta_{M+1}(z) \sin(2\pi (z + \beta)) \right| \, dz + \frac{2}{p} \cdot \frac{0.607056}{(M + 1)^2} \right)\\
 &\leq \frac{p(p + 1)}{(p - 1)^2}\left( \frac{2}{(M + 1)^2} \cdot \frac{p - 1}{p} \int_{-1/2}^{1/2}
 \left( \frac{\sin \left( \pi (M + 1) z \right)}{\sin \left( \pi z \right)} \right)^2 \left| \sin(2\pi (z + \beta)) \right| \, dz + \frac{2}{p} \cdot \frac{0.607056}{(M + 1)^2} \right).
    \end{align}
We bound the integral separately.
\begin{align}
    &\leq \int_{-1/2}^{1/2} \left( \frac{\sin \left( \pi (M + 1) z \right)}{\sin \left( \pi z \right)} \right)^2  |\sin\left( 2\pi (z + \beta) \right) | \, dz\\
 &= \int_{-1/2}^{1/2} \left( \frac{\sin \left( \pi (M + 1) z \right)}{\sin \left( \pi z \right)} \right)^2  |\sin(2\pi z)\cos\beta + \sin \beta \cos(2\pi z) |  \,dz\\
 &\leq \cos\left( \frac{1}{2\pi M} \right) \int_{-1/2}^{1/2} \left( \frac{\sin \left( \pi (M + 1) z \right)}{\sin \left( \pi z \right)} \right)^2  |\sin(2\pi z) |  dz & \text{(triangle inequality)} \\
 &\quad   + \sin\left( \frac{1}{2\pi M} \right)\int_{-1/2}^{1/2} \left( \frac{\sin \left( \pi (M + 1) z \right)}{\sin \left( \pi z \right)} \right)^2  | \cos(2\pi z) | \, dz.\\
 &\leq \int_{-1/2}^{1/2} \left( \frac{\sin \left( \pi (M + 1) z \right)}{\sin \left( \pi z \right)} \right)^2  |\sin(2\pi z) | \, dz \label{eqn:kim-integral-1-2-1}\\
 &\quad   + \frac{1}{\pi M} \int_{0}^{1/2} \left( \frac{\sin \left( \pi (M + 1) z \right)}{\sin \left( \pi z \right)} \right)^2  | \cos(2\pi z) | \, dz. \label{eqn:kim-integral-1-2-2}
\end{align}
As seen in the proof of Lemma \ref{lem:kim-integral-2.2}, the term \eqref{eqn:kim-integral-1-2-1} is bounded by
$$\frac{2(M + 1)^2}{\pi} \cdot \sin^2\left(\frac{B\pi}{M} \right) + \frac{4}{\pi} \log \left( \frac{2M}{B\pi} \right),$$
by \eqref{eqn:kim-2.2-1} and \eqref{eqn:kim-2.2-2}. We decompose \eqref{eqn:kim-integral-1-2-2} into two integrals as
\begin{align}
    & \quad \frac{1}{\pi M} \int_{0}^{C/M} \left( \frac{\sin \left( \pi (M + 1) z \right)}{\sin \left( \pi z \right)} \right)^2 dz + \frac{1}{\pi M} \int_{C/M}^{1/2} \left( \frac{\sin \left( \pi (M + 1) z \right)}{\sin \left( \pi z \right)} \right)^2 dz\\
    & \qquad \text{where $C > 0$ is a small number to be chosen later}\\
    &\leq \frac{1}{\pi M} \int_{0}^{C/M} (M + 1)^2 dz + \frac{1}{\pi M} \int_{C/M}^{1/2} \left( \frac{1}{2z } \right)^2 dz \qquad \text{since $\sin(\pi z) \geq 2z$ for $0 \leq z \leq \frac{1}{2}$}\\
    &\leq \frac{C}{\pi} \cdot \left( \frac{M + 1}{M} \right)^2 + \frac{M/C - 2}{4 \pi M} \qquad \text{for $M \geq 2C.$}
\end{align}
Therefore, for $M \geq 2 \max(B, C)$
\begin{align}
    &\text{LHS of \eqref{eqn:kim-integral-1-2}}\\
    &\leq \frac{p(p + 1)}{(p - 1)^2} \cdot \frac{1}{(M + 1)^2} \left( \frac{2(p - 1)}{p} \left(\frac{2(M + 1)^2}{\pi} \cdot \sin^2\left(\frac{B\pi}{M} \right) + \frac{4}{\pi} \log \left( \frac{2M}{B\pi} \right)   \right. \right. \\
    & \quad + \left. \left. \frac{C}{\pi} \cdot \left( \frac{M + 1}{M} \right)^2 + \frac{M/C - 2}{4 \pi M}  \right) + \frac{2 \cdot 0.607056}{p} \right)\\
    &\leq \frac{p(p + 1)}{(p - 1)^2} \cdot \frac{1}{(M + 1)^2} \lim_{p \rightarrow \infty} \left( \frac{2(p - 1)}{p} \left(\frac{2(M + 1)^2}{\pi} \cdot \sin^2\left(\frac{B\pi}{M} \right) + \frac{4}{\pi} \log \left( \frac{2M}{B\pi} \right)  \right. \right. \\
    & \quad + \left. \left. \frac{C}{\pi} \cdot \left( \frac{M + 1}{M} \right)^2 + \frac{M/C - 2}{4 \pi M}  \right) + \frac{2 \cdot 0.607056}{p} \right)\\
    & \qquad \text{since the argument of the limit is increasing in $p$}\\
    &= \frac{p(p + 1)}{(p - 1)^2} \cdot \frac{2}{(M + 1)^2}  \left(\frac{2(M + 1)^2}{\pi} \cdot \sin^2\left(\frac{B\pi}{M} \right) + \frac{4}{\pi} \log \left( \frac{2M}{B\pi} \right)  + \frac{C}{\pi} \cdot \left( \frac{M + 1}{M} \right)^2 + \frac{M/C - 2}{4 \pi M}  \right)\\
    &= \frac{p(p + 1)}{(p - 1)^2} \cdot \frac{1}{M^2}   \cdot 3.72936 \qquad \text{when $B = 0.24403$ and $C = 0.375$},
\end{align}
which suffices for the proof.
\end{proof}

We are finally ready to prove Proposition \ref{prop:kim-2-6-analog}, which will follow from the above lemmas. We restate the proposition here.

{
\renewcommand{\thetheorem}
{\ref{prop:kim-2-6-analog}}
\begin{proposition}
    For $M \in \mathbb{N}$ at least $3$ and $I = \left[0, \frac{1}{M} \right] \subseteq \left[0, 2\pi \right],$
    $$F_{I, M}^+(\theta) = \sum_{n = 0}^M b_n R_{n}(\cos \theta)$$
    for certain $b_0, \cdots, b_M$ satisfying
    \begin{align}
        |b_n| &\leq \left( \frac{p + 1}{p - 1} \right)^2 \cdot \frac{44.4751 \log M + 46.5666}{M^2} \qquad \text{for $n \geq 1$}, \\
        |b_0| &\leq \left( \frac{p + 1}{p - 1} \right)^2 \cdot \frac{314.255 \log M + 28.0356}{M^3}.
    \end{align}
\end{proposition}
\addtocounter{theorem}{-1}
}

\begin{proof}
  We largely follow the proof of the analogous result \cite[Theorem 2.6]{Kim24} by Kim, paying extra attention to constants. We start with the case of $n \geq 1$. The existence of the constants $\{b_n\}$ is guaranteed by \cite[Theorem 2.6]{Kim24}. Kim obtains
  \begin{equation}
      b_n = \int_0^\pi \chi_I(\theta) R_n(\cos \theta) \,d\mu_p(\theta) + \int_0^\pi \left( F_{I, M}^+(\theta) - \chi_I(\theta)  \right) R_n(\cos \theta) \,d\mu_p(\theta). \label{eqn:kim-decomp-of-bn}
  \end{equation}
The first integral is bounded by Lemma \ref{lem:kim-integral-simple}. As in \cite[Page 933]{DGMPT20}, we substitute $y = \frac{\theta}{2\pi} $ and $\beta = \frac{1}{2\pi M}$ and rewrite the second integral as
    \begin{align}
        &2\pi \int_{-1/2}^{1/2} \left( B_M(y) - s(y) \right) R_n(\cos 2\pi y ) \,d\mu_p(2\pi y)\\
        + ~ & 2\pi \int_{-1/2}^{1/2}
        \left( B_M(y - \beta ) - s(y - \beta) \right) R_n ( \cos 2 \pi y ) \,d\mu_p(2\pi y)
    \end{align}
    One then uses the definition of $B_M(y)$ to split these two integrals into ($2\pi$ times) the four integrals bounded in Lemmas \ref{lem:kim-integral-2-1}, \ref{lem:kim-integral-2.2}, and \ref{lem:kim-integral-1-2}. The bounds of these lemmas then yield the desired result.

    Next, we consider the case of $n=0$, finding stronger bounds for the above integrals. Using $R_0(\cos \theta) = \frac{p + 1}{p}$, we bound the first integral of \eqref{eqn:kim-decomp-of-bn} as
    $$\left| \int_0^\pi \chi_I(\theta) R_0(\cos \theta) \,d\mu_p \right| = \frac{p + 1}{p} \cdot \frac{p(p + 1)}{(p - 1)^2} \cdot \frac{2}{\pi}\left| \int_0^{1/M}  \sin^2(\theta^2) \,d\mu_p \right| \leq \frac{p + 1}{p} \cdot \frac{p(p + 1)}{(p - 1)^2} \cdot \frac{2}{3\pi M^3}.$$
Since $R_0 \leq \frac{p + 1}{p} \cdot T_0$, \eqref{eqn:obvious-fejer-integral-actually-cubic} gives
\begin{align*}
    \left| \frac{1}{2(M + 1)} \int_{-1/2}^{1/2}
 \Delta_{M+1}(y - \beta) R_0 ( \cos 2 \pi y ) \mu_p(2\pi y) \right| &\leq \frac{p + 1}{p} \cdot \frac{2.42822}{(M + 1)^3} \cdot \frac{p(p + 1)}{(p - 1)^2}\\
 \left| \frac{1}{2(M + 1)} \int_{-1/2}^{1/2}
 \Delta_{M+1}(y) R_0 ( \cos 2 \pi y ) \mu_p(2\pi y) \right| &\leq \frac{p + 1}{p} \cdot \frac{2}{(M + 1)^2} \cdot \frac{p(p + 1)}{(p - 1)^3}.
\end{align*}
We improve on Lemma \ref{lem:kim-integral-2-1} and obtain
\begin{align*}
    &\quad \left| \int_{-1/2}^{1/2}
 \left( V_M(y - \beta ) - s(y - \beta) \right) R_0 ( \cos 2 \pi y ) \mu_p(2\pi y) \right|\\
    &= \frac{p + 1}{p}\left| \int_{-1/2}^{1/2}
 \left( V_M(y - \beta ) - s(y - \beta) \right) \mu_p(2\pi y) \right| \\
 &\leq \frac{4(p - 1)}{p} \cdot \frac{p(p + 1)}{(p - 1)^2} \int_{-1/2}^{1/2}
  \min \left( 1, \frac{.14}{M^3 \left|y - \frac{1}{2\pi M} \right|^3} \right) \sin^2(2\pi y) dy \qquad \text{(Proposition \ref{prop:explicit-constant-vaaler-bound} )}\\
  & = 4 \left( \frac{p + 1}{p - 1} \right)^2 \cdot \frac{1}{M} \int_{-M/2 - 1/2\pi}^{M/2 - 1/2\pi}
  \min \left( 1, \frac{.14}{\left|u \right|^3} \right) \sin^2\left( \frac{1 + 2\pi u}{M}  \right) du \qquad \left( u = y - \frac{1}{2\pi M} \right)\\
  & \leq 4 \left( \frac{p + 1}{p - 1} \right)^2 \cdot \frac{1}{M^3} \int_{-M/2 - 1/2\pi}^{M/2 - 1/2\pi}
  \min \left( 1, \frac{.14}{\left|u \right|^3} \right) \left( 1 + 2\pi u  \right)^2 du\\
  & = 4 \left( \frac{p + 1}{p - 1} \right)^2 \cdot \frac{\log M}{M^3} \cdot S(M),
\end{align*}
where $S(M)$ is equal to
$$\frac{1}{\log M} \left( 3 \sqrt{.14} + \frac{8}{3} \cdot 0.14 \pi^2 + 4 \cdot 0.14 \pi^2 \left[ \log\left( \frac{\pi^2 M^2 - 1}{4\pi^2 (.14)^{2/3}} \right) - \frac{4}{\pi^2 M^2 - 1} - \frac{\pi^2 M^2 + 1}{(\pi^2 M^2 - 1)^2} \right] \right),
$$
which attains a maximum value of $14.70484\dots$ over $[3, \infty)$ at $M = 3.$ Thus, we have that
$$\left| \int_{-1/2}^{1/2}
 \left( V_M(y - \beta ) - s(y - \beta) \right) R_0 ( \cos 2 \pi y ) \mu_p(2\pi y) \right| \leq \left( \frac{p + 1}{p - 1} \right)^2 \cdot 58.8194 \cdot \frac{\log  M}{M^3}.$$
 However, Lemma \ref{lem:kim-integral-2-1} gives that this integral is bounded by $\left( \frac{p + 1}{p - 1}  \right)^2 \cdot  \frac{6.08868}{M^2},$ which is a stronger bound for small $M.$ Therefore, we can improve the constant $58.8194 $ to $50.0151.$

Finally, as in Lemma \ref{lem:kim-integral-2-1}, we have
$$\frac{p + 1}{p}\left| \int_{-1/2}^{1/2}
 \left( V_M(y) - s(y) \right) \mu_p(2\pi y) \right| = 0.$$
 Combining these bounds gives the desired result.
\end{proof}

\section{An explicit bound for \texorpdfstring{$ \left| V_M(x) - s(x) \right|$}{|VM(x)s(x)|}}\label{appendix:Vaaler}

An important ingredient for the proof of Murty-Sinha's original result, as well as our Proposition \ref{prop:k-dependent-extremality} on extremal primes, is the use of Beurling-Selberg polynomials: trigonometric polynomials that approximate the indicator function of an interval. One key inequality in this theory is the error bound (see \cite[$\S$1, (19)]{Mon94})
\begin{equation}\label{eqn:montgomery-result-on-vaaler-polynomial-introduction}
    \left| V_M(x) - s(x) \right| \ll \min\left( 1,  \frac{1}{M^3 \lVert x \rVert_{S^1}}  \right),
\end{equation}
where $s(x)$ denotes the sawtooth function and $V_M(x)$ denotes that $M$-th Vaaler polynomial, both defined at the beginning of Appendix \ref{subsection:integral-bounds}. Note that $\lVert x \rVert_{S^1}$ denotes the distance of $x$ to the nearest integer. This inequality is also useful for other problems in mathematics, including the analysis of exponential sums and the Dirichlet divisor problem (see \cite{huxley2003exponential}). Rather than using \eqref{eqn:montgomery-result-on-vaaler-polynomial-introduction}, which only provides an asymptotic bound, we devote this appendix to proving an explicit constant for this inequality.

\begin{proposition}\label{prop:explicit-constant-vaaler-bound}
    For all $M \in \mathbb{N},$
    $$\left| V_M(x) - s(x) \right| \leq \min\left( \frac{1}{2},  \frac{\pi^4 + 10 \pi^2 + 15}{480\pi} \cdot \frac{1}{M^3 \lVert x \rVert_{S^1}}  \right).$$
\end{proposition}

In addition to being used in the proof of Proposition \ref{prop:k-dependent-extremality}, we expect that this result could be useful in other applications as well. Despite its general applicability, we know of no proofs 
of \eqref{eqn:montgomery-result-on-vaaler-polynomial-introduction} in the literature, so we hope our proof of Proposition \ref{prop:explicit-constant-vaaler-bound} will be useful as a reference. 

\begin{remark}
    We note that the constant $\frac{\pi^4 + 10\pi^2 + 15}{480\pi} \leq 0.14$ in Proposition \ref{prop:explicit-constant-vaaler-bound} is by no means optimal. Numerical calculations for $M \leq 1000$ demonstrate that the true constant should be less than $0.02,$ and perhaps such a result can be extracted from numerical calculations by proving that the optimal constant is decreasing in $M.$ For our purposes, we content ourselves with the constant $\frac{\pi^4 + 10\pi^2 + 15}{480\pi},$ and since numerical calculations prove the result for $M = 1$ and $2$, we only prove the statement for $M \geq 3.$
\end{remark}

The constant bound of $\frac{1}{2}$ in Proposition \ref{prop:explicit-constant-vaaler-bound} can be seen immediately from an elementary bound of Vaaler, which gives
\begin{align}
    \left| V_M(x) - s(x) \right| &\leq \frac{1}{2M + 2} \cdot \Delta_{M + 1}(x) & \left( \text{\cite[Theorem 18]{Vaa85}} \right)\\
    &=  \frac{1}{2(M + 1)^2} \left( \frac{\sin \pi (M + 1) x}{ \sin \pi x}  \right)^2\\
    &\leq \frac{(M + 1)^2}{2(M + 1)^2} & \left(\text{since $\left| \frac{\sin \pi n x}{ \sin \pi x} \right| \leq n$} \right)\\
    &= \frac{1}{2}.
\end{align}
Therefore, it suffices to prove the decaying bound $\frac{\pi^4 + 10\pi^2 + 15 }{480\pi} \cdot \frac{1}{M^3 \lVert x \rVert_{S^1}}$ of Proposition \ref{prop:explicit-constant-vaaler-bound}. To accomplish this, we recall that the Fourier series of $s(x)$ and $V_M(x)$ are given by
\begin{align}
    s(x) &= \sum_{|n| \ge 1} \frac{i}{2\pi n} e^{2\pi i n x} & \left( \text{\cite[\S3.2.2(4)]{SS03}}  \right)\\
    V_M(x) &= \sum_{1 \leq | n| \leq M} \frac{i}{2\pi n} \hat{J}_{M + 1}(n) e^{2\pi i n x}, & \left( \text{\cite[Theorem 6, Theorem 18]{Vaa85}}  \right)
\end{align}
where
    $$\hat{J}_{M+1}(n) := \frac{\pi n}{M + 1} \left(1 -  \frac{|n|}{M + 1}  \right) \cot\left( \frac{\pi n}{M + 1} \right) + \frac{|n|}{M + 1}.$$
Subtracting these Fourier series reveals that
$$s(x) - V_M(x) = \frac{i}{2\pi} \sum_{|n| \geq 0} a_n e^{2\pi i n x},$$
where
\begin{align}
    a_n :=  \begin{cases}
        0 & n = 0\\
        \frac{1}{n} \cdot \left( 1 -  \frac{|n|}{M + 1}  \right) \left( 1 - \frac{\pi n}{M + 1} \cot \left( \frac{\pi n}{M + 1}  \right)  \right) & 1 \leq |n| \leq M\\
        \frac{1}{n} & |n| > M.
    \end{cases}
\end{align}

Summation by parts gives the following lemma.

\begin{lemma}\label{lem:partial-summation}
    For fixed $x \notin \mathbb{Z}$ and any sequence $\{a_n\}$ converging to zero as $n \rightarrow \infty$ and $n \rightarrow -\infty,$
    $$\sum_{|n| \geq 0} a_n e^{2\pi i n x} = \left( \frac{e^{2\pi i x}}{1 - e^{2\pi i x}} \right) \sum_{|n| \geq 0} \Delta a_n \cdot e^{2\pi i n x},$$
    where $\Delta a_n := a_{n + 1} - a_n$ denotes the forward difference operator
\end{lemma}

Therefore,
\begin{align}
     \left| \frac{i}{2\pi} \sum_{|n|\ge 0} a_n e^{2\pi i n x}\right| 
     &= \left|\frac{1}{2\pi}  \left(  \frac{e^{2\pi i x}}{1 - e^{2\pi i x}} \right)^3\sum_{n \in \mathbb{Z}} \Delta^{3} a_n\, e^{2\pi i n x}\right| \qquad \text{(applying Lemma \ref{lem:partial-summation} three times)}\\
     &= \left| \frac{1}{2\pi} \cdot  \frac{i^3 e^{3\pi i x}}{\left( 2 \sin(\pi x)\right)^3} \sum_{n \in \mathbb{Z}} \Delta^{3} a_n\, e^{2\pi i n x}\right|\\
    &\leq\frac{1}{2\pi}  \cdot   \frac{1}{64\|x\|_{S^1}^3} \cdot \sum_{n \in \mathbb{Z}} \left| \Delta^{3} a_n \right| \qquad \left( \text{since $\sin(\pi x) \geq 2x$ for $x \in \left[0, \frac{1}{2} \right]$} \right)\\ 
    &= \frac{1}{64 \pi \|x\|_{S^1}^3} \cdot \sum_{n \geq -1} \left| \Delta^{3} a_n \right|. \label{eqn:sum-after-triple-abel-summation}\\
    & \qquad \text{since $\left| a_n \right| = \left|a_{-n} \right|$ implies that $\left| \Delta^3 a_{-n} \right| = \left| \Delta^3 a_{n - 3} \right|$}
\end{align}
It thus suffices to bound the terms $\Delta^3 a_{n}.$
\begin{lemma}
    Let $M \geq 3.$ Then for $0 \leq n \leq M-3,$ we have
    $$\left| \Delta^3 a_n \right| \leq \frac{1}{(M+1)^4} \cdot \frac{2\pi^4}{15}.$$
\end{lemma}

\begin{proof}
    Define the function
    $$g(y) := \frac{1}{y} \cdot \left( 1 - \left| \frac{y}{M + 1}  \right| \right) \left( 1 - \frac{\pi y}{M + 1} \cot \left( \frac{\pi y}{M + 1}  \right)  \right),$$
    which is thrice differentiable over $(0, M].$ Additionally, add in the missing point discontinuity $g(0) = 0$
    so that for $0 \leq n \leq M,$ we have $a_n = g(n).$ We then have for $0 \leq n \leq M - 3$ that
    \begin{align}
        \left| \Delta^3 a_n  \right| &= \left| g'''(\xi_n) \right| \qquad \qquad \text{for some $\xi_n \in [n, n + 3] \subset [-1, M],$ by the mean value theorem} \\
        &= \frac{1}{(M+1)^4} \left| h'''\left(\frac{\xi_n}{M + 1} z\right) \right| \qquad \text{for $h(z) = \begin{cases}
        \frac{1}{z} \left(1 - |z| \right)  \left( 1 - \pi z \cot (\pi z)  \right) & z \neq 0\\
        0 & z = 0.
    \end{cases}$}\\
        &\leq \frac{1}{(M + 1)^4} \sup_{z \in \left[ -\frac{1}{M}, 1 \right]} \left| h'''(z) \right|\\
        &= \frac{1}{(M+1)^4} \left| \lim_{z \rightarrow 0} h'''(z) \right|\\\
        &= \frac{1}{(M+1)^4} \cdot \frac{2\pi^4}{15},
    \end{align}
    completing the proof.
\end{proof}

Next, note that $\Delta^3 a_n = a_{n + 3} - 3 a_{n + 2} + 3 a_{n + 1} - a_n.$ Therefore, for $n \geq M + 1,$
\begin{align}
    \left| \Delta^3 a_n \right| &= \left| \frac{1}{n + 3} - \frac{3}{n + 2} +\frac{3}{n + 1} - \frac{1}{n} \right|\\
    &=  \left| \frac{-6}{n(n + 1)(n + 2)(n + 3)} \right|\\
    &\leq \frac{6}{n^4}.
\end{align}

Finally, we bound $\left| \Delta^3 a_n \right|$ for the boundary values $n \in \{-1, M - 2, M - 1, M\}$, starting with $M - 2.$
\begin{align}
    & \left| \Delta^3 a_{M - 2} \right| \\
    &= \left| \frac{M^3 - 3M^2 + 2 M - 6}{(M - 2)(M - 1)M(M+1)} \right.\\
    & \qquad\left. + \frac{3\pi}{(M + 1)^2}\left( - \cot \left( \frac{\pi}{M + 1} \right)  + 2 \cot \left(\frac{2\pi}{M + 1} \right) - \cot \left( \frac{3\pi}{M + 1} \right)  \right) \right| \label{eqn:3rd-diff-M-2}\\
    &\leq \frac{1}{M^4} \left| M^4 \cdot \eqref{eqn:3rd-diff-M-2} \right|\\
    &\leq \frac{1}{M^4} \lim_{M \rightarrow \infty} \left| M^4 \cdot \eqref{eqn:3rd-diff-M-2} \right| \qquad \text{(by inspection)}\\
    &= \frac{6}{M^4}
\end{align}
Similarly
\begin{align}
    & \quad \left| \Delta^3 a_{M - 1} \right| \\
    &= \left| \frac{- 2M^3 - 2M^2 + 4 M - 6}{(M - 1)M(M + 1)(M+2)} + \frac{\pi}{(M + 1)^2}\left( 3\cot \left( \frac{\pi}{M + 1} \right) - 2 \cot \left(\frac{2\pi}{M + 1} \right) \right) \right| \label{eqn:3rd-diff-M-1}\\
    &= \frac{1}{M^3}   \left| M^3 \cdot  \eqref{eqn:3rd-diff-M-1} \right|\\
    &\leq \frac{1}{M^3}  \lim_{M \rightarrow \infty} \left| M^3 \cdot  \eqref{eqn:3rd-diff-M-1} \right|\\
    &\leq \frac{\pi^2}{3 M^3},
\end{align}
and
\begin{align}
    \left| \Delta^3 a_M \right| &= \left| \frac{M^3 + 5M^2 + 6M - 6}{M(M + 1)(M + 2)(M + 3)} - \frac{\pi}{(M + 1)^2} \cot \left( \frac{\pi}{M + 1}  \right) \right|  \label{eqn:3rd-diff-M}\\
    &\leq \frac{1}{M^3} \left| M^3 \cdot \eqref{eqn:3rd-diff-M} \right|\\
    &\leq \frac{1}{M^3} \lim_{M \rightarrow \infty }\left| M^3 \cdot \eqref{eqn:3rd-diff-M} \right|\\
    &= \frac{\pi^2}{3M^3}.
\end{align}
Lastly,
\begin{align}
    & \quad \left| \Delta^3 a_{-1} \right|\\
    &= \left| \frac{M - 1}{2(M + 1)} \left( 1 - \frac{2\pi}{M + 1} \cot \left( \frac{2\pi}{M + 1 } \right)\right) - \frac{2M}{M + 1} \left(1 - \frac{\pi}{M + 1} \cot\left( \frac{\pi}{M + 1} \right) \right) \right|  \label{eqn:3rd-diff-M--1}\\
    &\leq \frac{1}{M^3} \left| M^3 \cdot \eqref{eqn:3rd-diff-M--1} \right|\\
    &\leq \frac{1}{M^3} \lim_{M \rightarrow \infty }\left| M^3 \cdot \eqref{eqn:3rd-diff-M--1} \right|\\
    &= \frac{2\pi^2}{3  M^3}.
\end{align}

Putting these terms together, we obtain

\begin{align*}
     \left| \frac{i}{2\pi} \sum_{|n|\ge 0} a_n e^{2\pi i n x}\right| &\leq \frac{1}{64 \pi \|x\|_{S^1}^3} \cdot \sum_{n \geq -1} \left| \Delta^{3} a_n \right| \qquad \left( \text{by \eqref{eqn:sum-after-triple-abel-summation}}  \right)\\
     &= \frac{1}{64 \pi \|x\|_{S^1}^3} \left( \left| \Delta^3 a_{-1} \right| + \sum_{n = 0}^{M - 3} \left| \Delta^{3} a_n \right| + \left( \Delta^{3} a_{M-2} + \Delta^{3} a_{M - 1} + \Delta^{3} a_{M}  \right) + \sum_{n \geq M + 1} \left| \Delta^{3} a_n \right| \right)\\
     &\leq \frac{1}{64 \pi \|x\|_{S^1}^3} \left( \frac{2\pi^2}{3M^3} + \sum_{n = 0}^{M - 3} \frac{2 \pi^4}{15 M^4} + \left( \frac{6}{ M^4} + \frac{\pi^2}{3M^3} + \frac{\pi^2}{3 M^3}  \right) + \sum_{n \geq M + 1} \frac{6}{ n^4} \right)\\
     &\leq \frac{1}{64 \pi \|x\|_{S^1}^3} \left( \sum_{n = 0}^{M - 2} \frac{2 \pi^4}{15 M^4} + \frac{4\pi^2}{3M^3} + \int_M^\infty \frac{6}{ y^4} dy \right)\\
     &\leq \frac{1}{64 \cdot 15 \pi \|x\|_{S^1}^3} \cdot \frac{1}{M^3} \left( 2 \pi^4 + 20 \pi^2 + 30 \right),
\end{align*}
which suffices for the proof.

\section*{Acknowledgments}

The authors would like to acknowledge Peter Sarnak and Will Sawin for helpful comments and suggestions. This work was supported by National Science Foundation grant DMS-2349174. Hui Xue is supported by Simons Foundation grant MPS-TSM-00007911.

\bibliographystyle{plain}
\bibliography{ee-biblio.bib}

\end{document}